\documentclass[final,leqno,onefignum,onetabnum]{siamltexmm}
\usepackage{amsmath}
\usepackage{amssymb}
\usepackage{amsfonts}
\usepackage{graphicx} 
\usepackage{epsfig}
\usepackage{subfigure}
\usepackage{algorithmicx} % algorithmic boxes
\usepackage{algorithm} % algorithmic boxes
\usepackage{algpseudocode} % algorithmic environment 
\usepackage{caption}

\title{Compressive Conjugate Directions: Linear Theory\thanks{This work was supported by Stanford Exploration Project.}}

\author{Musa Maharramov\footnotemark[2] \and Stewart A. Levin\footnotemark[2]}

\begin{document}
\maketitle
\newcommand{\slugmaster}{%
\slugger{siims}{xxxx}{xx}{z}{z--z}}%slugger should be set to juq, siads, sifin, or siims

\renewcommand{\thefootnote}{\fnsymbol{footnote}}

\footnotetext[2]{{Department of Geophysics, Stanford University, 397 Panama Mall, Stanford, CA 94305 (\email{musa@sep.stanford.edu}, \email{stew@sep.stanford.edu}).}}
\renewcommand{\thefootnote}{\arabic{footnote}}



\pagestyle{myheadings}
\thispagestyle{plain}
\markboth{Musa Maharramov and Stewart A. Levin}{Compressive Conjugate Directions: Linear Theory}

\begin{abstract}
We present a powerful and easy-to-implement iterative algorithm for solving large-scale optimization problems that involve $L_1$/total-variation (TV) regularization. The method is based on combining the Alternating Directions Method of Multipliers (ADMM) with a Conjugate Directions technique in a way that allows reusing conjugate search directions constructed by the algorithm across multiple iterations of the ADMM. The new method achieves fast convergence by trading off multiple applications of the modeling operator for the increased memory requirement of storing previous conjugate directions. We illustrate the new method with a series of imaging and inversion applications.
\end{abstract}

\section{Introduction}
We address a class of regularized least-squares fitting problems of the form
\begin{equation}
\begin{aligned}
        & \| \mathbf{B} \mathbf{u} \|_1 \; + \; \frac{\alpha}{2}\|\mathbf{A}\mathbf{u}-\mathbf{d} \|^2_2 \;\rightarrow\; \min,\\
        &\mathbf{u}\; \in \; \mathbb{R}^N,\; \mathbf{d}\;\in\;\mathbb{R}^M,\; \mathbf{A}:\mathbb{R}^N\to\mathbb{R}^M,\;
\mathbf{B}:\mathbb{R}^N\to\mathbb{R}^K,\;K\le N,\;
\end{aligned}
\label{eq:opt0}
\end{equation}
where $\mathbf{d}$ is a known vector (data), $\mathbf{u}$ a vector of unknowns\footnote{sometimes referred to as ``model''}, and $\mathbf{A}, \mathbf{B}$ are linear operators. If $\mathbf{B}$ is the identity map, then problem (\ref{eq:opt0}) is a least-squares fitting with $L_1$ regularization,
\begin{equation}
\begin{aligned}
        \| \mathbf{u} \|_1 \; + \; \frac{\alpha}{2}\|\mathbf{A}\mathbf{u}-\mathbf{d} \|^2_2 \;\rightarrow\; \min.
\end{aligned}
\label{eq:opt1}
\end{equation}
If the unknown vector $\mathbf{u}$ is the discretization of a function, and $\mathbf{B}$ is the first-order finite difference operator
\[ \left(\mathbf{B}\mathbf{u}\right)_i\;=\; u_{i+1}-u_i,\; i=1,2,\ldots,N-1,\]
then problem (\ref{eq:opt0}) turns into a least-squares fitting with a total-variation (TV) regularization
\begin{equation}
\begin{aligned}
\| \nabla \mathbf{u} \|_1 \; + \; \frac{\alpha}{2}\|\mathbf{A}\mathbf{u}-\mathbf{d} \|^2_2 \;\rightarrow\; \min.
\end{aligned}
\label{eq:opt}
\end{equation}
On the one hand, in (\ref{eq:opt1}) we seek a model vector $\mathbf{u}$ such that forward-modeled data $\mathbf{A}\mathbf{u}$ match observed data $\mathbf{d}$ in the least squares sense, while imposing sparsity-promoting $L_1$ regularization. In (\ref{eq:opt}), on the other hand, we impose blockiness-promoting total-variation (TV) regularization. Note that rather than using a regularization parameter as a coefficient of the regularization term, we use a data-fitting weight $\alpha$. TV regularization (also known as the Rudin-Osher-Fatemi, or ROF, model \cite{ROF}) acts as a form of ``model styling'' that helps to preserve sharp contrasts and boundaries in the model even when spectral content of input data has a limited resolution. 

$L_1$-TV regularized least-squares fitting, a key tool in imaging and de-noising applications (see, e.g. \cite{ROF,Chambolle1997,VogelOman,Kim2007}), is beginning to play an increasingly important role in applications where the modeling operator $\mathbf{A}$ in (\ref{eq:opt0}) is computationally challenging to apply. In particular, in seismic imaging problems of exploration geophysics such as full-waveform inversion \cite{Tarantola,Fichtner2011} modeling of seismic wave propagation in a three-dimensional medium from multiple seismic sources is by far the greatest contributor to the computational cost of inversion, and reduction of the number of applications of the operator $\mathbf{A}$ is key to success in practical applications. 

$L_1$-regularized least-squares problems can be reduced to inequality-constrained quadratic programs and solved using interior-point methods based on, e.g., Newton \cite{Boyd} or nonlinear Conjugate Gradients \cite{Kim2007} methods. Alternatively, the resulting bound-constrained quadratic programs can be solved using gradient projection \cite{Figuer} or projected Conjugate Gradients \cite{Qiu2013}. A conceptually different class of techniques for solving $L_1$-regularized least-squares problems is based on homotopy methods \cite{Hastie,efron,Osborne}.

Another class of methods for solving (\ref{eq:opt0}) that merits a special mention applies splitting schemes for the sum of two operators. For example the \emph{iterative shrinking-thresholding algorithm} (ISTA) is based on applying \emph{forward-backward splitting} \cite{Bruck,Passty} to solving the $L_1$-regularized problem (\ref{eq:opt1}) by gradient descent \cite{ISTA3,ISTA1,ISTA2}:
\begin{equation}
\begin{aligned}
        \mathbf{y}_{k+1} \;=\; & \mathbf{u}_{k} \;-\; \gamma \alpha \mathbf{A}^T\left(\mathbf{A}\mathbf{u}_k-\mathbf{d}\right),\\
       \mathbf{u}_{k+1} \; = \;& \mathrm{shrink}\left\{\mathbf{y}_{k+1},\gamma\right\}, 
\end{aligned}
\label{eq:ISTA}
\end{equation}
where $\gamma>0$ is a sufficiently small step parameter, and the \emph{soft thresholding} or \emph{shrinkage} operator is the Moreau resolvent (see, e.g., \cite{Monotone}) of $\partial \gamma\|\mathbf{u}\|_1$,
\begin{equation}
        \begin{aligned}
                \mathrm{shrink}\left\{ \mathbf{y}, \gamma \right\} \;=\; & \left(1 + \partial \gamma \|\mathbf{y}\|_1\right)^{-1} \;=\; \mathrm{argmin}\,_{\mathbf{x}}\left\{ \gamma \|\mathbf{x}\|_1 + \frac{1}{2}\|\mathbf{y}-\mathbf{x}\|_2^2\right\}\;=\; \\
        & \frac{\mathbf{y}}{|\mathbf{y}|} \max\left(|\mathbf{y}|-\gamma,0\right),
\end{aligned}
\label{eq:shrink}
\end{equation}
and $\partial \;=\;\partial_{\mathbf{u}}$ denotes the subgradient \cite{Rockafellar1,Monotone}, and the absolute value of a vector is computed component-wise. The typically slow convergence of the first-order method (\ref{eq:ISTA}) can be accelerated by an over-relaxation step \cite{Nesterov}, resulting in the \emph{Fast ISTA} algorithm (FISTA) \cite{Beck2}:
\begin{equation}
\begin{aligned}
        \mathbf{y}_{k+1} \;=\; & \mathbf{u}_{k} \;-\; \gamma \alpha \mathbf{A}^T\left(\mathbf{A}\mathbf{u}_k-\mathbf{d}\right),\\
       \mathbf{z}_{k+1} \; = \;& \mathrm{shrink}\left\{\mathbf{y}_{k+1},\gamma\right\},\\
            \zeta_{k+1} \;=\; & \left(1+\sqrt{1+4\zeta_k^2}\right)/2,\\
         \mathbf{u}_{k+1}\;=\; & \mathbf{y}_{k+1} + \frac{\zeta_k-1}{\zeta_{k+1}}\left(\mathbf{y}_{k+1}-\mathbf{y}_k\right),
\end{aligned}
\label{eq:FISTA}
\end{equation}
where $\zeta_1=1$ and $\gamma$ is sufficiently small. 

It is important to note that algorithm (\ref{eq:FISTA}) is applied to the $L_1$-regularized problem (\ref{eq:opt1}), not the TV-regularized problem (\ref{eq:opt}). An accelerated algorithm for solving a TV-regularized \emph{denoising problem}\footnote{with $\mathbf{A}=\mathbf{I}$ in (\ref{eq:opt})} was proposed in \cite{Beck1} and applied the Nesterov relaxation \cite{Nesterov} to solving the dual of the TV-regularized denoising problem \cite{Chambolle2004}. However, using a similar approach to solving (\ref{eq:opt}) with a non-trivial operator $\mathbf{A}$ results in accelerated schemes that still require inversion of $\mathbf{A}$ \cite{Beck1,Gold2014} and thus lack the primary appeal of the accelerated gradient descent methods---i.e., a single application of $\mathbf{A}$ and its transpose per iteration\footnote{In \cite{Beck1} inversion of $\mathbf{A}$ is replaced by a single gradient descent, however, over-relaxation is applied to the dual variable.}. 

The advantage of (\ref{eq:FISTA}) compared with simple gradient descent is that Nesterov's over-relaxation step requires storing two previous solution vectors and provides improved search direction for minimization. Note, however, that the step length $\gamma$ is inversely proportional to the Lipschitz constant of $\alpha\mathbf{A}^T\left(\mathbf{A}\mathbf{u}-\mathbf{d}\right)$ \cite{Beck2} and may be small in practice.

A very general approach to solving problems (\ref{eq:opt0}) involving either $L_1$ or TV regularization is provided by primal-dual methods. For example, in TV-regularized least-squares problem (\ref{eq:opt}), by substituting
\begin{equation}
        \mathbf{z}=\mathbf{B}\mathbf{u}
        \label{eq:zBu}
\end{equation}
and adding (\ref{eq:zBu}) as a constraint, we obtain an equivalent equality-constrained optimization problem
\begin{equation}
\begin{aligned}
& \| \mathbf{z} \|_1 \; + \; \frac{\alpha}{2}\|\mathbf{A} \mathbf{u}-\mathbf{d} \|^2_2 \;  \;\rightarrow\; \min,\\
&\mathbf{z}\;=\; \mathbf{B} \mathbf{u}.
\end{aligned}
\label{eq:opt2}
\end{equation}
The optimal solution of (\ref{eq:opt2}) corresponds to the saddle-point of its Lagrangian
\begin{equation}
        L_0\left(\mathbf{u},\mathbf{z},\boldsymbol\mu\right)\;=\;
        \|\mathbf{z}\|_1 + \frac{\alpha}{2} \| \mathbf{A}\mathbf{u} - \mathbf{d}\|_2^2 + \boldsymbol \mu^T\left(\mathbf{z}-\mathbf{B}\mathbf{u}\right),
        \label{eq:L0}
\end{equation}
that can be found by the \emph{Uzawa method} \cite{Uzawa}. The Uzawa method finds the saddle point by alternating a minimization with respect to the primal variables $\mathbf{u},\mathbf{z}$ and ascent over the dual variable $\boldsymbol \mu$ for the objective function equal to the standard Lagrangian (\ref{eq:L0}), $L=L_0$,
\begin{equation}
\begin{aligned}
        \left(\mathbf{u}_{k+1},\mathbf{z}_{k+1}\right)\;=\; & \mathrm{argmin}\, L\left(\mathbf{u},\mathbf{z},\boldsymbol \mu_{k}\right),\\
                                 \boldsymbol \mu_{k+1}\;=\; & \boldsymbol \mu_k \; + \; \lambda \left[\mathbf{z}_{k+1}-\mathbf{B}\mathbf{u}_{k+1}\right]
\end{aligned}
        \label{eq:dual}
\end{equation}
for some positive step size $\lambda$. Approach (\ref{eq:dual}), when applied to the Augmented Lagrangian \cite{Rockafellar56}, $L=L_+$,
\begin{equation}
        L_+\left(\mathbf{u},\mathbf{z},\boldsymbol\mu\right)\;=\;
        \|\mathbf{z}\|_1 + \frac{\alpha}{2} \| \mathbf{A}\mathbf{u} - \mathbf{d}\|_2^2 + \boldsymbol \mu^T\left(\mathbf{z}-\mathbf{B}\mathbf{u}\right) + \frac{\lambda}{2}\|\mathbf{z}-\mathbf{B}\mathbf{u}\|^2_2,
        \label{eq:LAug}
\end{equation}
results in the \emph{method of multipliers} \cite{Hestenes69}. For problems (\ref{eq:opt0}) all these methods still require joint minimization with respect to $\mathbf{u}$ and $\mathbf{z}$ of some objective function that includes both $\|\mathbf{z}\|_1$ and a smooth function of $\mathbf{u}$. Splitting the joint minimization into separate steps of minimization with respect  $\mathbf{u}$, followed by minimization with respect to $\mathbf{z}$, results in the \emph{Alternating-Directions Method of Multipliers} (ADMM) \cite{Glowinski1975,Gabay1976,GlowinskiTallec1989,Eckstein1992,ADMM}. To establish a connection to the splitting techniques applied to the sum of two operators, we note that the ADMM is equivalent to applying the Douglas-Rachford splitting \cite{DR} to the problem
\begin{equation}
\partial \left[\| \mathbf{B} \mathbf{u} \|_1 \; + \; \frac{\alpha}{2}\|\mathbf{A}\mathbf{u}-\mathbf{d} \|^2_2 \right]  \ni \boldsymbol 0,
\label{eq:0subgrad}
\end{equation}
where $\partial$ is the subgradient, and problem (\ref{eq:0subgrad}) is equivalent to (\ref{eq:opt0}). The ADMM is a particular case of a primal-dual iterative solution framework with splitting \cite{Zhang2010}, where the minimization in (\ref{eq:dual}) is split into two steps, 
\begin{equation}
        \begin{aligned}
                \mathbf{u}_{k+1}\;=\; &  \mathrm{argmin}\, L\left(\mathbf{u},\mathbf{z}_{k},\boldsymbol \mu_{k}\right),\\
                \mathbf{z}_{k+1}\;=\; &  \mathrm{argmin}\, L\left(\mathbf{u}_{k+1},\mathbf{z},\boldsymbol \mu_k\right),\\
           \boldsymbol \mu_{k+1}\;=\; & \boldsymbol \mu_k \; + \; \lambda \left[\mathbf{z}_{k+1}-\mathbf{B}\mathbf{u}_{k+1}\right]
        \end{aligned}
        \label{eq:framework}
\end{equation}
For the ADMM, we substitute $L=L_+$ in (\ref{eq:framework}) but other choices of a modified Lagrange function $L$ are possible that may produce convergent primal-dual algorithms \cite{Zhang2010}. Making the substitution $L=L_+$ from (\ref{eq:LAug}) into (\ref{eq:framework}), and introducing a scaled vector of multipliers,
\begin{equation}
\mathbf{b}_k\;=\; \boldsymbol \mu_k /\lambda,\; k=0,1,2,\ldots
\label{eq:b}
\end{equation}
we obtain
\begin{equation}
\begin{aligned}
        \mathbf{u}_{k+1} & \;=\; \mathrm{argmin}\, \frac{\alpha}{2}\|\mathbf{A} \mathbf{u}-\mathbf{d} \|^2_2 \; +\; \frac{\lambda}{2} \| \mathbf{z}_k\;-\; \mathbf{B} \mathbf{u}+ \mathbf{b}_k \|_2^2,\\
        \mathbf{z}_{k+1} & \;=\; \mathrm{argmin}\,   \| \mathbf{z} \|_1 \; + \; \frac{\lambda}{2} \| \mathbf{z}\;-\; \mathbf{B} \mathbf{u}_{k+1}+ \mathbf{b}_k \|_2^2,\\
  \mathbf{b}_{k+1} & \;=\; \mathbf{b}_k +\mathbf{z}_{k+1} -  \mathbf{B}\mathbf{u}_{k+1},\; k=0,1,2,\ldots 
\end{aligned}
\label{eq:opt4}
\end{equation}
where we used the fact that adding a constant term $\lambda /2 \|\mathbf{b}_k\|^2_2$ to the objective function does not alter the solution.  In the iterative process (\ref{eq:opt4}), we apply splitting, minimizing 
\begin{equation}
\begin{aligned}
\| \mathbf{z} \|_1 \; + \; \frac{\alpha}{2}\|\mathbf{A} \mathbf{u}-\mathbf{d} \|^2_2 \; + \frac{\lambda}{2} \| \mathbf{z}\;-\; \mathbf{B} \mathbf{u}+ \mathbf{b}_k \|_2^2 
\end{aligned}
\label{eq:obj}
\end{equation}
alternately with respect to $\mathbf{u}$ and $\mathbf{z}$. Further we note that the minimization of (\ref{eq:obj}) with respect to $\mathbf{z}$ (in a splitting step with $\mathbf{u}$ fixed) is given trivially by the shrinkage operator (\ref{eq:shrink}),
\begin{equation}
\mathbf{z}_{k+1}\;=\;\mathrm{shrink}\left\{\mathbf{B} \mathbf{u}-\mathbf{b}_k,{1}/{\lambda}\right\}.
\label{eq:shrink1}
\end{equation}
Combining (\ref{eq:opt4}) and (\ref{eq:shrink1}) we obtain Algorithm \ref{alg:aug}.
\begin{algorithm}[htl]
        \caption{Alternating Direction Method of Multipliers (ADMM) for (\ref{eq:opt0})}
 \label{alg:aug}
\begin{algorithmic}[1]
\State $\mathbf{u}_{0} \;\gets\; \boldsymbol 0^N,\; \mathbf{z}_0^K\;\gets\; \boldsymbol 0$
\State {$\mathbf{b}_0\;\gets\; \boldsymbol 0^K$}
\For {$k\gets0,1,2,3,\ldots$}
\State {$\mathbf{u}_{k+1}\;\gets\; \mathrm{argmin}\,  \left\{\frac{\lambda}{2} \| \mathbf{z}_k - \mathbf{B} \mathbf{u} + \mathbf{b}_k \|_2^2 + \frac{\alpha}{2}\| \mathbf{A} \mathbf{u} - \mathbf{d} \|_2^2\right\}$ }
    \State {$\mathbf{z}_{k+1}\;\gets\; \mathrm{shrink}\left\{\mathbf{B} \mathbf{u}_{k+1}-\mathbf{b}_k,{1}/{\lambda}\right\}$}
    \State {$\mathbf{b}_{k+1}\;\gets\; \mathbf{b}_k + \mathbf{z}_{k+1} - \mathbf{B} \mathbf{u}_{k+1}$}
    \State {Exit loop if ${\|\mathbf{u}_{k+1}-\mathbf{u}_k\|_2}/{\| \mathbf{u}_k \|_2} \; \le \; \text{target accuracy}$}
\EndFor
\end{algorithmic}
\end{algorithm}

Minimization on the first line of (\ref{eq:opt4}) at each step of the ADMM requires inversion of the operator $\mathbf{A}$. In the first-order gradient-descent methods like (\ref{eq:FISTA}) a similar requirement is obviated by replacing the minimization with respect to variable $\mathbf{u}$ by gradient descent. However, for ill-conditioned problems the gradient may be a poor approximation to the optimal search direction. One interpretation of Nesterov's over-relaxation step in (\ref{eq:FISTA}) is that it provides a better search direction by perturbing the current solution update with a fraction of the previous update on the last line of (\ref{eq:FISTA}). The intermediate least-squares problem in (\ref{eq:opt4}) can be solved approximately using, for example, a few iterations of conjugate gradients. However, repeating multiple iterations of Conjugate Gradients at each step of the ADMM may be unnecessary. Indeed, as we demonstrate in the following sections, conjugate directions constructed at earlier steps of the ADMM can be reused because the matrix of the system of normal equations associated with the minimization on the first line of (\ref{eq:opt4}) does not change between ADMM steps\footnote{Only the right-hand sides of the system are updated as a result of thresholding.}. Therefore, we can trade the computational cost of applying the operator $\mathbf{A}$ and its transpose against the cost of storing a few solution and data-size vectors. As this approach is applied to the most general problem (\ref{eq:opt0}) with a non-trivial operator $\mathbf{B}$, in addition to the potential speed-up, this method has the advantage of working equally well for $L_1$ and $TV$-regularized problems.

We stress that our new approach does not improve the theoretical convergence properties of the classic ADMM method under the assumption of exact minimization in step 4 of Algorithm~\ref{alg:aug}. The asymptotic convergence rate is still $O(1/k)$ as with exact minimization \cite{He2012}. The new approach provides a numerically feasible way of implementing the ADMM for problems where a computationally expensive operator $\mathbf{A}$ precludes accurate minimization in step 4. However, the rate of convergence in the general method of multipliers (\ref{eq:dual}) is sensitive to the choice of parameter $\lambda$, and an improved convergence rate for some values of $\lambda$ can be accompanied with more ill-conditioned minimization problems at each step of (\ref{eq:opt4}) \cite{GlowinskiTallec1989}. By employing increasingly more accurate conjugate-directions solution of the minimization problem at each iteration of (\ref{eq:opt4}) the new method offsets the deteriorating condition of the intermediate least-squares problems, and achieves a faster practical convergence at early iterations.

Practical utility of the ADMM in applications that involve sparsity-promoting (\ref{eq:opt1}) or edge-preserving (\ref{eq:opt}) inversion is often determined by how quickly we can resolve sparse or blocky model components. These features can often be \emph{qualitatively} resolved within relatively few initial iterations of the ADMM (see discussion in the Appendix of \cite{GoldOsher09}). In our Section 4, fast recovery of such \emph{local} features will be one of the key indicators for judging the efficiency of the proposed method.

In the next section we describe two new algorithms, \emph{Steered} and \emph{Compressive Conjugate Gradients} based on the principle of reusing conjugate directions for multiple right-hand sides. In Section 3 we prove convergence and demonstrate that the new algorithm coincides with the exact ADMM in a finite number of iterations. Section 4 contains a practical implementation of the Compressive Conjugate Gradients method. We test the method on a series of problems from imaging and mechanics, and compare its performance against FISTA and ADMM with gradient descent and restarted conjugate gradients.

\section{Steered and Compressive Conjugate Directions}

Step 4 of Algorithm \ref{alg:aug} is itself a least-squares optimization problem of the form
\begin{equation}
        \|\mathbf{F}\mathbf{u}\;-\;\mathbf{v}_k\|_2^2 \; \rightarrow\; \min,
        \label{eq:Fuvk}
\end{equation}
where
%\begin{equation}
%        \mathbf{F}\;=\;  \alpha\mathbf{A}^T\mathbf{A} \;+\;\lambda\mathbf{B}^T\mathbf{B}\;+\; 
%        \delta\mathbf{I},
%        \label{eq:Fop}
%\end{equation}
\begin{equation}
        \mathbf{F}\;=\; 
\begin{bmatrix}
        \sqrt{\alpha}\mathbf{A}\\
        \sqrt{\lambda}\mathbf{B}
\end{bmatrix}
        \label{eq:Fop}
\end{equation}
and
%\begin{equation}
%        \mathbf{v}_k\;=\;\alpha\mathbf{A}^T\mathbf{d} \;+\;\lambda\mathbf{B}^T\left(\mathbf{b}_k-\mathbf{z}_k\right) \;+\;
%        \delta\left(\mathbf{y}_k+\mathbf{c}_k\right),\;\;k=0,1,2,\ldots
%        \label{eq:vk}
%\end{equation}
\begin{equation}
        \mathbf{v}_k\;=\;
        \begin{bmatrix}
                \sqrt{\alpha}\mathbf{d}\\
                \sqrt{\lambda}\left(\mathbf{z}_k+\mathbf{b}_k\right)
        \end{bmatrix}
        \label{eq:vk}
\end{equation}

Solving optimization problem (\ref{eq:Fuvk}) is mathematically equivalent to solving the following system of normal equations \cite{NLA},
\begin{equation}
        \mathbf{F}^T\mathbf{F}\mathbf{u}\;=\; \mathbf{F}^T\mathbf{v}_k,
        \label{eq:norm}
\end{equation}
as operator (\ref{eq:Fop}) has maximum rank. Solving (\ref{eq:norm}) has the disadvantage of squaring the condition number of operator (\ref{eq:Fop}) \cite{NLA}. When the operator $\mathbf{A}$ is available in a matrix form, and a factorization of operator $\mathbf{F}$ is numerically feasible, solving the normal equations (\ref{eq:norm}) should be avoided and a technique based on a matrix factorization should be applied directly to solving (\ref{eq:Fuvk}) \cite{BJORCK,SAAD}. However, when matrix $\mathbf{A}$ is not known explicitly or its size exceeds practical limitations of direct methods, as is the case in applications of greatest interest for us, an iterative algorithm, such as the Conjugate Gradients for Normal Equations (CGNE) \cite{BJORCK,SAAD}, can be used to solve (\ref{eq:norm}).  Solving (\ref{eq:Fuvk}) exactly may be unnecessary and we can expect that for large-scale problems only a few steps of an iterative method need be carried out. However, every iteration typically requires the application of operator $\mathbf{A}$ and its adjoint, and in large-scale optimization problems we are interested in minimizing the number of applications of these operations. For large-scale optimization problems we need an alternative to re-starting an iterative solver for each intermediate problem (\ref{eq:Fuvk}). We propose to minimize restarting iterations\footnote{avoiding restarting altogether in the theoretical limit of infinite computer storage} by devising a conjugate-directions technique for solving (\ref{eq:Fuvk}) with a non-stationary right-hand side. At each iteration of the proposed algorithm we find a search direction that is conjugate to previous directions with respect to the operator $\mathbf{F}^T\mathbf{F}$. In the existing conjugate direction techniques, iteratively constructed conjugate directions span the Krylov subspaces \cite{NLA},
\begin{equation}
        \mathcal{K}_k \;=\;        \mathrm{span}\,\left\{ \mathbf{F}^T\mathbf{v}_0, \left(\mathbf{F}^T\mathbf{F}\right) \mathbf{F}^T \mathbf{v}_0,\ldots,\left(\mathbf{F}^T\mathbf{F}\right)^k \mathbf{F}^T\mathbf{v}_0  \right\},\; k=0,1,\ldots\; .
        \label{eq:Krylov}
\end{equation}
However, in our approach we construct a sequence of vectors (search directions) that are conjugate with respect to operator $\mathbf{F}^T\mathbf{F}$ at the $k$th step but may not span the Krylov subspace $\mathcal{K}_k$. This complicates convergence analysis of our technique, but allows ``steering'' search directions by iteration-dependent right-hand sides. Since the right-hand side in (\ref{eq:Fuvk}) is the result of the shrinkage (\ref{eq:shrink1}) at previous iterations that steer or compress the solution, we call our approach ``steered'' or ``compressive'' conjugate directions. 

For the least-squares problem (\ref{eq:Fuvk}), we construct two sets of vectors for $k=0,1,2,\ldots$
\begin{equation}
\begin{aligned}
        &\left\{ \mathbf{p}_0,\mathbf{p}_1,\mathbf{p}_2,\ldots,\mathbf{p}_k \right\},\; \left\{ \mathbf{q}_0,\mathbf{q}_1,\mathbf{q}_2,\ldots,\mathbf{q}_k \right\},\\
        & \mathbf{q}_i\;=\; \mathbf{F} \mathbf{p}_i,\; i=0,1,2,\ldots,k,
\end{aligned}
\label{eq:pq}
\end{equation}
such that
\begin{equation}
        \mathbf{q}_i^T \mathbf{q}_j \;=\; \mathbf{p}_i^T \mathbf{F}^T \mathbf{F} \mathbf{p}_j \;=\; 0\; \mathrm{ if }\; i\;\neq\; j.
        \label{eq:conjdir}
\end{equation}
Equations (\ref{eq:pq}) and (\ref{eq:conjdir}) mean that the vectors $\mathbf{p}_i$ form \emph{conjugate directions} \cite{NLA,SAAD}. At each iteration we find an approximation $\mathbf{u}_k$ to the solution of (\ref{eq:Fuvk}) as a linear combination of vectors $\mathbf{p}_i,i=0,1,\ldots,k$, for which the residual
\begin{equation}
        \mathbf{r}_{k+1} \; = \; \mathbf{v}_{k+1} \;- \ \mathbf{F}\mathbf{u}_{k+1},
        \label{eq:res}
\end{equation}
is orthogonal to vectors $\mathbf{q}_i$,
\begin{equation}
        \mathbf{q}_i^T \mathbf{r}_{k+1} \;=\; \mathbf{q}_i^T \left(\mathbf{v}_{k+1} \;- \ \mathbf{F}\mathbf{u}_{k+1}\right)\;=\; 0,\; i=0,1,\ldots,k.
        \label{eq:resorth}
\end{equation}
Vector $\mathbf{p}_k$ is constructed as a linear combination of \emph{all} previous vectors $\mathbf{p}_i,i=0,1,\ldots,k$ and $\mathbf{F}^T\mathbf{r}_k$ so that the conjugacy condition in (\ref{eq:pq}) is satisfied. The resulting algorithm for \emph{arbitrary} $\mathbf{v}_k$ depending on $k$ is given by Algorithm \ref{alg:scd}.

\begin{algorithm}[htl]
        \caption{Steered Conjugate Directions for solving (\ref{eq:Fuvk})}
 \label{alg:scd}
\begin{algorithmic}[1]
\State $\mathbf{u}_{0} \;\gets\; \boldsymbol 0^N$
\State {$\mathbf{p}_0\;\gets\; \mathbf{F}^T \mathbf{v}_0,\;\mathbf{q}_0\;\gets\;\mathbf{F}\mathbf{p}_0,\; \delta_0\;\gets\; {\mathbf{q}_0^T \mathbf{q}_0}$}
\For {$k=0,1,2,3,\ldots$}
    \For {$i=0,1,\ldots,k$}
        \State {$\tau_i\;\gets\; {\mathbf{q}_i^T \mathbf{v}_{k}}/\delta_i$}
    \EndFor
    \State {$\mathbf{u}_{k+1}\;\gets\;\sum_{i=0}^{k} \tau_i \mathbf{p}_i$}
    \State{$\mathbf{r}_{k+1}\;\gets\;\mathbf{v}_{k+1}\;-\; \sum_{i=0}^{k}\tau_i  \mathbf{q}_i$}
        
    \State {$\mathbf{w}_{k+1}\;\gets\;  \mathbf{F}^T \mathbf{r}_{k+1}$}
    \State {$\mathbf{s}_{k+1}\;\gets\;  \mathbf{F} \mathbf{w}_{k+1}$}
    \For {$i=0,1,\ldots,k$}
        \State {$\beta_i\;\gets\;  - {\mathbf{q}_i^T \mathbf{s}_{k+1}}/\delta_i$}
    \EndFor
    \State {$\mathbf{p}_{k+1}\;\gets\;\sum_{i=0}^{k} \beta_i \mathbf{p}_i \; + \; \mathbf{w}_{k+1}$}
    \State {$\mathbf{q}_{k+1}\;\gets\;\sum_{i=0}^{k} \beta_i \mathbf{q}_i \; + \; \mathbf{s}_{k+1}$}
    \State {$\delta_{k+1} \;\gets\;\mathbf{q}_{k+1}^T \mathbf{q}_{k+1}$}
    \If {$\delta_{k+1} = 0$}
    \Comment {Use condition ``$\delta_{k+1}  <  \text{tolerance}$'' in practice}
         \State {$\delta_{k+1}\;\gets\;1,\; \mathbf{p}_{k+1}\;\gets\;\mathbf{0}^N,\; \mathbf{q}_{k+1}\;\gets\;\mathbf{0}^{M+K}$}
    \EndIf
    \State {Exit loop if ${\|\mathbf{u}_{k+1}-\mathbf{u}_k\|_2}/{\| \mathbf{u}_k \|_2} \; \le \; \text{target accuracy}$}
\EndFor
\end{algorithmic}
\end{algorithm}

Note that the above algorithm is not specific to a particular sequence of right-hand-side vectors $\mathbf{v}_k$ and its applicability goes beyond solving the constrained optimization problems (\ref{eq:opt2}). The algorithm requires storing $2k+2$ vectors (\ref{eq:pq}), as well as one vector each for the current solution iterate $\mathbf{u}_k$, variable right-hand side $\mathbf{v}_k$, intermediate vectors $\mathbf{w}_k$ and $\mathbf{s}_k$. The requirement of storing a growing number of vectors makes the algorithm resemble the GMRES method \cite{SAAD} for solving linear systems with non-self-adjoint operators. However, in our case, this is a consequence of having a variable right-hand side, requiring re-computation of solution iterates as linear combinations of all of the previous search directions (\ref{eq:pq}). This requirement can be relaxed in applications where vector $\mathbf{v}_k$ is updated, for example, by the modified Lagrangian technique for solving a constrained optimization problem, and converges to a limit. In Section 4 we describe practical applications of the algorithm achieving fast convergence while storing only a subset of vectors (\ref{eq:pq}). The algorithm requires one application of $\mathbf{F}$ and its transpose at each iteration and $2k+3$ dot-products of large vectors. 

Combining Algorithms \ref{alg:aug} and \ref{alg:scd} we obtain the \emph{Compressive Conjugate Directions} Algorithm~\ref{alg:ccd}.

\begin{algorithm}[htl]
        \caption{Compressive Conjugate Directions for (\ref{eq:opt0})}
 \label{alg:ccd}
\begin{algorithmic}[1]
\State $\mathbf{u}_{0} \;\gets\; \boldsymbol 0^N,\; \mathbf{z}_0\;\gets\; \boldsymbol 0^K;\; \mathbf{b}_0\;\gets\; \boldsymbol 0^K,\; \mathbf{v}_0\;\gets\;\begin{bmatrix}
                \sqrt{\alpha}\mathbf{d}\\
                \sqrt{\lambda}\left(\mathbf{z}_0+\mathbf{b}_0\right)
\end{bmatrix}$
\State {$\mathbf{p}_0\;\gets\; \mathbf{F}^T \mathbf{v}_0,\;\mathbf{q}_0\;\gets\;\mathbf{F}\mathbf{p}_0,\;\delta_0\;\gets\; \mathbf{q}^T_0\mathbf{q}_0$}
\For {$k=0,1,2,3,\ldots$}
    \For {$i=0,1,\ldots,k$}
        \State {$\tau_i\;\gets\; {\mathbf{q}_i^T \mathbf{v}_{k}}/\delta_i$}
    \EndFor
    \State {$\mathbf{u}_{k+1}\;\gets\;\sum_{i=0}^{k} \tau_i \mathbf{p}_i$}
    \State {$\mathbf{z}_{k+1}\;\gets\; \mathrm{shrink}\left\{\mathbf{B} \mathbf{u}_{k+1}-\mathbf{b}_k,{1}/{\lambda}\right\}$}
    \State {$\mathbf{b}_{k+1}\;\gets\; \mathbf{b}_k +\mathbf{z}_{k+1} - \mathbf{B} \mathbf{u}_{k+1}$}
    \State {$\mathbf{v}_{k+1}\;\gets\;\begin{bmatrix}
                \sqrt{\alpha}\mathbf{d}\\
                \sqrt{\lambda}\left(\mathbf{z}_{k+1}+\mathbf{b}_{k+1}\right)
\end{bmatrix}$}
    \State{$\mathbf{r}_{k+1}\;\gets\;\mathbf{v}_{k+1}\;-\; \sum_{i=0}^{k}\tau_i  \mathbf{q}_i$}
        
    \State {$\mathbf{w}_{k+1}\;\gets\;  \mathbf{F}^T \mathbf{r}_{k+1}$}
    \State {$\mathbf{s}_{k+1}\;\gets\;  \mathbf{F} \mathbf{w}_{k+1}$}
    \For {$i=0,1,\ldots,k$}
        \State {$\beta_i\;\gets\;  - {\mathbf{q}_i^T \mathbf{s}_{k+1}}/\delta_i$}
    \EndFor
    \State {$\mathbf{p}_{k+1}\;\gets\;\sum_{i=0}^{k} \beta_i \mathbf{p}_i \; + \; \mathbf{w}_{k+1}$}
    \State {$\mathbf{q}_{k+1}\;\gets\;\sum_{i=0}^{k} \beta_i \mathbf{q}_i \; + \; \mathbf{s}_{k+1}$}
    \State {$\delta_{k+1} \;\gets\;\mathbf{q}_{k+1}^T \mathbf{q}_{k+1}$}
    \If {$\delta_{k+1} = 0$}
    \Comment {Use condition ``$\delta_{k+1}  <  \text{tolerance}$'' in practice}
         \State {$\delta_{k+1}\;\gets\;1,\; \mathbf{p}_{k+1}\;\gets\;\mathbf{0}^N,\; \mathbf{q}_{k+1}\;\gets\;\mathbf{0}^{M+K}$}
    \EndIf
    \State {Exit loop if ${\|\mathbf{u}_{k+1}-\mathbf{u}_k\|_2}/{\| \mathbf{u}_k \|_2} \; \le \; \text{target accuracy}$}
\EndFor
\end{algorithmic}
\end{algorithm}

\section{Convergence Analysis}
Convergence properties of the ADMM were studied in many publications and are well known. However, here we provide a self-contained proof of convergence for Algorithm~\ref{alg:aug} that mostly follows the presentation of \cite{ADMM}. Later, we use this result to study the convergence of Algorithm~\ref{alg:ccd}.  
\begin{theorem}\label{thm:aug}
        Assume that $M\ge N$, operators $\mathbf{A}$, $\mathbf{B}$ are maximum rank, and 
        \begin{equation}
        \begin{aligned}
                \mathbf{u}\; & = \; \mathbf{u}^\ast,\\
                \mathbf{z}\; & = \; \mathbf{z}^\ast\;=\; \mathbf{B}\mathbf{u}^\ast,
        \end{aligned}
        \label{eq:solution0}
        \end{equation}
is the unique solution of problem (\ref{eq:opt2}). Assume that a vector $\mathbf{b}^\ast$ is defined as
        \begin{equation}
        \begin{aligned}
                \mathbf{b}^\ast \; & = \; \boldsymbol \mu^\ast/\lambda,
        \end{aligned}
        \label{eq:multiplier0}
        \end{equation}
where $\boldsymbol \mu^\ast$ is the vector of Lagrange multipliers for the equality constraint in (\ref{eq:opt2}).  Algorithm~\ref{alg:aug} then converges to this solution if $\lambda > 0$, that is,
        \begin{equation}
        \begin{aligned}
                \mathbf{u}_k\; \rightarrow \; \mathbf{u}^\ast,\; \mathbf{z}_k\;  \rightarrow \; \mathbf{z}^\ast,\; \mathbf{b}_k\;  \rightarrow \; \mathbf{b}^\ast,\; k\;\rightarrow\; \infty.
        \end{aligned}
        \label{eq:conv0}
        \end{equation}
\end{theorem}
\begin{proof} Problem (\ref{eq:opt2}) has a convex objective function and equality constraints, hence (\ref{eq:solution0},\ref{eq:multiplier0}) is a saddle point of its Lagrangian (\ref{eq:L0}) \cite{Boyd}. Substituting $\mathbf{z}_{k+1},\mathbf{u}_{k+1}$ from Algorithm~\ref{alg:aug}, we have
        \begin{equation}
                \begin{aligned}
                        & L_0\left (\mathbf{z}^\ast,\mathbf{u}^\ast,\boldsymbol \mu^\ast\right) \; \le \; 
                        L_0\left(\mathbf{z}_{k+1},\mathbf{u}_{k+1},\boldsymbol \mu^\ast\right)\; \Longleftrightarrow \\
                        p^\ast \;=\; & \|\mathbf{B}\mathbf{u}^\ast\|_1 \;+\; \frac{\alpha}{2}\| \mathbf{A}\mathbf{u}^\ast-\mathbf{d}\|_2^2 \; = \;  \|\mathbf{z}^\ast\|_1 \;+\; \frac{\alpha}{2}\| \mathbf{A}\mathbf{u}^\ast-\mathbf{d}\|_2^2 \; \le \; \\
                                       & \|\mathbf{z}_{k+1}\|_1 \; + \; \frac{\alpha}{2}\| \mathbf{A}\mathbf{u}_{k+1}-\mathbf{d}\|_2^2 \; + \; \boldsymbol\mu^{\ast T}\left(\mathbf{z}_{k+1}-\mathbf{B}\mathbf{u}_{k+1}\right)\; = \; \\
                                     & p_{k+1} \; + \;  \boldsymbol\mu^{\ast T}\left(\mathbf{z}_{k+1}-\mathbf{B}\mathbf{u}_{k+1}\right)\; = \; p_{k+1}  \; + \;  \lambda \mathbf{b}^{\ast T}\left(\mathbf{z}_{k+1}-\mathbf{B}\mathbf{u}_{k+1}\right),
                \end{aligned}
                \label{eq:lowerest}
        \end{equation}
        where $p^\ast$ is the optimal value of the objective function and $p_{k+1}$ is its approximation at iteration $k$ of the algorithm. Inequality (\ref{eq:lowerest}) provides a lower bound for the objective function estimate $p_{k+1}$. Step 4 of the algorithm is equivalent to
\begin{equation}
        \alpha \mathbf{A}^T\mathbf{A} \mathbf{u}_{k+1}\; + \;\lambda \mathbf{B}^T\mathbf{B} \mathbf{u}_{k+1} \;=\;  \alpha \mathbf{A}^T\mathbf{d} + \lambda \mathbf{B}^T\left(\mathbf{z}_k + \mathbf{b}_k\right).
        \label{eq:ukp1}
\end{equation}
Substituting the expression for $\mathbf{b}_k$ from steps 6 into (\ref{eq:ukp1}), we obtain
\begin{equation}
        \alpha \mathbf{A}^T\mathbf{A} \mathbf{u}_{k+1}\; \;=\;  \alpha \mathbf{A}^T\mathbf{d} + \lambda \mathbf{B}^T\left(\mathbf{z}_k - \mathbf{z}_{k+1} + \mathbf{b}_{k+1}\right).
        \label{eq:ukp2}
\end{equation}
Equality (\ref{eq:ukp2}) is equivalent to 
\begin{equation}
        \mathbf{u}_{k+1}\;=\; \mathrm{argmin}\, \frac{\alpha}{2} \|\mathbf{A}\mathbf{u}-\mathbf{d}\|_2^2 \; - \; \lambda \left(\mathbf{z}_k-\mathbf{z}_{k+1} + \mathbf{b}_{k+1}\right)^T\mathbf{B}\mathbf{u}.
        \label{eq:ukp3}
\end{equation}
Substituting $\mathbf{u}_{k+1}$ and $\mathbf{u}^\ast$ into the right-hand side of (\ref{eq:ukp3}), we obtain
\begin{equation}
        \begin{aligned}
        \frac{\alpha}{2} \|\mathbf{A}\mathbf{u}_{k+1}-\mathbf{d}\|_2^2 \; \le \; &  \frac{\alpha}{2} \|\mathbf{A}\mathbf{u}^\ast-\mathbf{d}\|_2^2  \;  + \\
                                                                                 & \lambda \left(\mathbf{z}_k-\mathbf{z}_{k+1} + \mathbf{b}_{k+1}\right)^T\mathbf{B}\left(\mathbf{u}_{k+1}-\mathbf{u}^\ast\right).
        \end{aligned}
        \label{eq:ukp4}
\end{equation}
Step 5 is equivalent to
\begin{equation}
        \begin{aligned}
                \boldsymbol 0 \; \in \; \partial_{\mathbf{z}} \|\mathbf{z}\|_1 \; + \; \lambda \left( \mathbf{z}_{k+1} - \mathbf{B}\mathbf{u}_{k+1} + \mathbf{b}_k\right)\; = \;  \partial_{\mathbf{z}} \|\mathbf{z}\|_1 \; + \; \lambda {b}_{k+1},\\
                \mathbf{z}_{k+1} \; = \mathrm{argmin}\, \left \{\| \mathbf{z} \|_1 \; + \; \lambda \mathbf{b}_{k+1}^T\mathbf{z} \right\},
        \end{aligned}
        \label{eq:ukp5}
\end{equation}
where we used the expression for $\mathbf{b}_k$ from step 6. Substituting $\mathbf{z}=\mathbf{z}_{k+1}$ and $\mathbf{z}=\mathbf{z}^\ast$ into the right-hand side of the second line of (\ref{eq:ukp5}), we obtain
\begin{equation}
 \|\mathbf{z}_{k+1}\|_1 \; \le \; \| \mathbf{z}^\ast\|_1 \; + \; \lambda \mathbf{b}_{k+1}^T\left(\mathbf{z}^\ast - \mathbf{z}_{k+1}\right).
 \label{eq:ukp6}
\end{equation}
Adding (\ref{eq:ukp4}) and (\ref{eq:ukp6}), we get
\begin{equation}
        \begin{aligned}
                p_{k+1} \; \le \; & p^\ast  \;  + \; \lambda \mathbf{b}_{k+1}^T\left(\mathbf{z}^\ast - \mathbf{z}_{k+1}\right) + \\                                                                                 & \lambda \left(\mathbf{z}_k-\mathbf{z}_{k+1} + \mathbf{b}_{k+1}\right)^T\mathbf{B}\left(\mathbf{u}_{k+1}-\mathbf{u}^\ast\right),
        \end{aligned}
        \label{eq:upperest}
\end{equation}
an upper bound for $p_{k+1}$. Adding (\ref{eq:lowerest}) and (\ref{eq:upperest}), we get
\begin{equation}
\begin{aligned}
        0 \; \le \;  &  \lambda \mathbf{b}^{\ast T}\left(\mathbf{z}_{k+1}-\mathbf{B}\mathbf{u}_{k+1}\right) \; + \; \lambda \mathbf{b}_{k+1}^T\left(\mathbf{z}^\ast - \mathbf{z}_{k+1}\right) + \\                                                                                 & \lambda \left(\mathbf{z}_k-\mathbf{z}_{k+1} + \mathbf{b}_{k+1}\right)^T\mathbf{B}\left(\mathbf{u}_{k+1}-\mathbf{u}^\ast\right),
\end{aligned}
\end{equation}
or after rearranging,
\begin{equation}
\begin{aligned}
                   0 \; \le \;   & \lambda \left( \mathbf{b}^\ast - \mathbf{b}_{k+1}\right)^T\left(\mathbf{z}_{k+1} - \mathbf{B}\mathbf{u}_{k+1}\right) \; - \; \lambda \left( \mathbf{z}_{k} - \mathbf{z}_{k+1}\right)^T\left(\mathbf{z}_{k+1} - \mathbf{B}\mathbf{u}_{k+1}\right) \; + \\
                    & \lambda \left( \mathbf{z}_{k} - \mathbf{z}_{k+1}\right)^T\left(\mathbf{z}_{k+1} - \mathbf{z}^\ast\right).
\end{aligned}
\label{eq:upper1}
\end{equation}
We will now use (\ref{eq:upper1}) to derive an upper estimate for 
\[
        \|\mathbf{b}_k-\mathbf{b}^\ast\|_2^2 \; + \; 
        \|\mathbf{z}_{k}-\mathbf{z}^\ast\|_2^2.
\]
Using step 6 of Algorithm~\ref{alg:aug} for the first term in (\ref{eq:upper1}) and introducing $\boldsymbol\rho_{k+1} = \mathbf{z}_{k+1}-\mathbf{B}\mathbf{u}_{k+1}$, we get
\begin{equation}
\begin{aligned}
        & \lambda \left( \mathbf{b}^\ast - \mathbf{b}_{k+1}\right)^T\boldsymbol\rho_{k+1} \; = \\
        & \lambda \left( \mathbf{b}^\ast - \mathbf{b}_{k} -  \boldsymbol\rho_{k+1}\right)^T\boldsymbol\rho_{k+1} \; = \; \lambda \left( \mathbf{b}^\ast - \mathbf{b}_{k}\right)^T \boldsymbol\rho_{k+1}  - \lambda\| \boldsymbol\rho_{k+1}\|_2^2 \; = \\ &
        \lambda \left( \mathbf{b}^\ast - \mathbf{b}_{k}\right)^T \left(\mathbf{b}_{k+1} - \mathbf{b}_{k}\right) - \frac{\lambda}{2} \| \boldsymbol\rho_{k+1}\|_2^2 - \frac{\lambda}{2} \| \boldsymbol\rho_{k+1}\|_2^2 \; = \\
        & \lambda \left( \mathbf{b}^\ast - \mathbf{b}_{k}\right)^T \left(\mathbf{b}_{k+1} - \mathbf{b}_{k}\right)- \frac{\lambda}{2} \| \boldsymbol\rho_{k+1}\|_2^2  - \frac{\lambda}{2} \left(\mathbf{b}_{k+1} - \mathbf{b}_{k}\right)^T\left(\mathbf{b}_{k+1} - \mathbf{b}_{k}\right)\; = \\
        &  - \lambda \left( \mathbf{b}_{k} - \mathbf{b}^\ast \right)^T \left[\left(\mathbf{b}_{k+1} - \mathbf{b}^\ast\right)- \left(\mathbf{b}_{k} - \mathbf{b}^\ast\right)\right] - \frac{\lambda}{2} \| \boldsymbol\rho_{k+1}\|_2^2 - \\
        & \frac{\lambda}{2} \left[\left(\mathbf{b}_{k+1} - \mathbf{b}^\ast\right)- \left(\mathbf{b}_{k} - \mathbf{b}^\ast\right)\right]^T\left[\left(\mathbf{b}_{k+1} - \mathbf{b}^\ast\right)- \left(\mathbf{b}_{k} - \mathbf{b}^\ast\right)\right]\; = \\
        & \frac{\lambda}{2}\|\mathbf{b}_k - \mathbf{b}^\ast\|_2^2 - \frac{\lambda}{2}\|\mathbf{b}_{k+1} - \mathbf{b}^\ast\|_2^2 - \frac{\lambda}{2} \| \boldsymbol\rho_{k+1}\|_2^2. 
\end{aligned}
\label{eq:first}
\end{equation}
Substituting (\ref{eq:first}) into (\ref{eq:upper1}), we obtain 
\begin{equation}
\begin{aligned}
                   0 \; \le \;   &  \frac{\lambda}{2}\|\mathbf{b}_k - \mathbf{b}^\ast\|_2^2 - \frac{\lambda}{2}\|\mathbf{b}_{k+1} - \mathbf{b}^\ast\|_2^2 - \frac{\lambda}{2} \| \boldsymbol\rho_{k+1}\|_2^2\; - \; \lambda \left( \mathbf{z}_{k} - \mathbf{z}_{k+1}\right)^T\boldsymbol\rho_{k+1} \; + \\
                    & \lambda \left( \mathbf{z}_{k} - \mathbf{z}_{k+1}\right)^T\left(\mathbf{z}_{k+1} - \mathbf{z}^\ast\right)\; = \\
                    &  \frac{\lambda}{2}\|\mathbf{b}_k - \mathbf{b}^\ast\|_2^2 - \frac{\lambda}{2}\|\mathbf{b}_{k+1} - \mathbf{b}^\ast\|_2^2 - \frac{\lambda}{2} \| \boldsymbol\rho_{k+1}\|_2^2\; - \;\lambda \left( \mathbf{z}_{k} - \mathbf{z}_{k+1}\right)^T\boldsymbol\rho_{k+1} \; + \\
                    & \lambda \left( \mathbf{z}_{k} - \mathbf{z}_{k+1}\right)^T\left[\left(\mathbf{z}_{k+1} - \mathbf{z}_k\right) + \left(\mathbf{z}_k - \mathbf{z}^\ast\right)\right]\; =  \\
            &  \frac{\lambda}{2}\|\mathbf{b}_k - \mathbf{b}^\ast\|_2^2 - \frac{\lambda}{2}\|\mathbf{b}_{k+1} - \mathbf{b}^\ast\|_2^2 - \frac{\lambda}{2} \| \boldsymbol\rho_{k+1}\|_2^2\; - \;\lambda \left( \mathbf{z}_{k} - \mathbf{z}_{k+1}\right)^T\boldsymbol\rho_{k+1} \; - \\
            & \lambda \left( \mathbf{z}_{k} - \mathbf{z}_{k+1}\right)^T \left(\mathbf{z}_{k} - \mathbf{z}_{k+1}\right) +  \lambda \left( \mathbf{z}_{k} - \mathbf{z}_{k+1}\right)^T \left(\mathbf{z}_{k} - \mathbf{z}^\ast\right)\; =  \\
          &  \frac{\lambda}{2}\|\mathbf{b}_k - \mathbf{b}^\ast\|_2^2 - \frac{\lambda}{2}\|\mathbf{b}_{k+1} - \mathbf{b}^\ast\|_2^2 - \frac{\lambda}{2}\left(\mathbf{z}_k-\mathbf{z}_{k+1} + \boldsymbol\rho_{k+1}\right)^T\left(\mathbf{z}_k-\mathbf{z}_{k+1} + \boldsymbol\rho_{k+1}\right)\; - \\
          &  \frac{\lambda}{2}\|\mathbf{z}_k-\mathbf{z}_{k+1} \|_2^2  + \lambda \left( \mathbf{z}_{k} - \mathbf{z}_{k+1}\right)^T \left(\mathbf{z}_{k} - \mathbf{z}^\ast\right)\; =\\
& \frac{\lambda}{2}\|\mathbf{b}_k - \mathbf{b}^\ast\|_2^2 - \frac{\lambda}{2}\|\mathbf{b}_{k+1} - \mathbf{b}^\ast\|_2^2 - \frac{\lambda}{2} \|\mathbf{z}_k-\mathbf{z}_{k+1} + \boldsymbol\rho_{k+1}\|_2^2 -\frac{\lambda}{2}\|\mathbf{z}_k-\mathbf{z}_{k+1} \|_2^2\;  + \\
         & \lambda \left[ \left(\mathbf{z}_{k} - \mathbf{z}^\ast\right) - \left(\mathbf{z}_{k+1}-\mathbf{z}^\ast\right)\right]^T \left(\mathbf{z}_{k} - \mathbf{z}^\ast\right)\; = \\
& \frac{\lambda}{2}\|\mathbf{b}_k - \mathbf{b}^\ast\|_2^2 - \frac{\lambda}{2}\|\mathbf{b}_{k+1} - \mathbf{b}^\ast\|_2^2 - \frac{\lambda}{2} \|\mathbf{z}_k-\mathbf{z}_{k+1} + \boldsymbol\rho_{k+1}\|_2^2\; - \\
       & \frac{\lambda}{2}\|\left(\mathbf{z}_k-\mathbf{z}^\ast\right) - \left(\mathbf{z}_{k+1}-\mathbf{z}^\ast\right) \|_2^2  + \lambda \left[ \left(\mathbf{z}_{k} - \mathbf{z}^\ast\right) - \left(\mathbf{z}_{k+1}-\mathbf{z}^\ast\right)\right]^T \left(\mathbf{z}_{k} - \mathbf{z}^\ast\right)\; =\\
 &  \frac{\lambda}{2}\|\mathbf{b}_k - \mathbf{b}^\ast\|_2^2 - \frac{\lambda}{2}\|\mathbf{b}_{k+1} - \mathbf{b}^\ast\|_2^2 - \frac{\lambda}{2} \|\mathbf{z}_k-\mathbf{z}_{k+1} + \boldsymbol\rho_{k+1}\|_2^2\; - \\
& \frac{\lambda}{2}\|\mathbf{z}_{k+1}-\mathbf{z}^\ast\|_2^2 +  \frac{\lambda}{2}\|\mathbf{z}_{k}-\mathbf{z}^\ast\|_2^2, 
\end{aligned}
\label{eq:upper2}
\end{equation}
yielding
\begin{equation}
\begin{aligned}
        & \frac{\lambda}{2} \|\mathbf{z}_k-\mathbf{z}_{k+1} + \boldsymbol\rho_{k+1}\|_2^2\; \le \\
& \frac{\lambda}{2}\left(\|\mathbf{z}_{k}-\mathbf{z}^\ast\|_2^2 + \|\mathbf{b}_k - \mathbf{b}^\ast\|_2^2 \right)-\frac{\lambda}{2}\left(\|\mathbf{z}_{k+1}-\mathbf{z}^\ast\|_2^2 + \|\mathbf{b}_{k+1} - \mathbf{b}^\ast\|_2^2\right).
\end{aligned}
\label{eq:upper3}
\end{equation}
Expanding the left-hand side of (\ref{eq:upper3}), we obtain
\begin{equation}
        \begin{aligned}
                & \frac{\lambda}{2} \left( \|\mathbf{z}_k-\mathbf{z}_{k+1} \|_2^2 + 2\left(\mathbf{z}_k-\mathbf{z}_{k+1}\right)^T\boldsymbol\rho_{k+1}+ \|\boldsymbol\rho_{k+1}\|_2^2\right)\; \le \\         
& \frac{\lambda}{2}\left(\|\mathbf{z}_{k}-\mathbf{z}^\ast\|_2^2 + \|\mathbf{b}_k - \mathbf{b}^\ast\|_2^2 \right)-\frac{\lambda}{2}\left(\|\mathbf{z}_{k+1}-\mathbf{z}^\ast\|_2^2 + \|\mathbf{b}_{k+1} - \mathbf{b}^\ast\|_2^2\right).
        \end{aligned}
        \label{eq:upper4}
\end{equation}
Let us prove that the middle term in the left-hand side of (\ref{eq:upper4}) is non-negatve,

\[ 0\; \le \; \left(\mathbf{z}_{k} - \mathbf{z}_{k+1}\right)^T\boldsymbol \rho_{k+1} \; = \; \left(\mathbf{z}_{k} - \mathbf{z}_{k+1}\right)^T \left(\mathbf{b}_{k+1} - \mathbf{b}_{k}\right)\]

where we used step 6 of Algorithm \ref{alg:aug}. Indeed, since $\mathbf{z}_{k+1}$ minimizes (\ref{eq:obj}) with $\mathbf{u}=\mathbf{u}_{k+1}$, using the convexity of $L_1$ norm, we have for $\mathbf{z}=\mathbf{z}_{k+1}$,  
\begin{equation}
        \begin{aligned}
                & \partial_z \frac{\lambda}{2}\|\mathbf{z}-\mathbf{B}\mathbf{u}_{k+1}+\mathbf{b}_{k}\|_2^2=\lambda\left(\mathbf{z}-\mathbf{B}\mathbf{u}_{k+1}+\mathbf{b}_{k}\right) \in -\partial \|\mathbf{z}\|_1\; \Rightarrow \\
                & \|\mathbf{z}_{k+1}\|_1 - \|\mathbf{z}_{k}\|_1 \; \le \;  \left(\mathbf{z}_{k}-\mathbf{z}_{k+1}\right)^T\left(\mathbf{z}_{k+1}-\mathbf{B}\mathbf{u}_{k+1}+\mathbf{b}_{k}\right)= \left(\mathbf{z}_{k}-\mathbf{z}_{k+1}\right)^T\mathbf{b}_{k+1}.
        \end{aligned}
        \label{eq:upper5}
\end{equation}
Similarly, since $\mathbf{z}_{k}$ minimizes (\ref{eq:obj}) for $\mathbf{u}=\mathbf{u}_k$ and $\mathbf{b}=\mathbf{b}_{k-1}$, for $\mathbf{z}=\mathbf{z}_{k}$ we have
\begin{equation}
        \begin{aligned}
                & \partial_z \frac{\lambda}{2}\|\mathbf{z}-\mathbf{B}\mathbf{u}_{k}+\mathbf{b}_{k-1}\|_2^2=\lambda\left(\mathbf{z}-\mathbf{B}\mathbf{u}_{k}+\mathbf{b}_{k-1}\right) \in -\partial \|\mathbf{z}\|_1\; \Rightarrow \\
                & \|\mathbf{z}_{k}\|_1 - \|\mathbf{z}_{k+1}\|_1\; \le \;  \left(\mathbf{z}_{k+1}-\mathbf{z}_{k}\right)^T\left(\mathbf{z}_{k}-\mathbf{B}\mathbf{u}_{k}+\mathbf{b}_{k-1}\right)= \left(\mathbf{z}_{k+1}-\mathbf{z}_{k}\right)^T\mathbf{b}_{k}.
        \end{aligned}
        \label{eq:upper6}
\end{equation}
In both (\ref{eq:upper5}) and (\ref{eq:upper6}) we used step 6 of Algorithm \ref{alg:aug} and the fact that for any convex function $f(\mathbf{x})$
\[ f(\mathbf{x}_0) + \boldsymbol \xi^T \left(\mathbf{x}-\mathbf{x}_0\right) \; \le \; f(\mathbf{x})\; \Leftrightarrow \; f(\mathbf{x}_0)  - f(\mathbf{x}) \; \le \; -\boldsymbol \xi^T \left(\mathbf{x}-\mathbf{x}_0\right),\;   \text{ if }  \boldsymbol \xi\in \partial f(\mathbf{x}_0), \]
where $\partial$ is subgradient \cite{Rockafellar1}. Summing (\ref{eq:upper5}) and (\ref{eq:upper6}) we get
\begin{equation}
                0\; \le \; \left(\mathbf{z}_{k}-\mathbf{z}_{k+1}\right)^T\left(\mathbf{b}_{k+1}-\mathbf{b}_{k}\right).
        \label{eq:upper7}
\end{equation}
From (\ref{eq:upper7}) and (\ref{eq:upper4}), we have
\begin{equation}
        \begin{aligned}
                &  \|\mathbf{z}_k-\mathbf{z}_{k+1} \|_2^2 +  \|\boldsymbol\rho_{k+1}\|_2^2\; \le \\         
& \left(\|\mathbf{z}_{k}-\mathbf{z}^\ast\|_2^2 + \|\mathbf{b}_k - \mathbf{b}^\ast\|_2^2 \right)-\left(\|\mathbf{z}_{k+1}-\mathbf{z}^\ast\|_2^2 + \|\mathbf{b}_{k+1} - \mathbf{b}^\ast\|_2^2\right),
        \end{aligned}
        \label{eq:upper8}
\end{equation}
or
\begin{equation}
        \begin{aligned}
                & \|\mathbf{z}_{k+1}-\mathbf{z}^\ast\|_2^2 + \|\mathbf{b}_{k+1} - \mathbf{b}^\ast\|_2^2 \; \le \\
                & \|\mathbf{z}_{k}-\mathbf{z}^\ast\|_2^2 + \|\mathbf{b}_k - \mathbf{b}^\ast\|_2^2 - \|\mathbf{z}_{k+1}-\mathbf{z}_{k}\|_2^2 -  \|\boldsymbol\rho_{k+1}\|_2^2.
        \end{aligned}
        \label{eq:upper9}
\end{equation}
From (\ref{eq:upper9}) we can see that the sequence $\|\mathbf{z}_{k}-\mathbf{z}^\ast\|_2^2 + \|\mathbf{b}_k - \mathbf{b}^\ast\|_2^2$ and consequently $\mathbf{z}_k$ and $\mathbf{b}_k$ are bounded. Summing (\ref{eq:upper8}) for $k=0,1,\ldots,\infty$, we obtain convergence of the series
\begin{equation}
        \begin{aligned}
                \sum_{k=0}^\infty { \left\{  \|\mathbf{z}_k-\mathbf{z}_{k+1} \|_2^2 +  \|\boldsymbol\rho_{k+1}\|_2^2 \right \}} \; \le \;\|\mathbf{z}_{0}-\mathbf{z}^\ast\|_2^2 + \|\mathbf{b}_0 - \mathbf{b}^\ast\|_2^2.
        \end{aligned}
        \label{eq:upper10}
\end{equation}
From (\ref{eq:upper10}) follows
\begin{equation}
        \mathbf{z}_k-\mathbf{z}_{k+1} \to 0,\;\;\mathbf{z}_{k}-\mathbf{B}\mathbf{u}_{k}\to0,\;k\to \infty.
        \label{eq:lim2}
\end{equation}
Now using (\ref{eq:upperest}) we obtain
\begin{equation}
        \begin{aligned}
                p_{k+1} \; -\;  p^\ast  \; & \le \;\lambda \mathbf{b}_{k+1}^T\left(\mathbf{z}^\ast - \mathbf{z}_{k+1}\right) + \lambda \left(\mathbf{z}_k-\mathbf{z}_{k+1} + \mathbf{b}_{k+1}\right)^T\mathbf{B}\left(\mathbf{u}_{k+1}-\mathbf{u}^\ast\right)\;=\\
   & \lambda \mathbf{b}^T_{k+1}\left(\mathbf{z}_k-\mathbf{z}_{k+1}\right) + \lambda \mathbf{b}^T_{k+1}\left(\mathbf{z}^\ast-\mathbf{z}_k\right)\; + \\
   & \lambda \left(\mathbf{z}_k - \mathbf{z}_{k+1}\right)^T\mathbf{B}\left(\mathbf{u}_{k+1}-\mathbf{u}^\ast\right) + \lambda \mathbf{b}^T_{k+1}\mathbf{B}\left(\mathbf{u}_{k+1}-\mathbf{u}^\ast\right)\;=\\
   & \lambda \mathbf{b}^T_{k+1}\left(\mathbf{z}_k-\mathbf{z}_{k+1}\right)\; + \;\lambda \left(\mathbf{z}_k - \mathbf{z}_{k+1}\right)^T\mathbf{B}\left(\mathbf{u}_{k+1}- \mathbf{u}^\ast\right) \; + \\
   & \lambda \mathbf{b}^T_{k+1}\left(\mathbf{z}^\ast-\mathbf{z}_k\right)\; + \; \lambda \mathbf{b}^T_{k+1}\mathbf{B}\left(\mathbf{u}_{k+1}-\mathbf{u}^\ast\right)\;=\\
   & \lambda \mathbf{b}^T_{k+1}\left(\mathbf{z}_k-\mathbf{z}_{k+1}\right)\; + \;\lambda \left(\mathbf{z}_k - \mathbf{z}_{k+1}\right)^T\mathbf{B}\left(\mathbf{u}_{k+1}- \mathbf{u}^\ast\right) \; + \\
   & \lambda \mathbf{b}^T_{k+1}\left(\mathbf{B}\mathbf{u}_{k+1}-\mathbf{z}_{k+1} + \mathbf{z}_{k+1} - \mathbf{z}_{k} + \mathbf{z}^\ast - \mathbf{B}\mathbf{u}^\ast\right)\;\to\; 0,\; k \to\infty,
        \end{aligned}
        \label{eq:upperest2}
\end{equation}
where the right-hand side of (\ref{eq:upperest2}) converges to zero because of (\ref{eq:lim2}), boundedness of $\mathbf{z}_k$ and $\mathbf{b}_k$ and $\mathbf{z}^\ast=\mathbf{B}\mathbf{u}^\ast$.
Likewise, from (\ref{eq:lowerest}) we have
\begin{equation}
        \begin{aligned}
                p^\ast - p_{k+1} \;\le\; \lambda \mathbf{b}^{\ast T}\left(\mathbf{z}_{k+1}-\mathbf{B}\mathbf{u}_{k+1}\right)\; \to\; 0,\; k\to \infty.
                \end{aligned}
                \label{eq:lowerest2}
        \end{equation}
        Combining (\ref{eq:upperest2}) and (\ref{eq:lowerest2}) we obtain $p_k\to p^\ast$---i.e., value of the objective function estimate at iteration $k$ converges to the true minimum as $k\to\infty$. From the bounded sequence $\mathbf{u}_k\in \mathbb{R}^N$ we can extract a convergent subsequence
\begin{equation}
        \mathbf{u}_{k_i}\;\to\; \mathbf{u}^{\ast \ast}.
        \label{eq:u2}
\end{equation}
Because our objective function is continuous, $\mathbf{u}^{\ast\ast}$ is a solution of (\ref{eq:opt0}) and (\ref{eq:opt2}). However, if $\mathbf{A}$ is maximum rank the objective function of (\ref{eq:opt0}) is strictly convex, hence $\mathbf{u}^\ast=\mathbf{u}^{\ast\ast}$. The sequence $\mathbf{u}_k$ must converge to $\mathbf{u}^\ast$ because otherwise we would be able to extract a subsequence convergent to a different limit and repeat the above analysis.

And finally, to prove that $\mathbf{b}_k\to\mathbf{b}^\ast$, we see that from the Karush-Kuhn-Tucker (KKT) conditions \cite{Boyd} for (\ref{eq:opt2}) we have
\begin{equation}
        \alpha\mathbf{A}\mathbf{A}^T \mathbf{u}^\ast \;=\; \mathbf{A}^T\mathbf{d}\; + \;\lambda \mathbf{B}^T \mathbf{b}^\ast.
        \label{eq:kkt1}
\end{equation}
Passing (\ref{eq:ukp2}) to limit as $k\to\infty$, using (\ref{eq:lim2}) and replacing $\mathbf{b}_{k+1}$ with a convergent subsequence as necessary, we get
\begin{equation}
        \alpha\mathbf{A}\mathbf{A}^T \mathbf{u}^\ast \;=\; \mathbf{A}^T\mathbf{d}\; + \;\lambda \mathbf{B}^T \lim \mathbf{b}_k.
        \label{eq:kkt2}
\end{equation}
Since $\mathbf{B}$ is maximum rank, $\mathrm{rank}\,\mathbf{B}=K\le N$, (\ref{eq:kkt2}) means that $\lim \mathbf{b}_k = \mathbf{b}^\ast$.
\end{proof}

Note that our our proof does not depend on the selection of starting values for $\mathbf{u}_0$, $\mathbf{z}_{0}$ and $\mathbf{b}_{0}$, and this fact will be used later on in proving the convergence of Algorithm~\ref{alg:ccd}. Before we study convergence properties of Algorithm~\ref{alg:ccd}, we prove one auxiliary result.
\begin{theorem}\label{thm:spaces}
        Algorithm \ref{alg:ccd} constructs a sequence of subspaces of $\mathbb{R}^N$ spanning expanding sets of conjugate directions,
 \begin{equation}
        \begin{aligned}
         & S_k\;=\;\mathrm{span}\,\left\{\mathbf{p}_0,\mathbf{p}_1,\ldots,\mathbf{p}_k\right\},\; k=0,1,2,\ldots\\
         & S_0 \subseteq S_1 \subseteq S_2 \subseteq \ldots \subseteq S_k \subseteq \ldots
                \label{eq:spaces}
        \end{aligned}
 \end{equation}
such that
\begin{equation}
        \lim_{k\to\infty} S_k\;=\; S\; \subseteq \; \mathbb{R}^N.
        \label{eq:S}
\end{equation}
Under the assumptions of Theorem~\ref{thm:aug}, solution of the constrained optimization problem
\begin{equation}
\begin{aligned}
& \| \mathbf{z} \|_1 \; + \; \frac{\alpha}{2}\|\mathbf{A} \mathbf{u}-\mathbf{d} \|^2_2 \;  \;\rightarrow\; \min,\\
&\mathbf{z}\;=\; \mathbf{B} \mathbf{u},\\
& \mathbf{u}\; \in\; S.
\end{aligned}
\label{eq:opt5}
\end{equation}
matches the solution of (\ref{eq:opt2}).
\end{theorem}
\begin{proof} If $S = \mathbb{R}^N$ statement of the theorem is trivial, so we assume that $\mathrm{dim}\,S<N$. Since our problem is finite-dimensional, the limit (\ref{eq:S}) is achieved at a finite iteration,
\begin{equation}
\exists k_1\;        \forall k \ge k_1:\; S_k \; \equiv \; S.
\label{eq:k1}
        \end{equation}
        steps 4-7 of Algorithm~\ref{alg:ccd} are equivalent to projecting the solution of the system of normal equations (\ref{eq:norm}) onto the space $S_k$. If $p_{k+1}=0$ in steps 20-22, then the right-hand side of (\ref{eq:norm}) for any $k\ge k_1$ can be represented as a linear combination of vectors from $S_{k_1} \equiv S$. Steps 8 and 9 of Algorithm~\ref{alg:ccd} are equivalent to steps 5 and 6 of Algorithm~\ref{alg:aug}. Step 10 prepares the right-hand side of (\ref{eq:norm}) for the minimization in step 4 of Algorithm~\ref{alg:aug} for iteration $k+1$. However, since the right-hand side of (\ref{eq:norm}) is a linear combination of vectors $\mathbf{p}_0,\mathbf{p}_1,\ldots,\mathbf{p}_k$ that span $S_k\equiv S$, steps 4-7 of Algorithm~\ref{alg:ccd} are equivalent to the exact solution of the unconstrained minimization problem in step 4 of Algorithm~\ref{alg:aug}. Hence, starting from iteration $k_1$ the two algorithms become equivalent. From Theorem~\ref{thm:aug} and 
        \[\forall k\ge k_1:\; \mathbf{u}_{k+1}\;\in\; S\]
follows that the solution of (\ref{eq:opt5}) coincides with that of (\ref{eq:opt2}).
\end{proof} 

Convergence of Algorithm~\ref{alg:ccd} now becomes a trivial corollary of theorems~\ref{thm:aug} and \ref{thm:spaces}.
\begin{theorem}\label{thm:ccd} Under the assumptions of Theorem~\ref{thm:aug}, Algorithm~\ref{alg:ccd} converges to the unique solution (\ref{eq:solution0}) of problem (\ref{eq:opt2}), and (\ref{eq:conv0}) holds.
\end{theorem}
\begin{proof} In the proof of Theorem~\ref{thm:spaces} we have demonstrated that starting from $k=k_1$ defined in (\ref{eq:k1}) Algorithm~\ref{alg:ccd} is mathematically equivalent to Algorithm~\ref{alg:aug} starting from an initial approximation $\mathbf{u}_{k_1-1}$, $\mathbf{z}_{k_1-1}$ and $\mathbf{b}_{k_1-1}$. Convergence of Algorithm~\ref{alg:aug} does not depend on these starting values, hence Algorithm~\ref{alg:ccd} converges to the same unique solution as Algorithm~\ref{alg:aug} and (\ref{eq:conv0}) holds.
\end{proof}

The result of Theorem~\ref{thm:ccd} indicates that our Compressive Conjugate Directions method matches the ADMM in exact arithmetic after a finite number of iterations, while avoiding direct inversion of operator $\mathbf{A}$. This obvously means that the (worst-case) asymptotic convergence rate of Algorithm~\ref{alg:ccd} matches that of the ADMM and is $O(1/k)$ \cite{He2012}.

\section{Limited-memory Compressive Conjugate Directions Method}

Algorithm~\ref{alg:ccd} (that we call ``unlimited-memory'' Compressive Conjugate Directions Method) requires storing all of the previous conjugate directions (\ref{eq:pq}) because in step 7 the algorithm computes the expansion
        \begin{equation}
                \mathbf{u}_{k+1} \; =\; \sum_{i=0}^k \tau_i \mathbf{p}_i,
                \label{eq:exp}
        \end{equation}
of these solution approximations with respect to all conjugate direction vectors (\ref{eq:pq}) at each iteration. It is a consequence of changing right-hand sides of the normal equations system (\ref{eq:Fuvk}) that \emph{all} of the coefficients of expansion (\ref{eq:exp}) may require updating. However, in a practical implementation we may expect that only the last $m+1$ expansion coefficients (\ref{eq:exp}) significantly change, and freeze the coefficients \[\tau_i,\;i\;< \;k-m\]  at and after iteration $k$. This approach requires storing up to $2m+2$ latest vectors
\begin{equation}
        \mathbf{p}_{k},\mathbf{p}_{k-1},\ldots,\mathbf{p}_{k-m},\;\;
        \mathbf{q}_{k},\mathbf{q}_{k-1},\ldots,\mathbf{q}_{k-m}.
        \label{eq:pqm}
\end{equation}
A ``limited-memory'' variant of the method is implemented in Algorithm~\ref{alg:lmccd} that stores vectors (\ref{eq:pqm}) in a circular first-in-first-out buffer. An index variable $j$ points to the latest updated element within the buffer. Once $j$ exceed the buffer size for the first time and is reset to point to the head of the buffer, a flag variable $cycle$ is set, indicating that a search direction is overwritten at each subsequent iteration of the algorithm. The projection of the current solution iterate onto the old vector $\tau_j\mathbf{p}_j$ (now to be overwritten in the buffer) is then accumulated in a vector $\tilde{\mathbf{u}}$; the corresponding contribution to the predicted data equals $\tau_j\mathbf{q}_j$ and is accumulated in a vector $\tilde{\mathbf{v}}$, 
        \begin{equation}
        \begin{aligned}
                \tilde{\mathbf{u}} \; =\; \sum_{i=0}^{k-m-1} \tau_i \mathbf{p}_i,\;\;
                \tilde{\mathbf{v}} \; =\; \sum_{i=0}^{k-m-1} \tau_i \mathbf{q}_i.
                \label{eq:tildes}
        \end{aligned}
        \end{equation}
        Contributions (\ref{eq:tildes}) to the solution and predicted data from the discarded vectors (\ref{eq:pq}) are then added back to the approximate solution and residual in steps 8 and 12 of Algorithm~\ref{alg:lmccd}.

        \begin{algorithm}[htl]
        \caption{Limited-Memory Compressive Conjugate Directions Method for (\ref{eq:opt0})}
 \label{alg:lmccd}
        \begin{algorithmic}[1]
        \State $m\;\gets\; \text{memory size}, \; \tilde{\mathbf{u}} \;\gets\; \boldsymbol 0^N,\; \tilde{\mathbf{v}} \;\gets\; \boldsymbol 0^{N+K},\; j\gets0,\; cycle\gets.false.$
        \State $\mathbf{u}_{0} \;\gets\; \boldsymbol 0,\; \mathbf{z}_0\;\gets\; \boldsymbol 0^K;\; \mathbf{b}_0\;\gets\; \boldsymbol 0^K,\; \mathbf{v}_0\;\gets\;\begin{bmatrix}
                \sqrt{\alpha}\mathbf{d}\\
                \sqrt{\lambda}\left(\mathbf{z}_0+\mathbf{b}_0\right)
\end{bmatrix}$
\State {$\mathbf{p}_0\;\gets\; \mathbf{F}^T \mathbf{v}_0,\;\mathbf{q}_0\;\gets\;\mathbf{F}\mathbf{p}_0,\;\delta_0\;\gets\; \mathbf{q}^T_0\mathbf{q}_0$}
\For {$k=0,1,2,3,\ldots$}
    \For {$i=0,1,\ldots,\min(k,m)$}
    \State {$\tau_i\;\gets\; {\mathbf{q}_i^T \left(\mathbf{v}_{k}- \tilde{\mathbf{v}}\right)}/\delta_i$}
    \EndFor
    \State {$\mathbf{u}_{k+1}\;\gets\; \tilde{\mathbf{u}} \;+\; \sum_{i=0}^{\min(k,m)} \tau_i \mathbf{p}_i$}
    \State {$\mathbf{z}_{k+1}\;\gets\; \mathrm{shrink}\left\{\mathbf{B} \mathbf{u}_{k+1}-\mathbf{b}_k,{1}/{\lambda}\right\}$}
    \State {$\mathbf{b}_{k+1}\;\gets\; \mathbf{b}_k +\mathbf{z}_{k+1} - \mathbf{B} \mathbf{u}_{k+1}$}
    \State {$\mathbf{v}_{k+1}\;\gets\;\begin{bmatrix}
                \sqrt{\alpha}\mathbf{d}\\
                \sqrt{\lambda}\left(\mathbf{z}_{k+1}+\mathbf{b}_{k+1}\right)
\end{bmatrix}$}
\State{$\mathbf{r}_{k+1}\;\gets\;\mathbf{v}_{k+1}\;-\; \sum_{i=0}^{\min(k,m)}\tau_i  \mathbf{q}_i \;-\; \tilde{\mathbf{v}}$}
    \State {$\mathbf{w}_{k+1}\;\gets\;  \mathbf{F}^T \mathbf{r}_{k+1}$}
    \State {$\mathbf{s}_{k+1}\;\gets\;  \mathbf{F} \mathbf{w}_{k+1}$}
    \For {$i=0,1,\ldots,\min(k,m)$}
        \State {$\beta_i\;\gets\;  - {\mathbf{q}_i^T \mathbf{s}_{k+1}}/\delta_i$}
    \EndFor
    \State {$j\gets j+1$}
    \If {$j = m+1$}
      \State {$j\gets0,\; cycle\gets.true.$} 
    \EndIf
    \If {$cycle$}
    \State {$\tilde{\mathbf{u}}\;\gets\; \tilde{\mathbf{u}} \;+\; \tau_j\mathbf{p}_j $} 
    \State {$\tilde{\mathbf{v}}\;\gets\; \tilde{\mathbf{v}} \;+\; \tau_j\mathbf{q}_j $} 
    \EndIf
    \State {$\mathbf{p}_{j}\;\gets\;\sum_{i=0}^{\min(k,m)} \beta_i \mathbf{p}_i \; + \; \mathbf{w}_{k+1}$}
    \State {$\mathbf{q}_{j}\;\gets\;\sum_{i=0}^{\min(k,m)} \beta_i \mathbf{q}_i \; + \; \mathbf{s}_{k+1}$}
    \State {$\delta_{j} \;\gets\;\mathbf{q}_{j}^T \mathbf{q}_{j}$}
    \If {$\delta_{j} = 0$}
    \Comment {Use condition ``$\delta_{j}  <  \text{tolerance}$'' in practice}
         \State {$\delta_{j}\;\gets\;1,\; \mathbf{p}_{j}\;\gets\;\mathbf{0}^N,\; \mathbf{q}_{j}\;\gets\;\mathbf{0}^{M+K}$}
    \EndIf
    \State {Exit loop if ${\|\mathbf{u}_{k+1}-\mathbf{u}_k\|_2}/{\| \mathbf{u}_k \|_2} \; \le \; \text{target accuracy}$}
\EndFor
%    \algstore{lm}
\end{algorithmic}
\end{algorithm}
%        \begin{algorithm}[htl]
%                \ContinuedFloat
%        \caption{Limited-Memory Compressive Conjugate Directions Method (continued)}
%        \begin{algorithmic}[1]
%        \algrestore{lm}
%\end{algorithmic}
%\end{algorithm}

\subsection{Trade-off between the number of iterations and problem condition number}
\label{subs:tradeoff}

In practical implementations of the ADMM when the operator $\mathbf{A}$ does not lend itself to direct solution methods, an iterative method can be used to solve the minimization problem in step 4 of Algorithm~\ref{alg:aug} \cite{GoldOsher09}. Algorithm~\ref{alg:rcg}, representing such an approach, uses a fixed number of iterations $N_c$ of CGNE in step 4. At each iteration of the ADMM conjugate gradients are hot-restarted from the previous solution approximation $\mathbf{u}_{k}$. For comparison purposes we will refer to this method as \emph{restarted Conjugate Gradients} or \emph{RCG}.

\begin{algorithm}
       \caption{ADMM and hot-restarted CG (\emph{RCG})}
 \label{alg:rcg}
\begin{algorithmic}[1]
        \State $\mathbf{u}_{0} \;\gets\; \boldsymbol 0^N,\; \mathbf{z}_0\;\gets\; \boldsymbol 0^K,\; \mathbf{b}_0\;\gets\; \boldsymbol 0^K,\; N_c\;\gets\; \text{prescribed number of CG iterations}$
\State {$\mathbf{p}_0\;\gets\; \mathbf{F}^T \mathbf{v}_0,\;\mathbf{q}_0\;\gets\;\mathbf{F}\mathbf{p}_0$}
\For {$k=0,1,2,3,\ldots$}
\State {Solve \[\mathbf{u}_{k+1}\;\gets\; \mathrm{argmin}\,  \left\{\frac{\lambda}{2} \| \mathbf{z}_k - \mathbf{B} \mathbf{u} + \mathbf{b}_k \|_2^2 + \frac{\alpha}{2}\| \mathbf{A} \mathbf{u} - \mathbf{d} \|_2^2\right\},\]}
\Statex {$\;\;\;\; $ starting from $\mathbf{u}_{k}$ and using $N_c$ iterations of CGNE.}
    \State {$\mathbf{z}_{k+1}\;\gets\; \mathrm{shrink}\left\{\mathbf{B} \mathbf{u}_{k+1}-\mathbf{b}_k,{1}/{\lambda}\right\}$}
    \State {$\mathbf{b}_{k+1}\;\gets\; \mathbf{b}_k + \mathbf{z}_{k+1} - \mathbf{B} \mathbf{u}_{k+1}$}
    \State {Exit loop if ${\|\mathbf{u}_{k+1}-\mathbf{u}_k\|_2}/{\| \mathbf{u}_k \|_2} \; \le \; \text{target accuracy}$}
\EndFor
\end{algorithmic}
\end{algorithm}

Note that Algorithm~\ref{alg:rcg} with $N_c=1$ performs a single step of gradient descent when solving the following intermediate least-squares minimization problem in step 4,
\begin{equation}
\begin{aligned}
        \mathbf{u}_{k+1} \;=\; \mathrm{argmin}\, \frac{\alpha}{2}\|\mathbf{A} \mathbf{u}-\mathbf{d} \|^2_2 \; +\; \frac{\lambda}{2} \| \mathbf{z}_k\;-\; \mathbf{B} \mathbf{u}+ \mathbf{b}_k \|_2^2.
\end{aligned}
\label{eq:interm}
\end{equation}
The performance of Algorithm~\ref{alg:rcg} depends on the condition number of the leasts-squares problem (\ref{eq:interm}) \cite{NLA}: for well-conditioned problems only a small number of conjugate gradients iterations $N_c$ may achieve a sufficiently accurate approximation to $\mathbf{u}_{k+1}$. The condition number of (\ref{eq:interm}) depends on properties of operators $\mathbf{A}$ and $\mathbf{B}$, as well as the value of $\lambda$. In applications with a simple modeling operator $\mathbf{A}$, such as is the case in denoising with $\mathbf{A}=\mathbf{I}$, a value of $\lambda$ may be experimentally selected so as to reduce the condition number of (\ref{eq:interm}). However, a trade-off may exist between the condition number of (\ref{eq:interm}) and the number of ADMM iterations in the outer loop (Step 3) of Algorithm~\ref{alg:aug}: well-conditioned interim least-squares problems may result in a significantly higher number of ADMM iterations. Such a trade-off is a well-known phenomenon in applications of the Augmented Lagrangian Method of Multipliers for smooth objective functions, see, e.g., \cite{GlowinskiTallec1989}. For example, large values of $\lambda$ in (\ref{eq:opt4}) more strongly penalize violations of the equality constraint, as in the Quadratic Penalty Function Method \cite{Nocedal} with a larger penalty and a more ill-conditioned quadratic minimization.  Of course, in the case of ADMM applied to (\ref{eq:opt0}), a non-smooth objective function, arbitrary and potentially ill-conditioned operator $\mathbf{A}$, and (most importantly) alternating splitting minimization of the modified Augmented Lagrangian (\ref{eq:opt4})\footnote{``modified'' because of the added constant term $\lambda /2 \|\mathbf{b}_k\|^2_2$} complicate the picture. In fact, for an arbitrary $\mathbf{A}$, the condition number of (\ref{eq:interm}) is not always an increasing function of $\lambda$. Some of the numerical examples described in the following subsections exhibit this trade-off between the condition number of the intermediate least-squares problem (\ref{eq:interm}) and the number of ADMM iterations: the better the condition-number of (\ref{eq:interm}), the more ADMM iterations are typically required. The main advantage of our Compressive Conjugate Directions approach implemented in Algorithms~\ref{alg:ccd} and \ref{alg:lmccd} is that information on the geometry of the objective function (\ref{eq:interm}) accumulates through \emph{external} ADMM iterations thus potentially reducing the amount of effort required to perform minimization of (\ref{eq:interm}) at each step. Since our objective is a practical implementation of the ADMM for (\ref{eq:opt0}) with computationally expensive operators $\mathbf{A}$, the overall number of operator $\mathbf{A}$ and $\mathbf{A}^T$ applications required to achieve given accuracy will be the principal benchmark for measuring the performance of various algorithms.

\section{Applications}

In this section we apply the method of Compressive Conjugate Directions to solving $L_1$ and TV-regularized inversion problems for several practical examples.

\subsection{Image Denoising}
A popular image denoising technique for removing short-wavelength random Gaussian noise from an image is based on solving (\ref{eq:opt}) with $\mathbf{A}=\mathbf{I}$. Vector $\mathbf{d}$ is populated with a noisy image, a denoised image is returned in $\mathbf{u}$,
\[\mathbf{u}=u_{i,j},\;i=1,\ldots,N_y,\;j=1,\ldots,N_x,\]
with an \emph{anisotropic TV norm} in (\ref{eq:opt}) defined by the linear gradient operator
\begin{equation}
        \nabla         \mathbf{u} \;=\; 
        \begin{bmatrix} 
                \nabla_x \mathbf{u} \\
                \nabla_y \mathbf{u} 
        \end{bmatrix} \;=\; 
        \begin{bmatrix} 
                {u}_{i,2} - {u}_{i,1} \\
                \cdots\\
                {u}_{i,N_x} - {u}_{i,N_x-1} \\
                \cdots\\
                {u}_{2,j} - {u}_{1,j} \\
                \cdots\\
                {u}_{N_y,j} - {u}_{N_y-1,j}
        \end{bmatrix},\; i=1,\ldots,N_y,\; j=1,\ldots,N_x.
        \label{eq:aniso}
\end{equation}
Here, the dimension of the model space is $N=N_x\times N_y$ with $M=N$ and $K=N- N_x - N_y$. Since operator $\mathbf{A}=\mathbf{I}$ is trivial, minimization of the number of operator applications in this problem carries no practical advantage. The only reason for providing this example is to demonstrate the stability of the proposed Compressive Conjugate Directions method with respect to choosing a value of $\lambda$.

Figure~\ref{fig:trueimg} shows the true, noise-free $382 \times 382$ image used in this experiment. Random Gaussian noise with a standard deviation $\sigma$ of $15\%$ of maximum signal amplitude was added to the true image to produce the noisy image of Figure~\ref{fig:noisyimg}. All low-wavenumber or ``blocky'' components of the noise below a quarter of the Nyquist wavenumber were filtered out, leaving only high-wavenumber ``salt-and-pepper'' noise. Parameter $\alpha=10$ was chosen experimentally based on the desired trade-off of fidelity and ``blockiness'' of the resulting denoised image. The result of solving (\ref{eq:opt}) using Algorithm~\ref{alg:rcg} with $\lambda=1$, one hundred combined applications of $\mathbf{A}$ and $\mathbf{A}^T$, and $N_c=1$ is shown in Figure~\ref{fig:rcg}. The result of applying our limited-memory Conjugate Directions Algorithm~\ref{alg:lmccd} for $m=50$ is shown in Figure~\ref{fig:ccd}\footnote{Here, this matches the results for \emph{any} memory size $m>0$ due to a well-conditioned problem (\ref{eq:interm}).}. Note that $N_c=1$ means that only a single step of Conjugate Gradients, or a single gradient descent, is made in step 4 of Algorithm~\ref{alg:rcg}. For this choice of $\lambda$, problem (\ref{eq:interm}) is very well conditioned, with a condition number of $\kappa = 1.8$. A single iteration of gradient descent achieves sufficient accuracy of minimization (\ref{eq:interm}) and for $\lambda=1$ there is no practical advantage in using our method as both methods perform equally well, see Figure~\ref{fig:e1}. In fact, the overhead of storing and using conjugate directions from previous iterations may exceed the cost of operator $\mathbf{A}$ and its adjoint applications if the latter are computationally cheap. 
\begin{figure}[htb]
  \begin{center}
          \subfigure[]{\includegraphics[width=.46\textwidth]{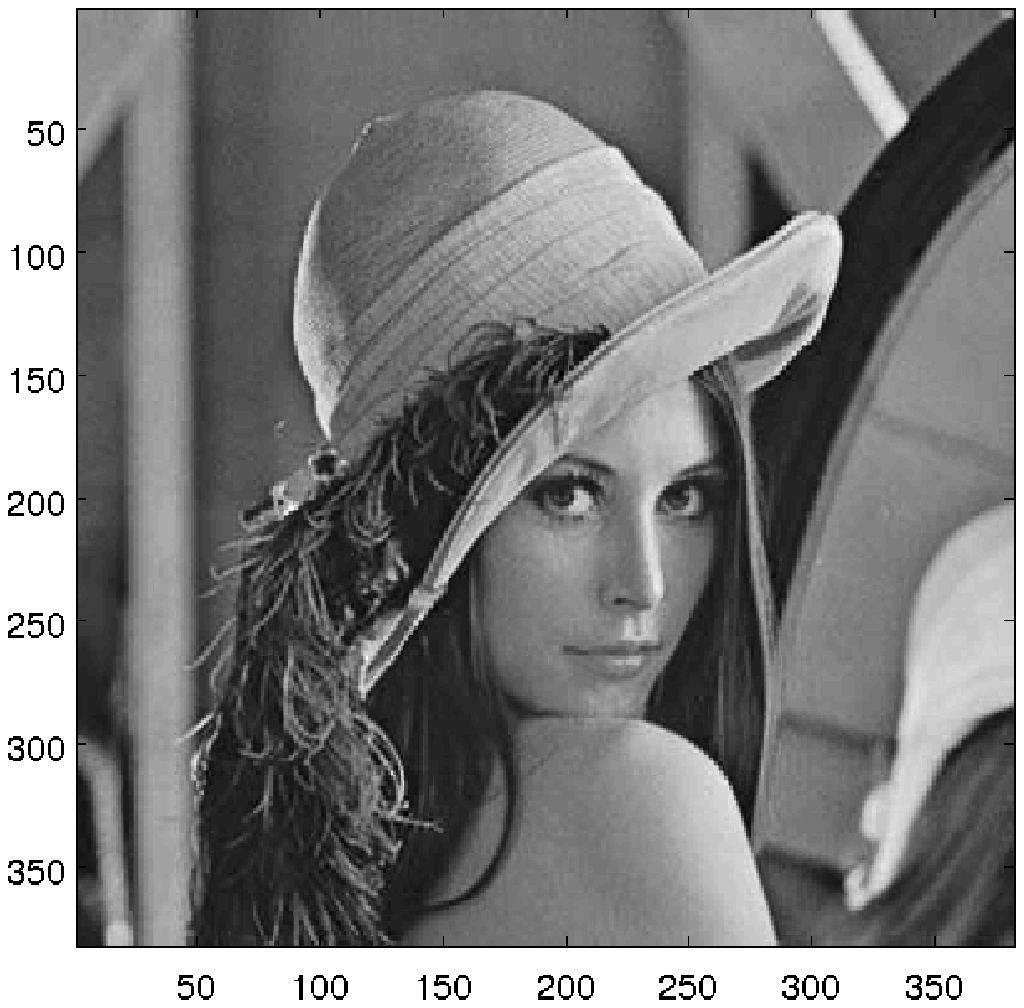}\label{fig:trueimg}}
          \quad
          \subfigure[]{\includegraphics[width=.46\textwidth]{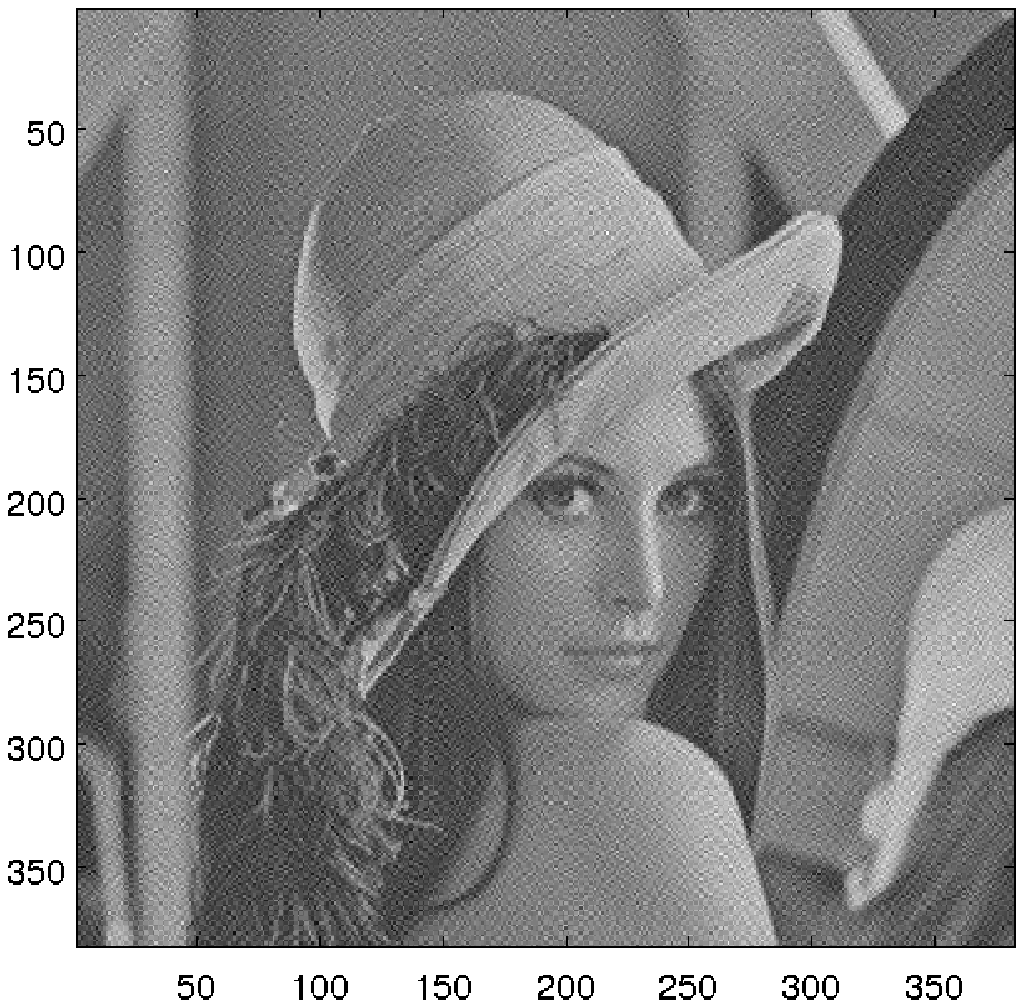}\label{fig:noisyimg}}\\
          \vspace{-.3cm}
          \subfigure[]{\includegraphics[width=.46\textwidth]{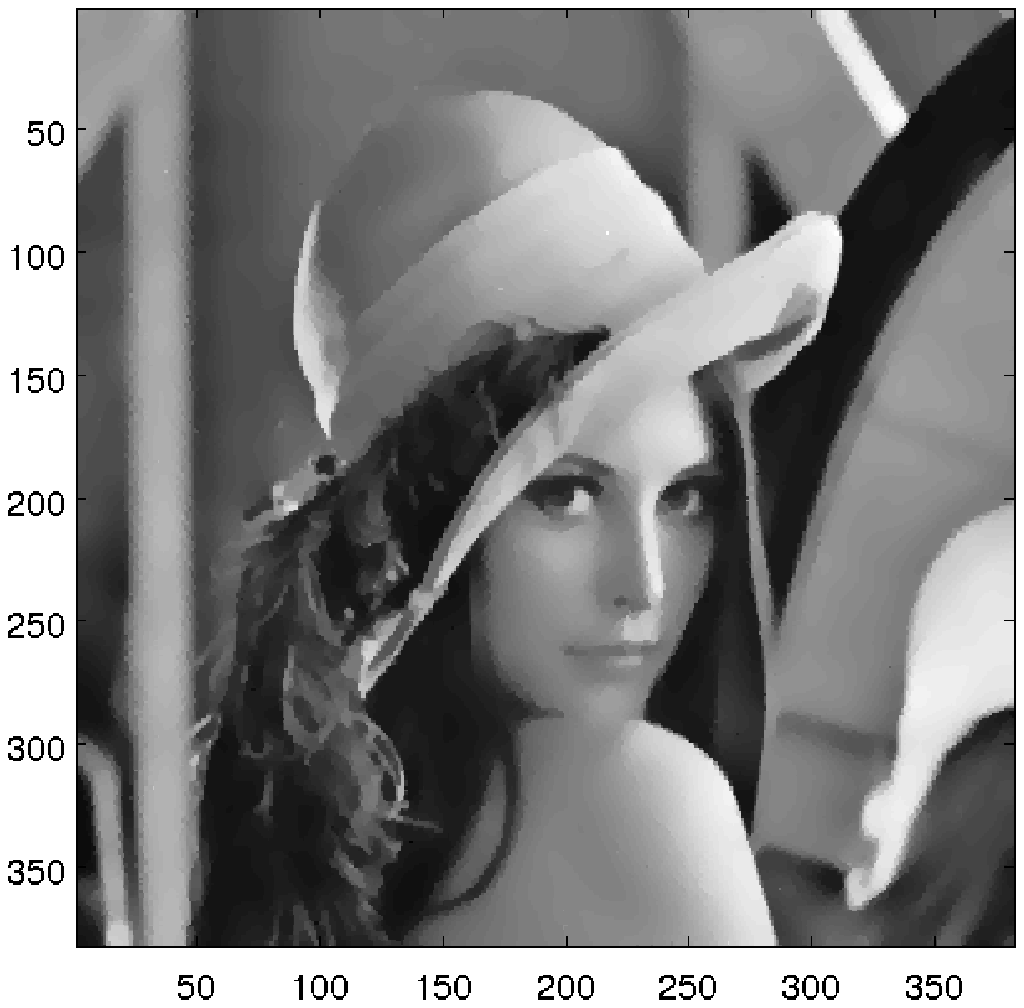}\label{fig:ccd}}
          \quad
          \subfigure[]{\includegraphics[width=.46\textwidth]{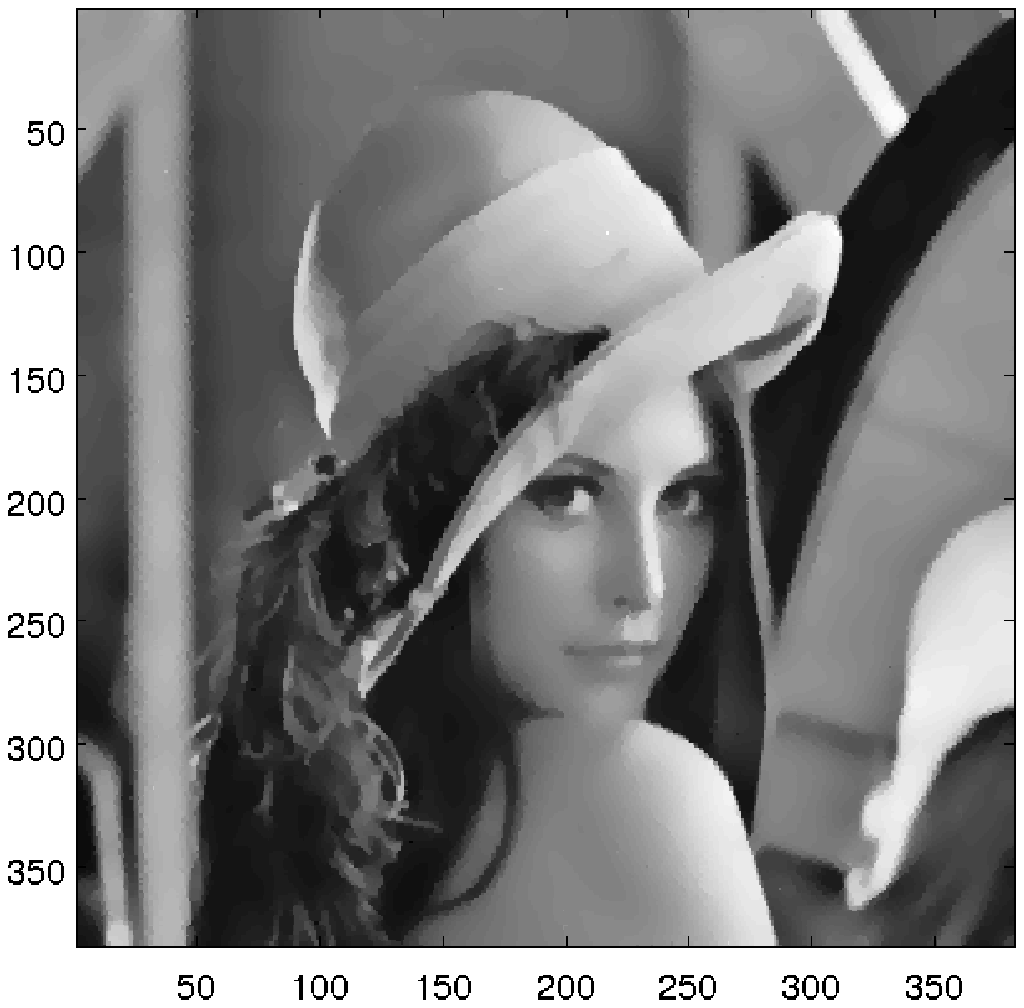}\label{fig:rcg}}
          \vspace{-.5cm}
          \caption{(a) Clean image; (b) Noisy image contaminated with Gaussian noise with $\sigma=15\%$ of maximum amplitude; (c) Image denoised using Algorithm~\ref{alg:lmccd} with $\alpha=10,\;\lambda=1$ and memory size $m=50$; (d) Image denoised using Algorithm~\ref{alg:rcg} with $\alpha=10,\;\lambda=1,\;N_c=1$.}
  \end{center}
\end{figure}
The approximation errors of applying the limited-memory Compressive Conjugate Directions Algorithm~\ref{alg:lmccd} with $m=50$ versus Algorithm~\ref{alg:rcg} with $N_c=1,5,10$ for $\lambda=10^2, 10^3, 10^4$ are shown in Figures~\ref{fig:e1},\ref{fig:e100},\ref{fig:e1000},\ref{fig:e10000}. Note that larger values of $\lambda$ result in increasingly larger condition numbers of (\ref{eq:interm}) shown on top of the plots. The performance of Algorithm~\ref{alg:rcg} here depends on a choice of $N_c$: increasing $N_c$ as required to achieve a sufficiently accurate approximate solution of (\ref{eq:interm}) results in fewer available ADMM iterations for a fixed ``budget'' of operator $\mathbf{A}$ and adjoint applications. However, Algorithm~\ref{alg:lmccd} accumulates conjugate directions (\ref{eq:pq}) computed at earlier iterations and requires only one application of the operator and its adjoint per ADMM iteration. Note that at iteration steps less than $N_c$, Algorithm~\ref{alg:rcg} may still outperform Algorithm~\ref{alg:lmccd} as it conducts more Conjugate Gradient iterations per solution of each problem (\ref{eq:interm}). However, once the ADMM iteration count exceeds the largest $N_c$, and sufficient information is accumulated by Algorithm~\ref{alg:lmccd} about the geometry of the objective function, the Compressive Conjugate Directions outperforms Algorithm~\ref{alg:rcg}.
\begin{figure}[htb]
  \begin{center}
          \subfigure[]{\includegraphics[width=.48\textwidth]{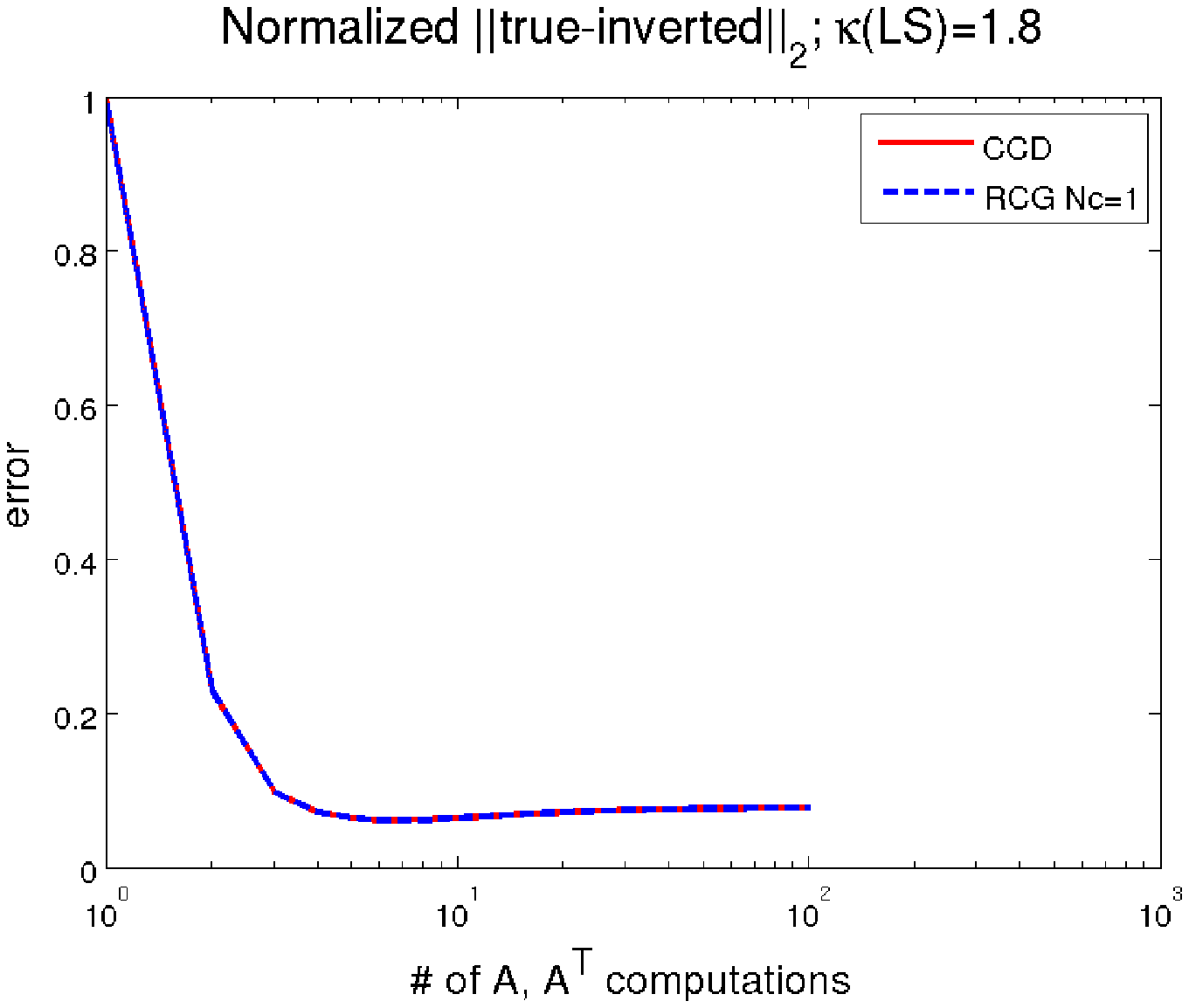}\label{fig:e1}}
          \quad
          \subfigure[]{\includegraphics[width=.48\textwidth]{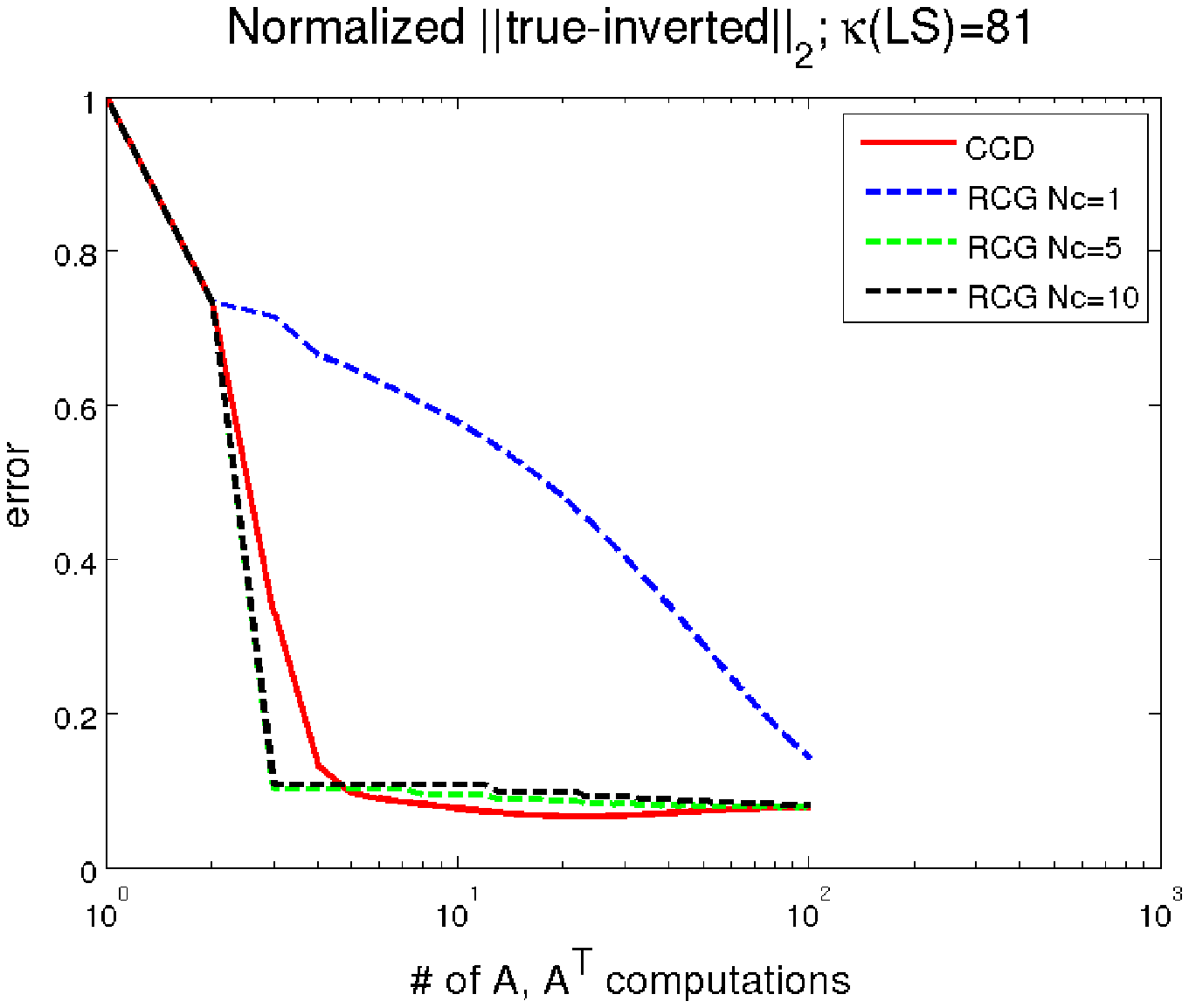}\label{fig:e100}}\\
          \vspace{-.3cm}
          \subfigure[]{\includegraphics[width=.48\textwidth]{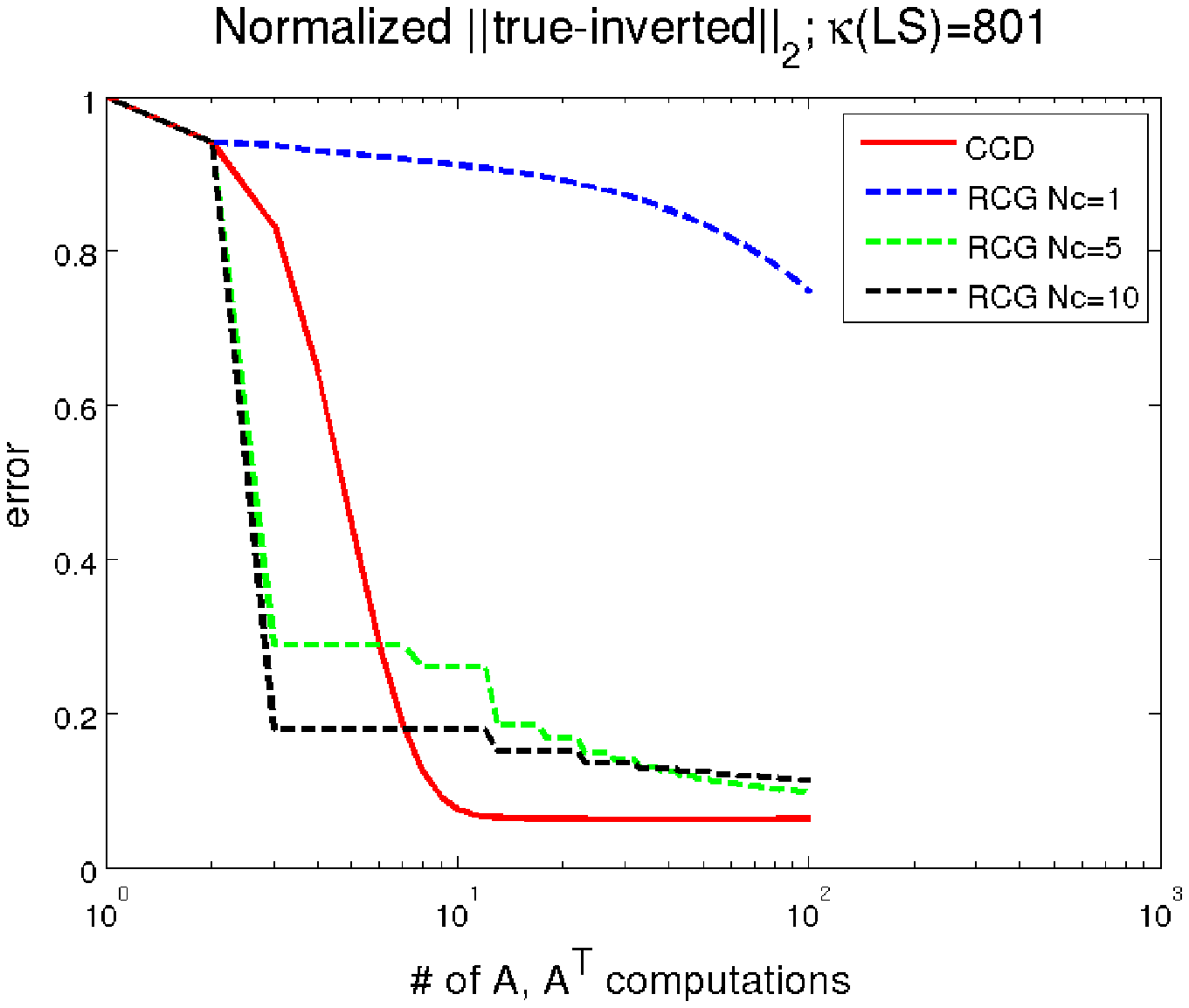}\label{fig:e1000}}
          \quad
          \subfigure[]{\includegraphics[width=.48\textwidth]{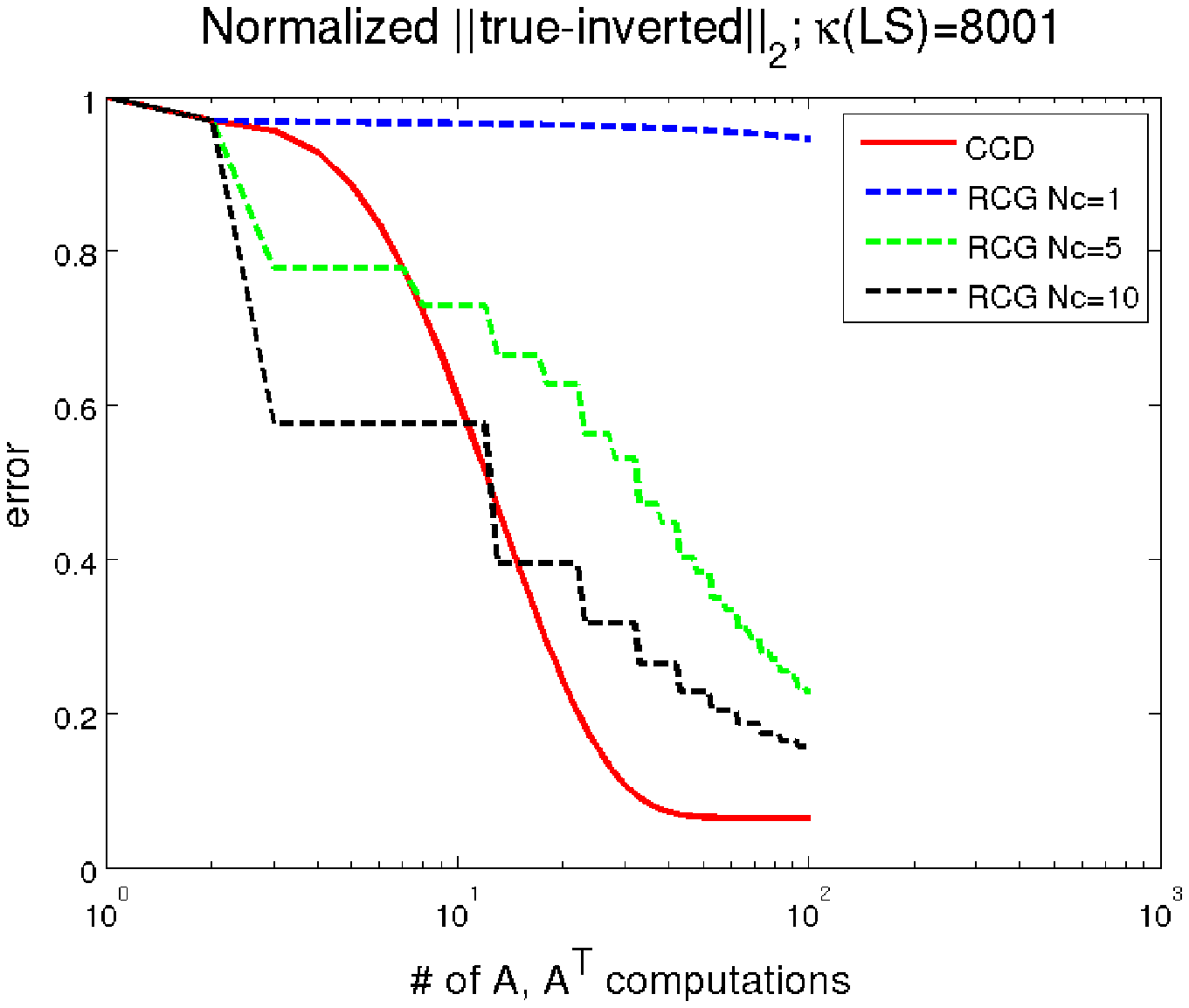}\label{fig:e10000}}
          \vspace{-.5cm}
          \caption{ Performance of Algorithm~\ref{alg:lmccd} with $m=20$ versus Algorithm~\ref{alg:rcg} with varying $N_c$ for (a) $\lambda=1$ (b) $\lambda=100$ (c) $\lambda=1000$ (d) $\lambda=10000$.}
  \end{center}
\end{figure}
Note that this example does not demonstrate the trade-off between the condition number of (\ref{eq:interm}) and the number of ADMM iterations. The reason for this is that for large $\lambda$ convergence is achieved relatively quickly within a number of iterations comparable to a number of Conjugate Gradients steps required to solve (\ref{eq:interm}). However, this example demonstrate another feature of the proposed Compressive Conjugate Directions Method: compared with a technique based on a restarted iterative solution of (\ref{eq:interm}), the method may be less sensitive to a suboptimal choice of $\lambda$.

\subsection{Inversion of Dilatational Point Pseudo-sources}

In our second example, we demonstrate our method on a geomechanical inversion problem with a non-trivial forward-modeling operator $\mathbf{A}$. Here, we are interested in inverting subsurface sources of deformation from noisy measurements of surface displacements, such as GPS, tilt-meter and InSAR observations.

 The forward modeling operator simulates vertical surface displacements in response to distributed dilatational (e.g. pressure change) sources \cite{SEGDEF}. Our modeling operator is defined as
\begin{equation}
\mathbf{A} \mathbf{u}\;=\; d(z),\; d(z)\;=\; c \int_0^A \frac{D u(\xi) d\xi  }  {\left(D^2 + (z-\xi)^2\right)^{3/2}},
\label{eq:F}
\end{equation}
where we assume that $\mathbf{u}=u(\xi), \xi\in [0,A]$ is a relative pore pressure change along a horizontal segment $[0,A]$ of a reservoir at a constant depth $D$, $\mathbf{d}=d(x),x\in [0,A]$ is the induced vertical displacement on the surface, and a factor $c$ is determined by the poroelastic medium properties, and reservoir dimensions. In this example, for demonstration purposes we consider a two-dimensional model, but a three-dimensional model is studied in subsection~\ref{subs:3d}. Operator (\ref{eq:F}) is a smoothing integral operator that, after discretization and application of a simple quadrature, can be represented by a dense matrix. Analytical representation of the surface displacement modeling operator (\ref{eq:F}) is possible for simple homogeneous media; however, modeling surface displacements in highly heterogeneous media will involve computationally expensive numerical methods such as Finite Elements \cite{Kosloff1980}.

In this experiment we seek to recover a spiky model of subsurface sources shown in Figure~\ref{fig:strue} from noisy observations of the induced surface displacements shown in Figure~\ref{fig:sdata}. 
\begin{figure}[htb]
  \begin{center}
          \subfigure[]{\includegraphics[width=.48\textwidth]{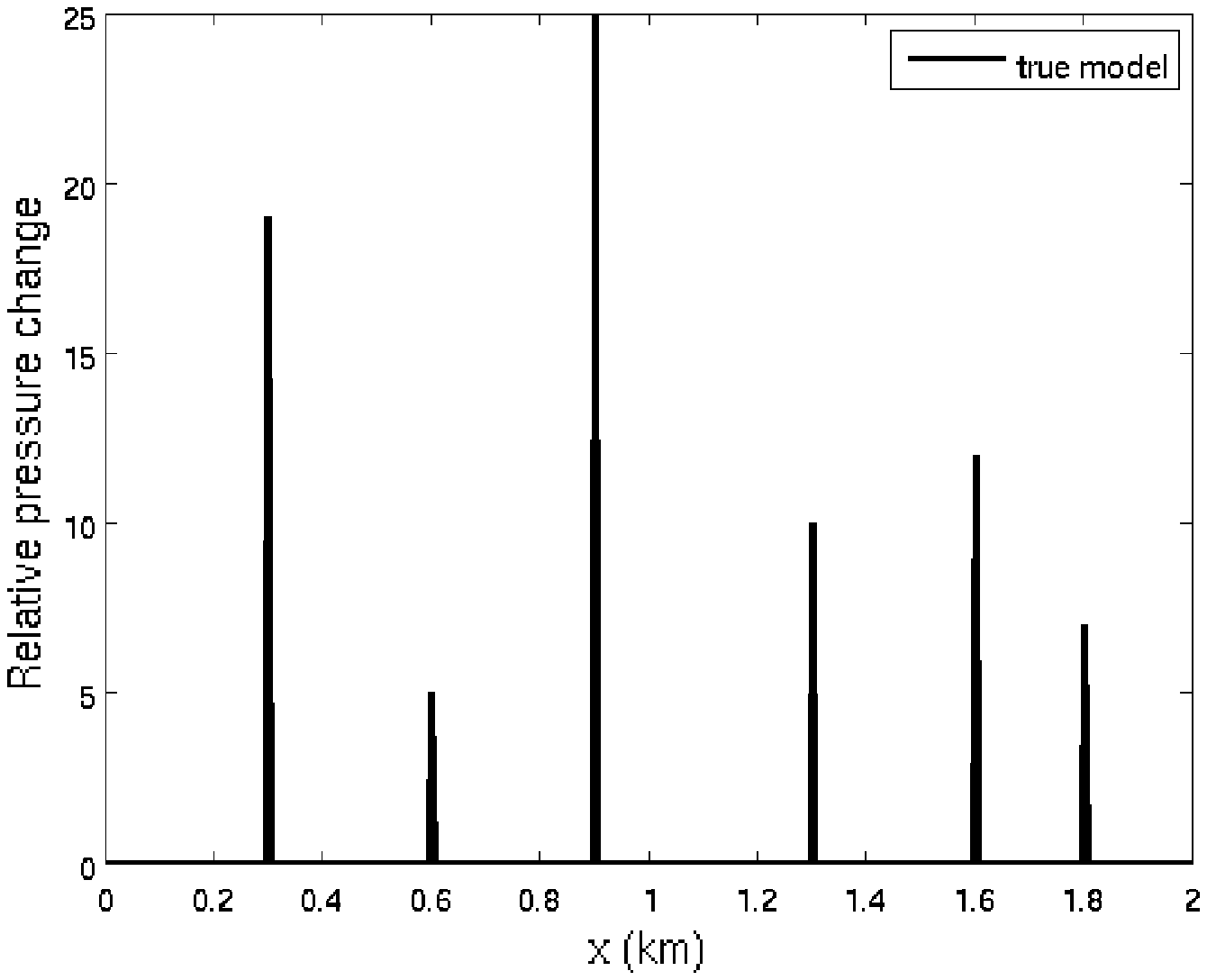}\label{fig:strue}}
          \quad
          \subfigure[]{\includegraphics[width=.48\textwidth]{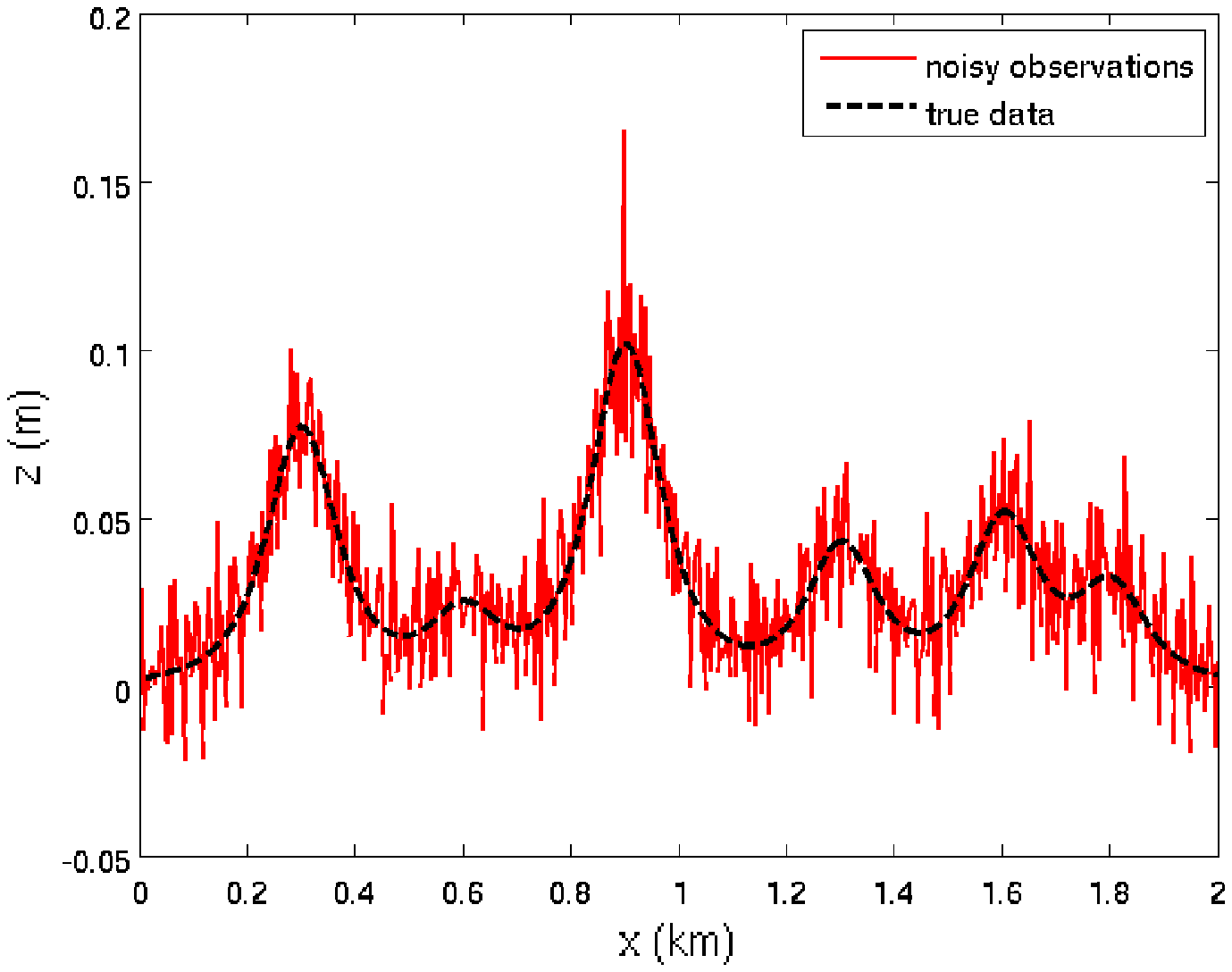}\label{fig:sdata}}
          \vspace{-.5cm}
          \caption{ (a) A spiky true pseudosources; (b) the resulting true (black) and noisy (red) surface displacements. }
  \end{center}
\end{figure}
Such sparse dilatational pseudo-sources are mathematically equivalent to concentrated reservoir pressure changes in hydrogeology and exploration geophysics, as well as expanding spherical lava chambers (the ``Mogi model'') in volcanology \cite{SEGDEF}. We forward-modeled surface displacements due to the sources of Figure~\ref{fig:strue} using operator (\ref{eq:F}), and, as in our denoising tests, added random Gaussian noise with $\sigma=15\%$ of the maximum data amplitude. Prior to adding the noise, all low-wavenumber noise components below a fifth of the Nyquist wavenumber were muted, leaving only the high-wavenumber noise shown in Figure~\ref{fig:sdata}. 

We set $D=.1$ km, $A=2$ km, $c=10^{-2}$ in (\ref{eq:F}), and discretized both the model and data space using a 500-point uniform grid, $N=M=500$. We solve problem (\ref{eq:opt1}) with $\alpha=10000$, and our objective is to accurately identify locations of the spikes in Figure~\ref{fig:strue} and their relative magnitudes, carrying out as few applications of operator (\ref{eq:F}) as possible. 
\begin{figure}[htb]
  \begin{center}
          \subfigure[]{\includegraphics[width=.48\textwidth]{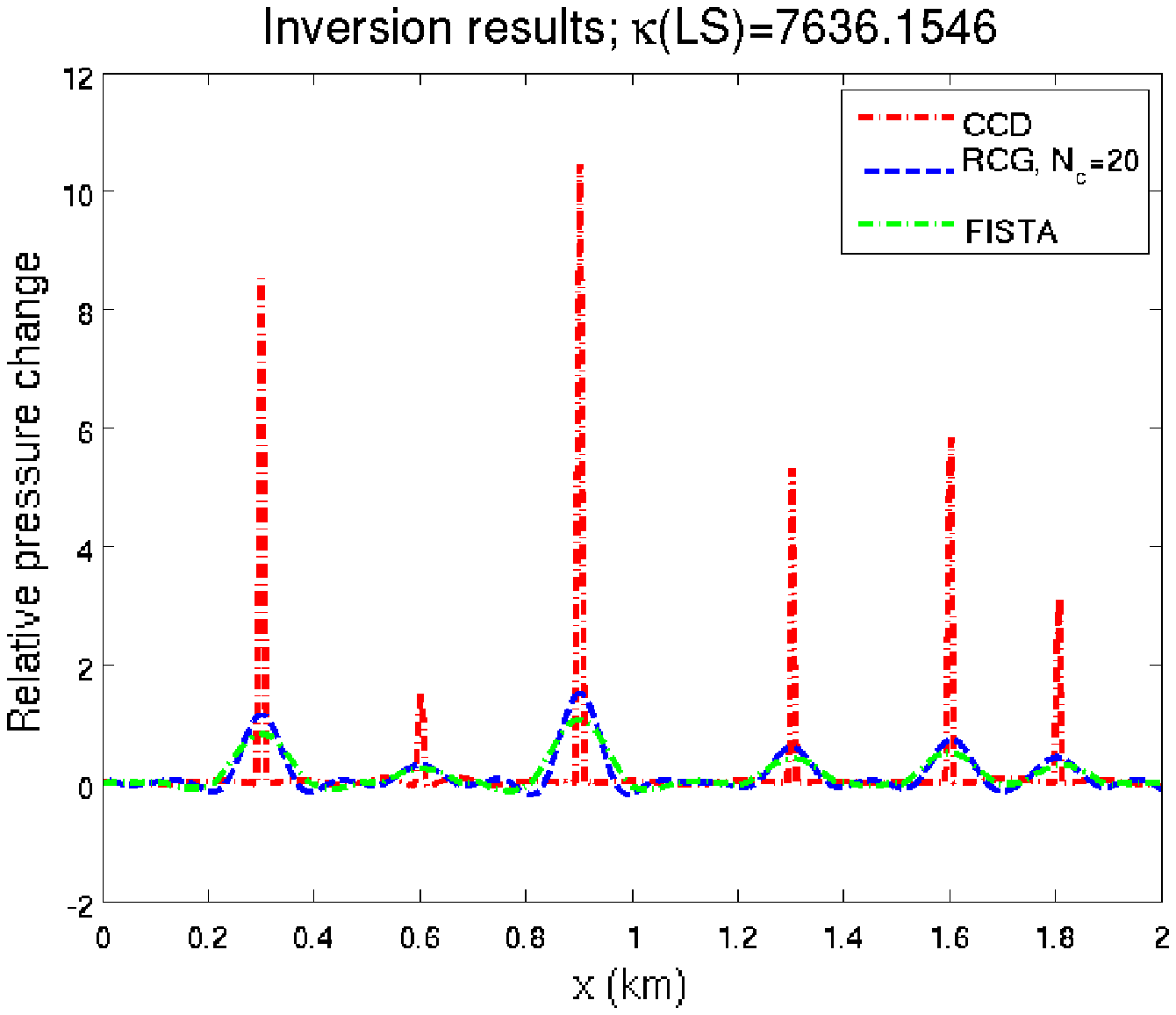}\label{fig:s005inv}}
          \subfigure[]{\includegraphics[width=.48\textwidth]{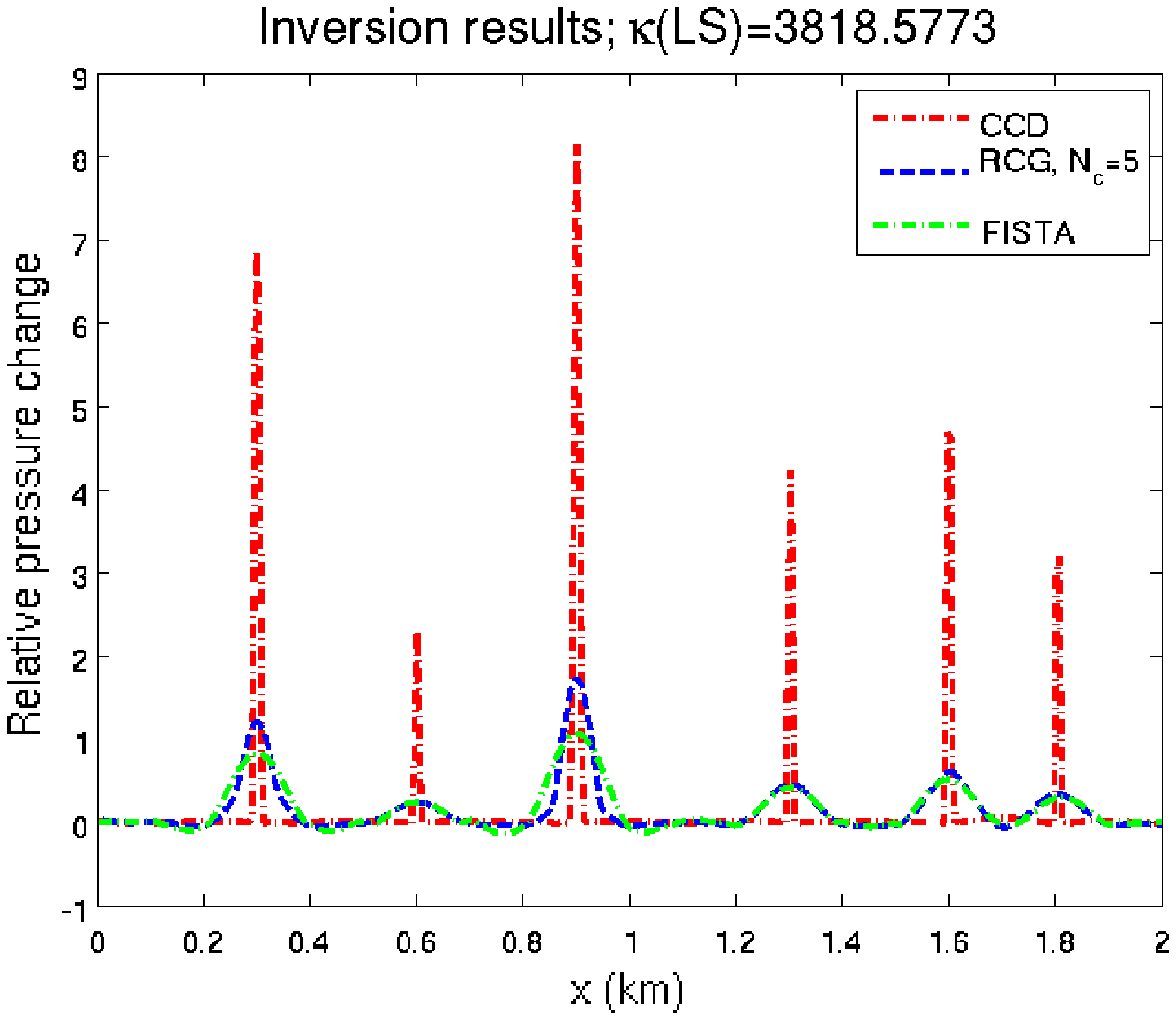}\label{fig:s01inv}}\\
          \vspace{-.3cm}
          \subfigure[]{\includegraphics[width=.48\textwidth]{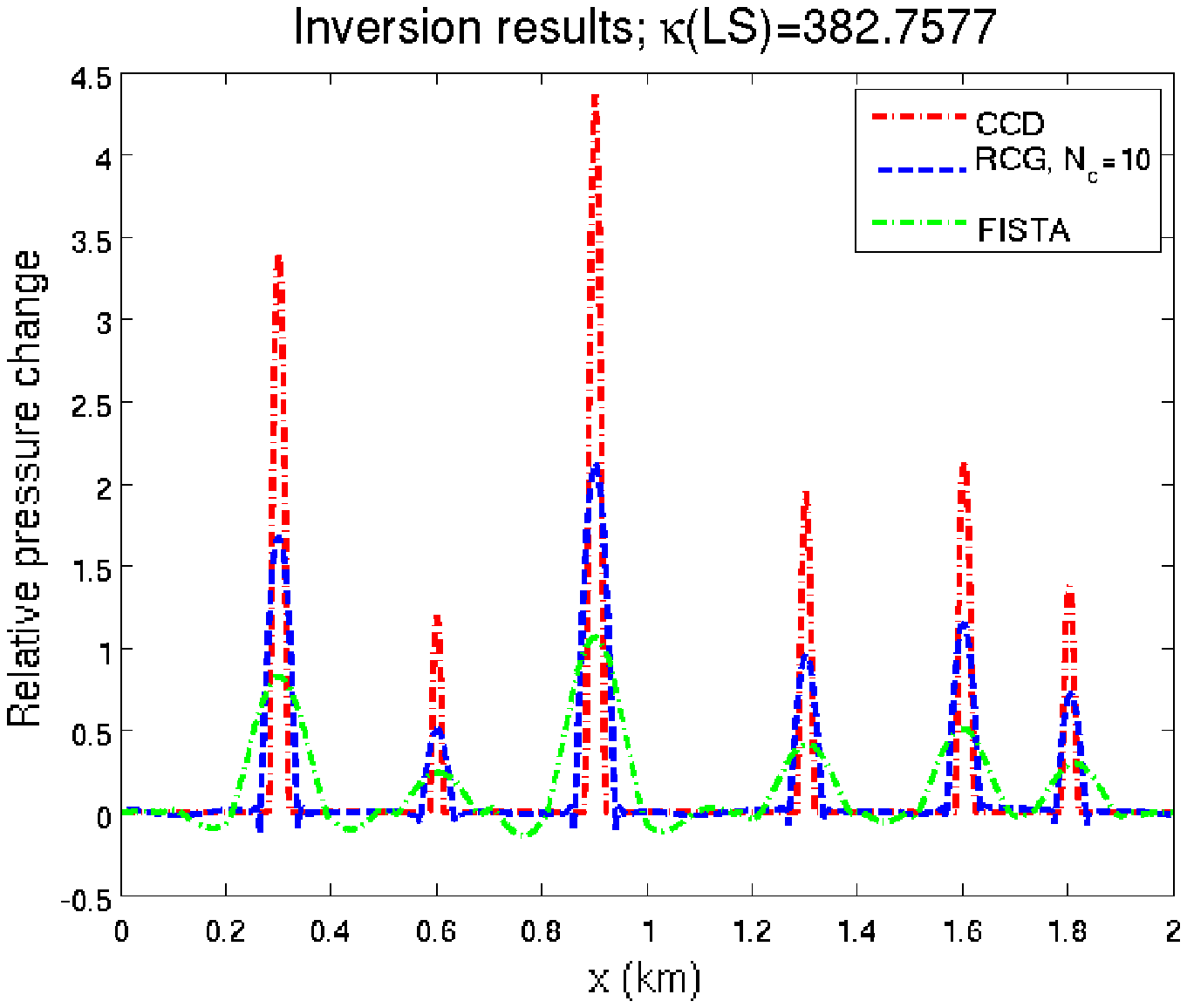}\label{fig:s1inv}}
          \subfigure[]{\includegraphics[width=.48\textwidth]{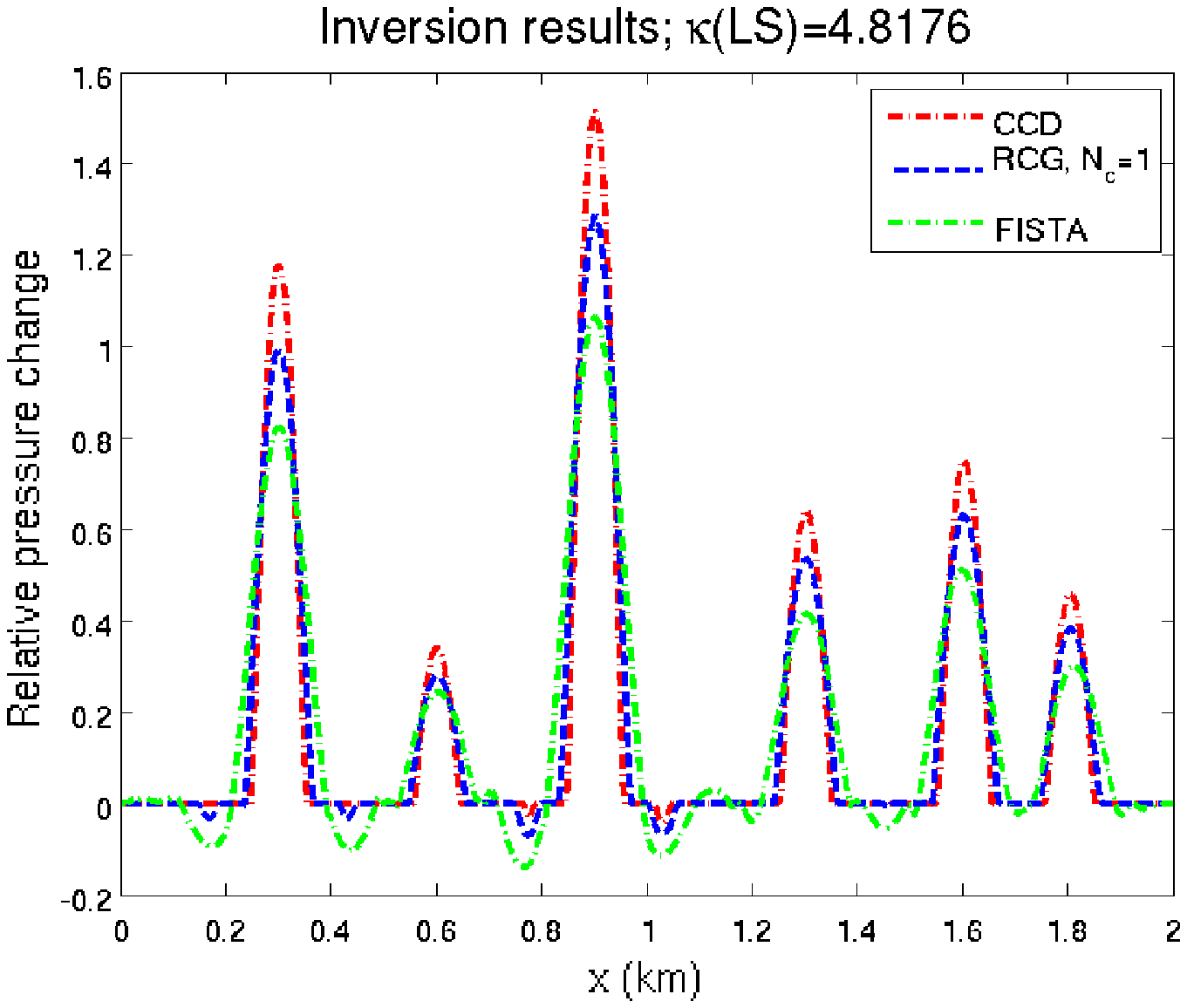}\label{fig:s100inv}}
          \vspace{-.5cm}
          \caption{ Inversion results for CCD (red), RCG (blue), FISTA (green) after 100 operator and adjoint applications for (a) $\lambda=.05$; (b) $\lambda=0.1$; (c) $\lambda=1$; (d) $\lambda=100$. Note that FISTA does not use $\lambda$ and the same FISTA results are shown in all plots but using different vertical scales. Improving condition number of (\ref{eq:interm}) is accompanied by slower convergence. Compressive Conjugate Directions method most accurately resolves the spiky model at early iterations, and performs well when (\ref{eq:interm}) is ill-conditioned.}
  \end{center}
\end{figure}
Inversion results of using the limited-memory Compressive Conjugate Directions Algorithm~\ref{alg:lmccd} with $m=100$, ADMM with restarted Conjugate Gradients Algorithm~\ref{alg:rcg} and FISTA of (\ref{eq:FISTA}) are shown in Figures~\ref{fig:s005inv},\ref{fig:s01inv},\ref{fig:s1inv},\ref{fig:s100inv} for $\lambda=0.05, 0.1, 1, 100$. In each case one hundred combined products of operators $\mathbf{A}$ and $\mathbf{A}^T$ with vectors were computed. We used the maximum FISTA step size of $\tau=10^{-4}$ in (\ref{eq:FISTA}) computed for operator (\ref{eq:F}).  
\begin{figure}[htb]
  \begin{center}
          \subfigure[]{\includegraphics[width=.48\textwidth]{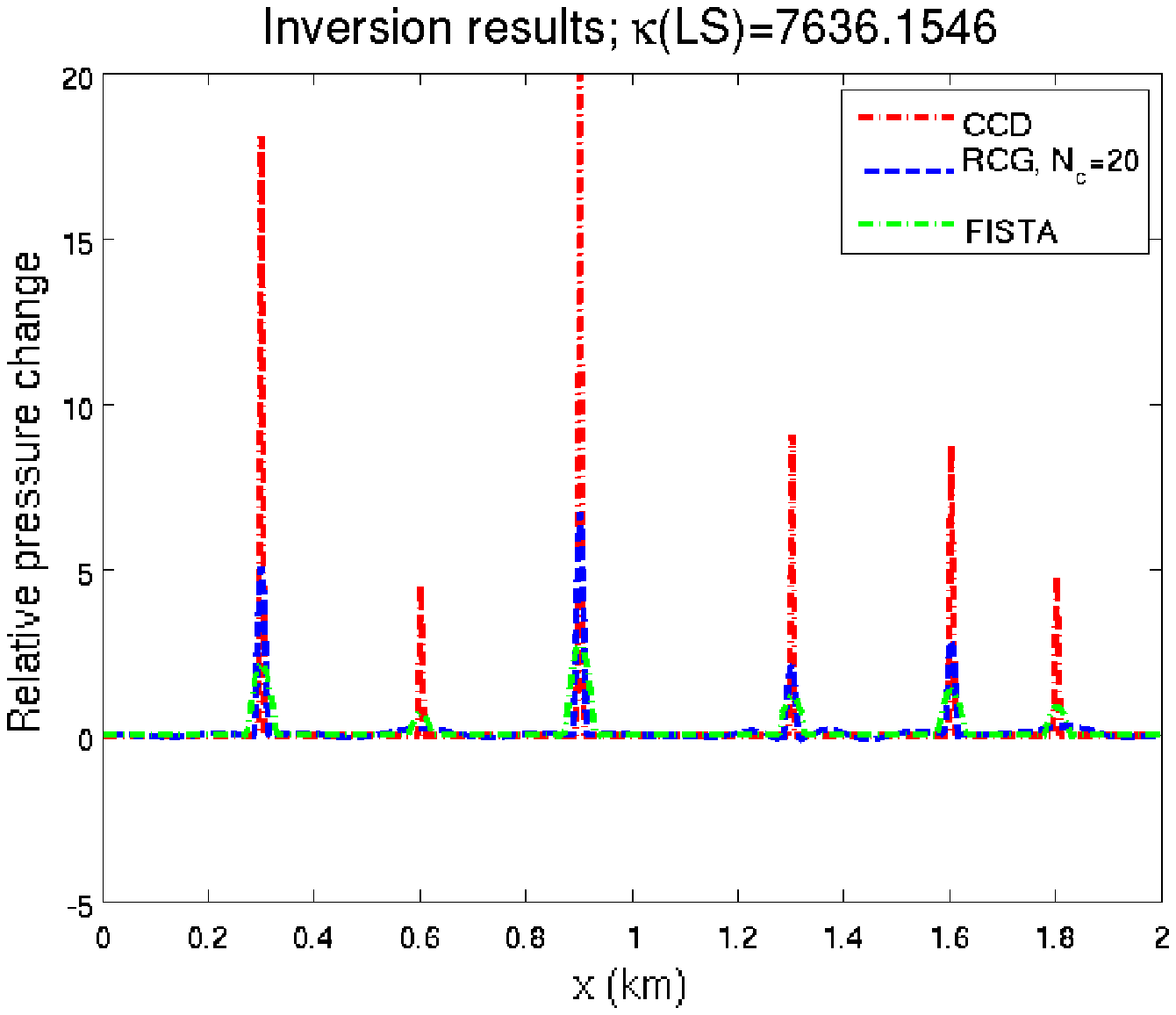}\label{fig:1s005inv}}
          \subfigure[]{\includegraphics[width=.48\textwidth]{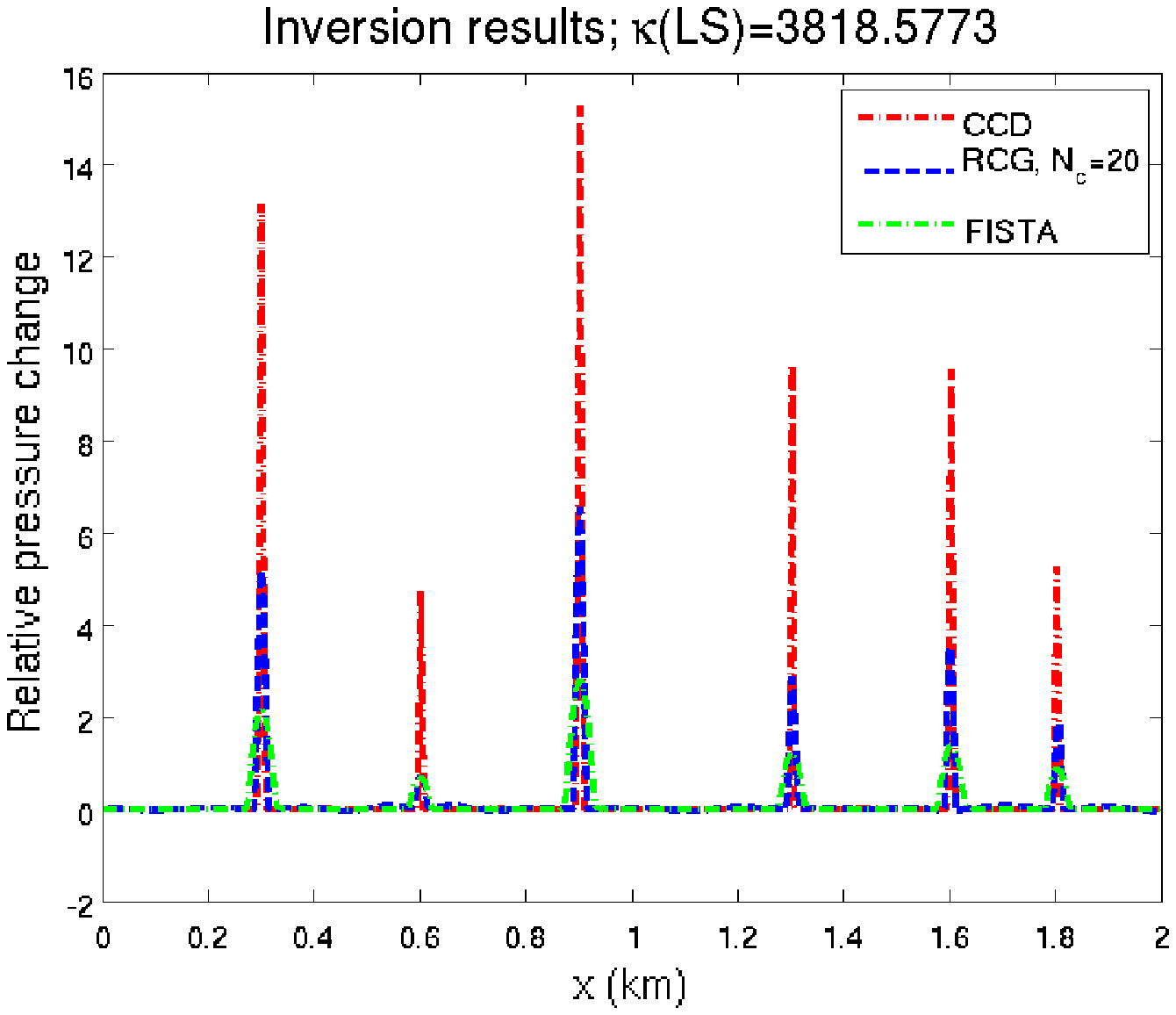}\label{fig:1s01inv}}\\
          \vspace{-.3cm}
          \subfigure[]{\includegraphics[width=.48\textwidth]{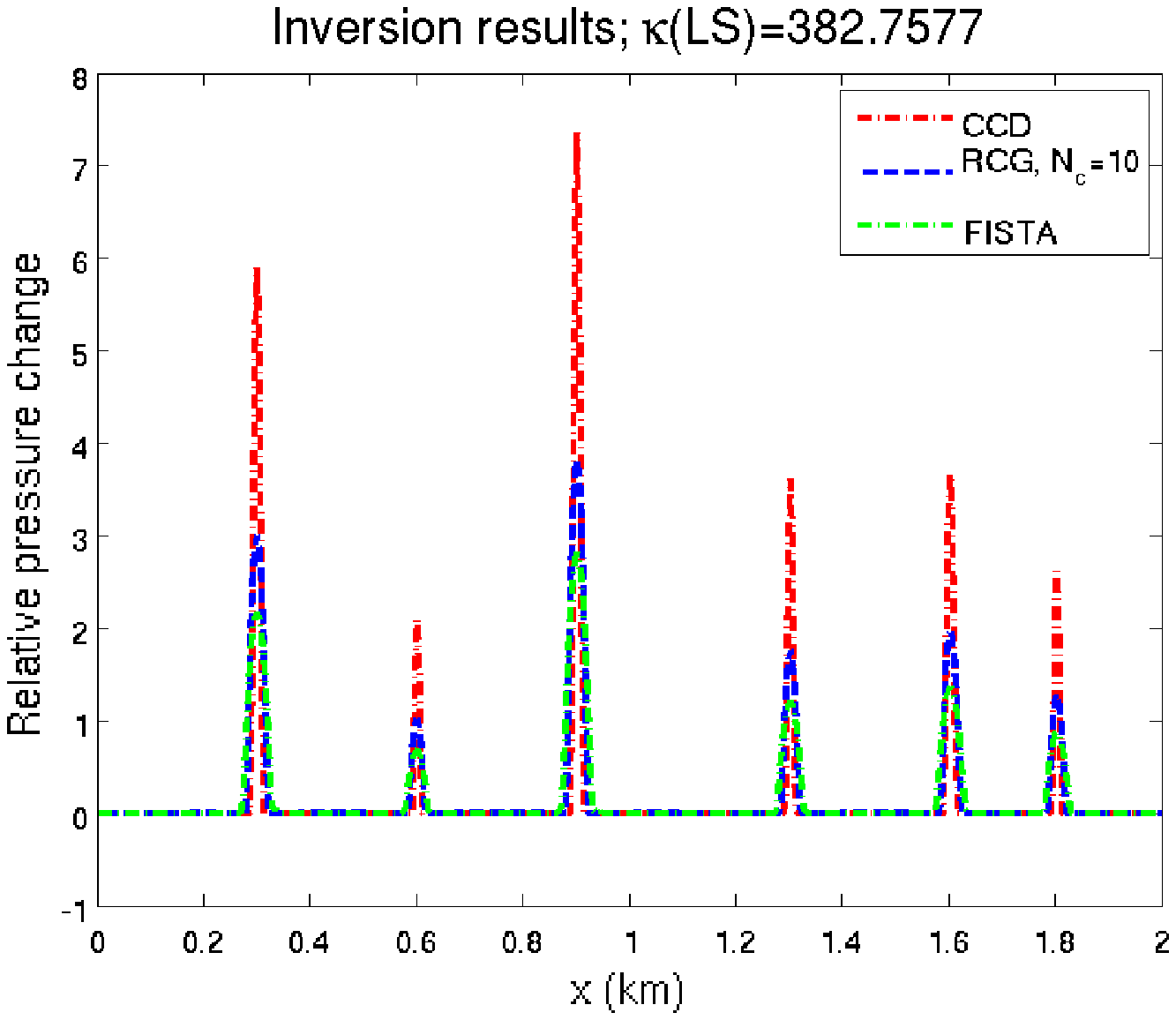}\label{fig:1s1inv}}
          \subfigure[]{\includegraphics[width=.48\textwidth]{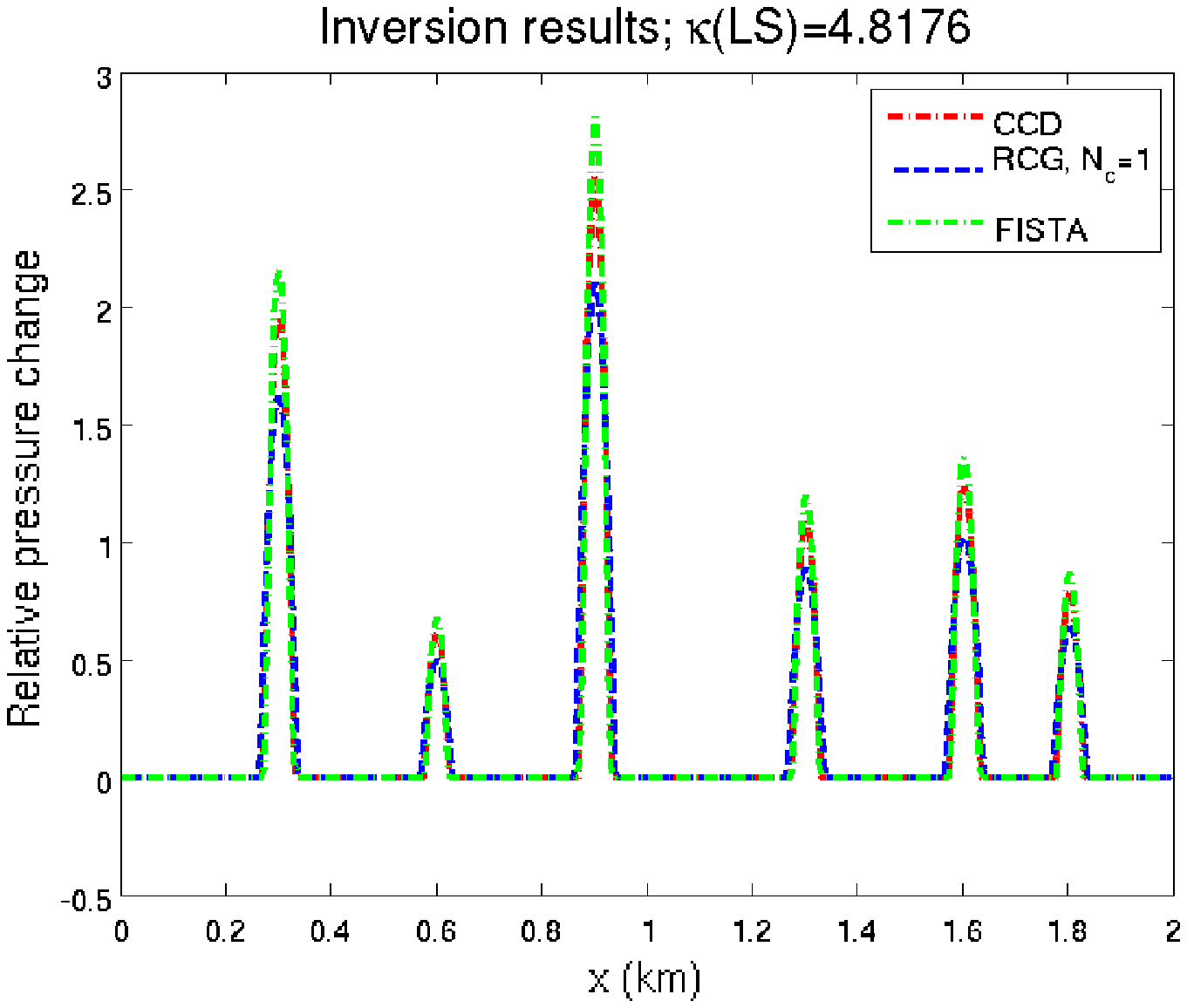}\label{fig:1s100inv}}
          \vspace{-.5cm}
          \caption{  Inversion results for CCD (red), RCG (blue), FISTA (green) after 1000 operator and adjoint applications for (a) $\lambda=.05$; (b) $\lambda=0.1$; (c) $\lambda=1$; (d) $\lambda=100$. Note that FISTA does not use $\lambda$ and the same FISTA results are shown in all plots but using different vertical scales. Compressive Conjugate Directions method still retains its advantage in resolving the spiky model at earlier iterations. \emph{Asymptotically} faster convergence of FISTA kicks in when $\lambda=100$ with a well-conditioned (\ref{eq:interm}), when the ADMM convergence is slowed---compare with Figure~\ref{fig:as100}.}
  \end{center}
\end{figure}
These results indicate that the Compressive Conjugate Directions method achieves qualitative recovery of the spiky model at early iterations. Superiority of the new method is especially pronounced when the intermediate least-squares minimization problem (\ref{eq:interm}) is ill-conditioned (see plot tops). The method retains its advantage after 1000 operator and adjoint applications, as shown in Figures~\ref{fig:1s005inv},\ref{fig:1s01inv},\ref{fig:1s1inv},\ref{fig:1s100inv}. Note that the error plots of the CCD in Figures~\ref{fig:cs005},\ref{fig:cs01},\ref{fig:cs1},\ref{fig:cs100} exhibit a trade-off between the convergence rate and condition number of problem (\ref{eq:interm}) discussed earlier in this subsection~\ref{subs:tradeoff}: a more ill-conditioned (\ref{eq:interm}) is associated with a faster convergence rate of the new method.

\begin{figure}[htb]
  \begin{center}
          \subfigure[]{\includegraphics[width=.48\textwidth]{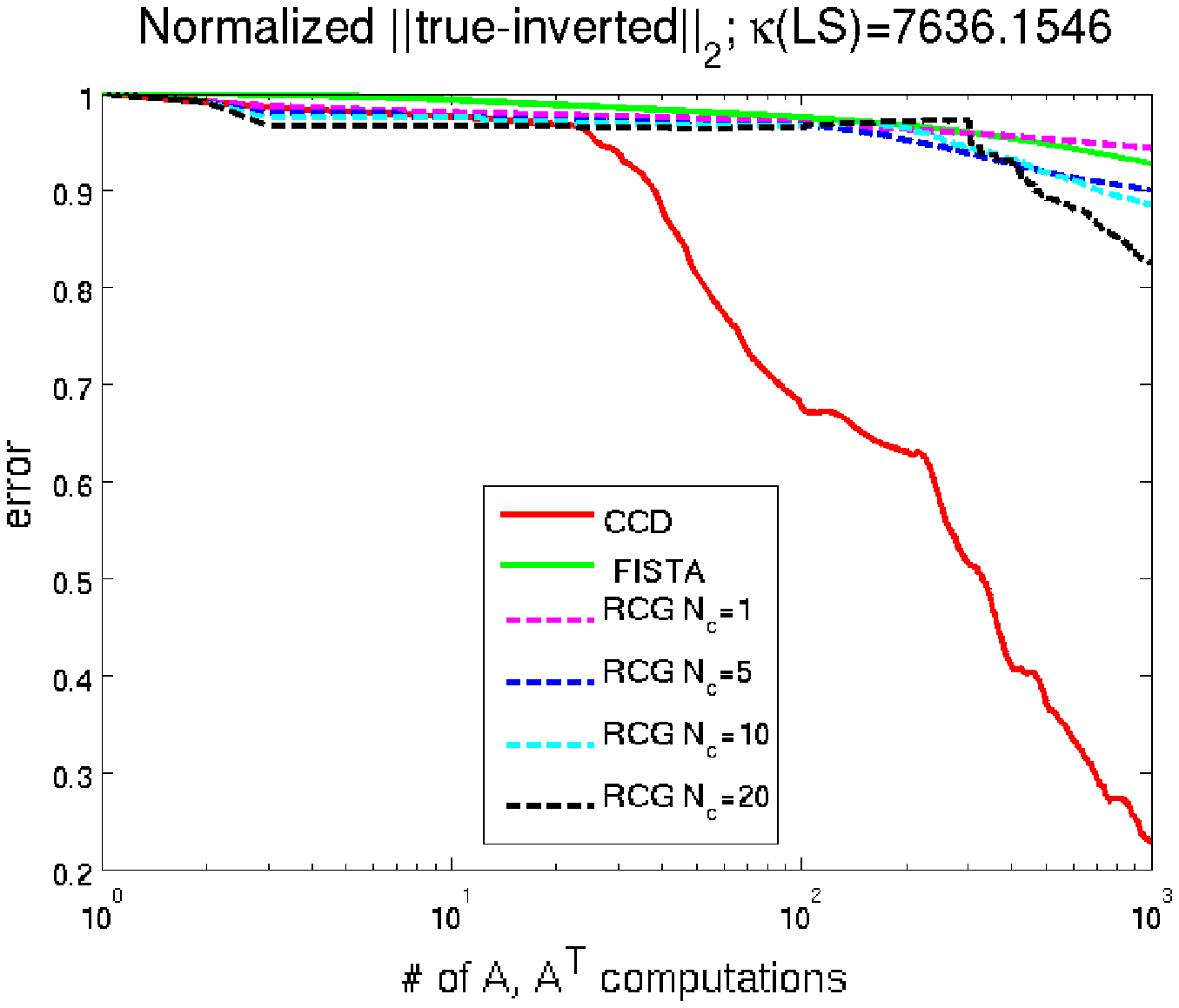}\label{fig:cs005}}
          \subfigure[]{\includegraphics[width=.48\textwidth]{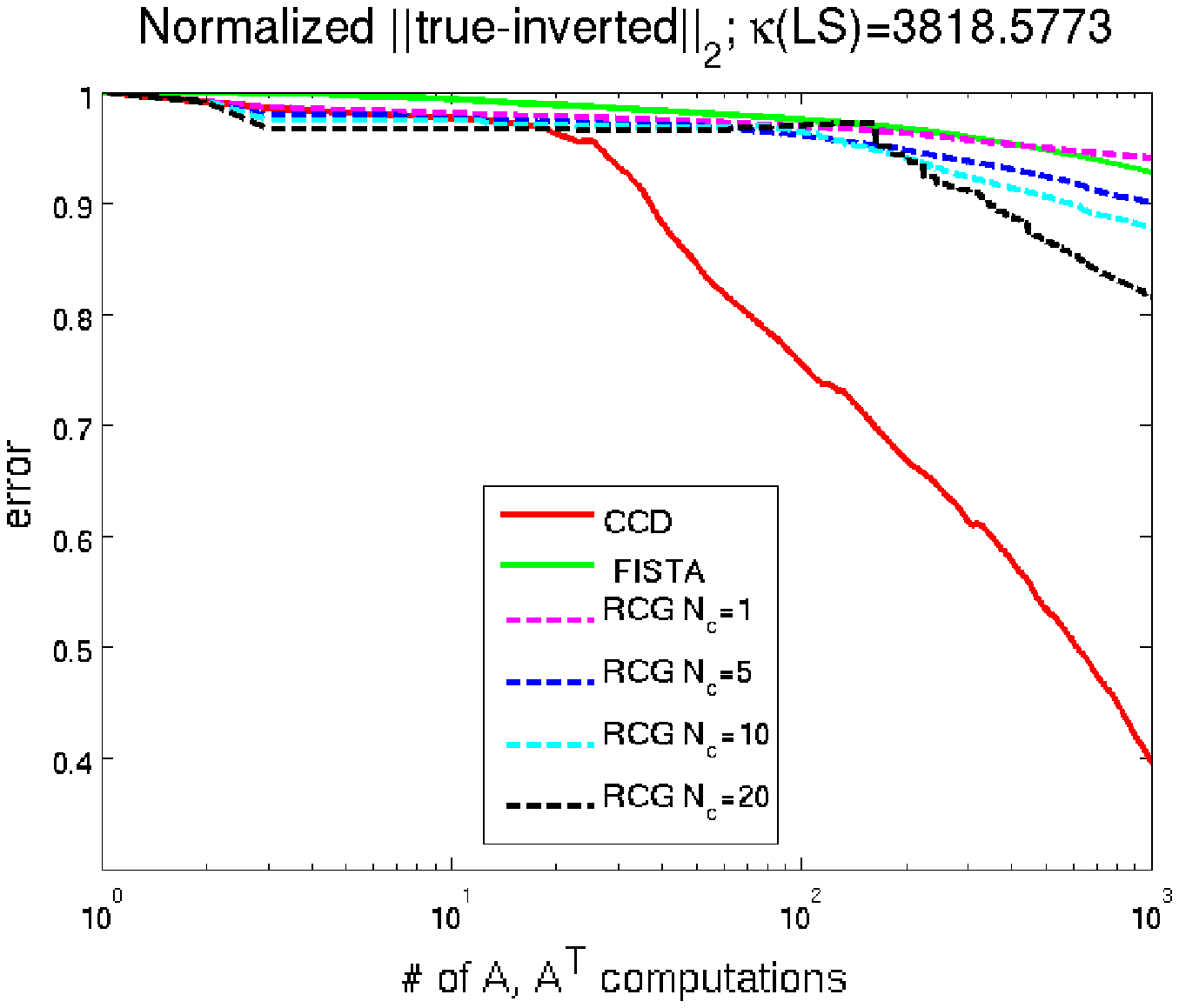}\label{fig:cs01}}\\
          \vspace{-.3cm}
          \subfigure[]{\includegraphics[width=.48\textwidth]{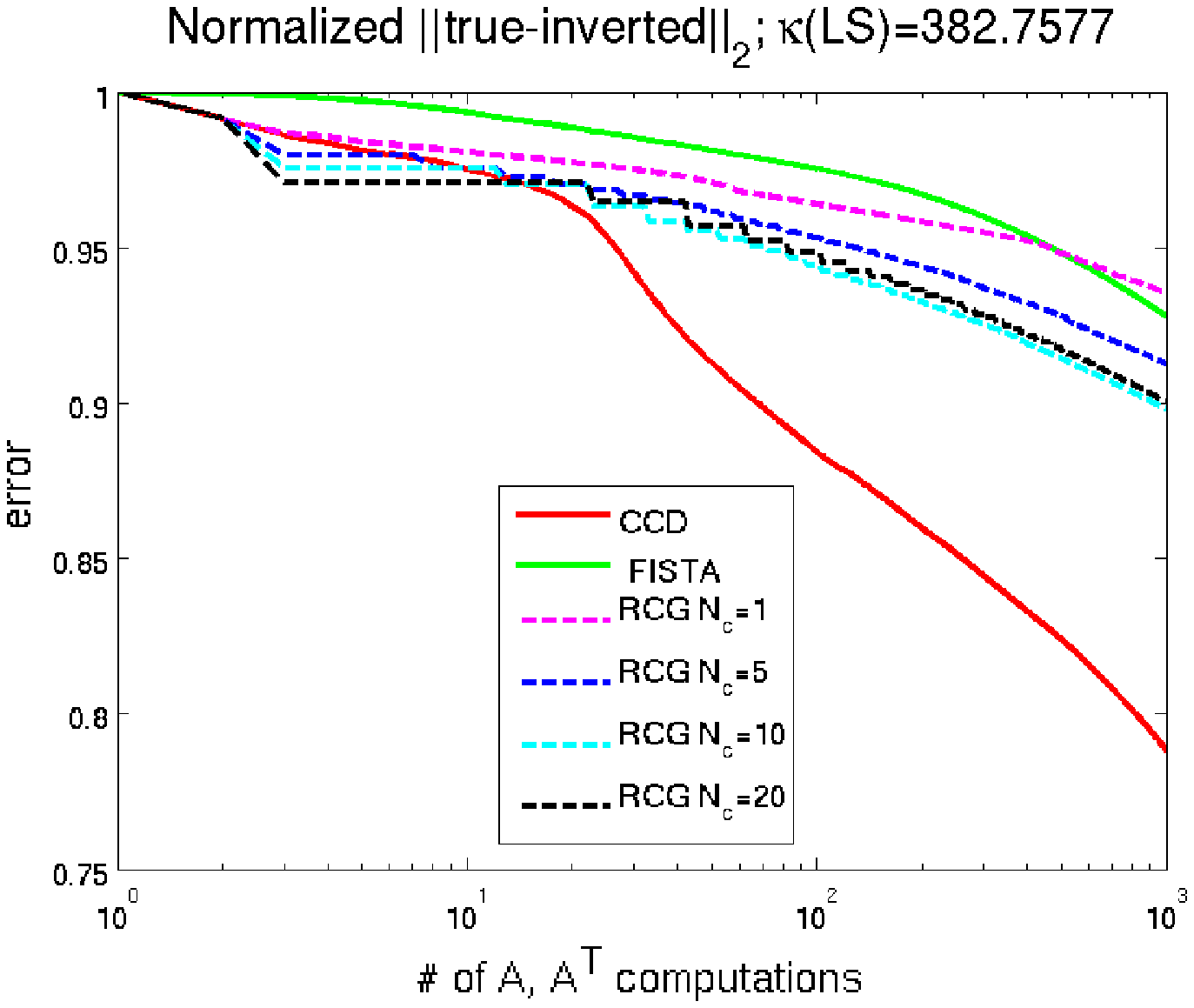}\label{fig:cs1}}
          \subfigure[]{\includegraphics[width=.48\textwidth]{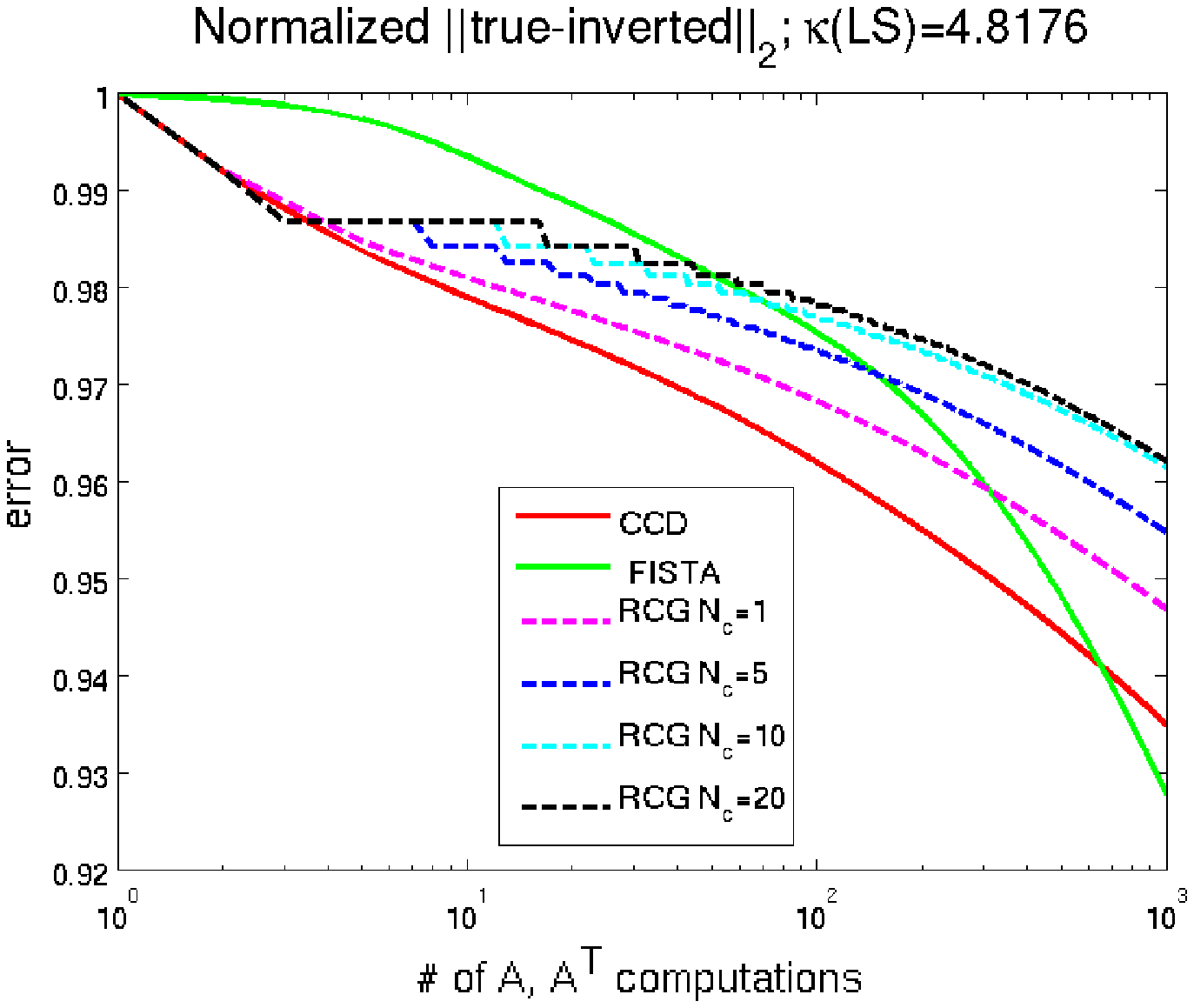}\label{fig:cs100}}
          \vspace{-.5cm}
          \caption{  Convergence curves for CCD (solid red), RCG (dashed), FISTA (solid green) for (a) $\lambda=.05$; (b) $\lambda=0.1$; (c) $\lambda=1$; (d) $\lambda=100$---compare with Figures~\ref{fig:1s005inv},\ref{fig:1s01inv},\ref{fig:1s1inv},\ref{fig:1s100inv}.}
  \end{center}
\end{figure}

Figures~\ref{fig:as005},\ref{fig:as01},\ref{fig:as1},\ref{fig:as100} show error plots for the CCD, ADMM with \emph{exact} minimization of (\ref{eq:interm}), and FISTA. The said trade-off between the convergence rate and condition number of (\ref{eq:interm}) is exhibited by the ADMM. The CCD curves approach the convergence rates of the ADMM once Algorithm~\ref{alg:lmccd} has accumulated enough information about the geometry of the objective function in vectors (\ref{eq:pqm}). Note that the advantage of a faster asymptotic convergence rate of FISTA kicks in only when the ADMM-based methods use values of $\lambda$ that are not optimal for their convergence---see Figures~\ref{fig:cs100} and \ref{fig:as100}. In this case (\ref{eq:interm}) is very well conditioned, and its adequate solution requires only a single step of gradient descent at each iteration of the ADMM, depriving conjugate-gradients-based methods of their advantage. FISTA, being based on accelerating a gradient-descent method, now \emph{asymptotically} beats the convergence rates of the other techniques but this happens too late through the iterations to be of practical significance. In other words, in this particular example FISTA can beat the ADMM (and CCD) only if the latter use badly selected values of $\lambda$. Generalizing this observation about FISTA and ADMM for problem (\ref{eq:opt1}) with a general operator $\mathbf{A}$ goes beyond the scope of our work.

\begin{figure}[htb]
  \begin{center}
          \subfigure[]{\includegraphics[width=.48\textwidth]{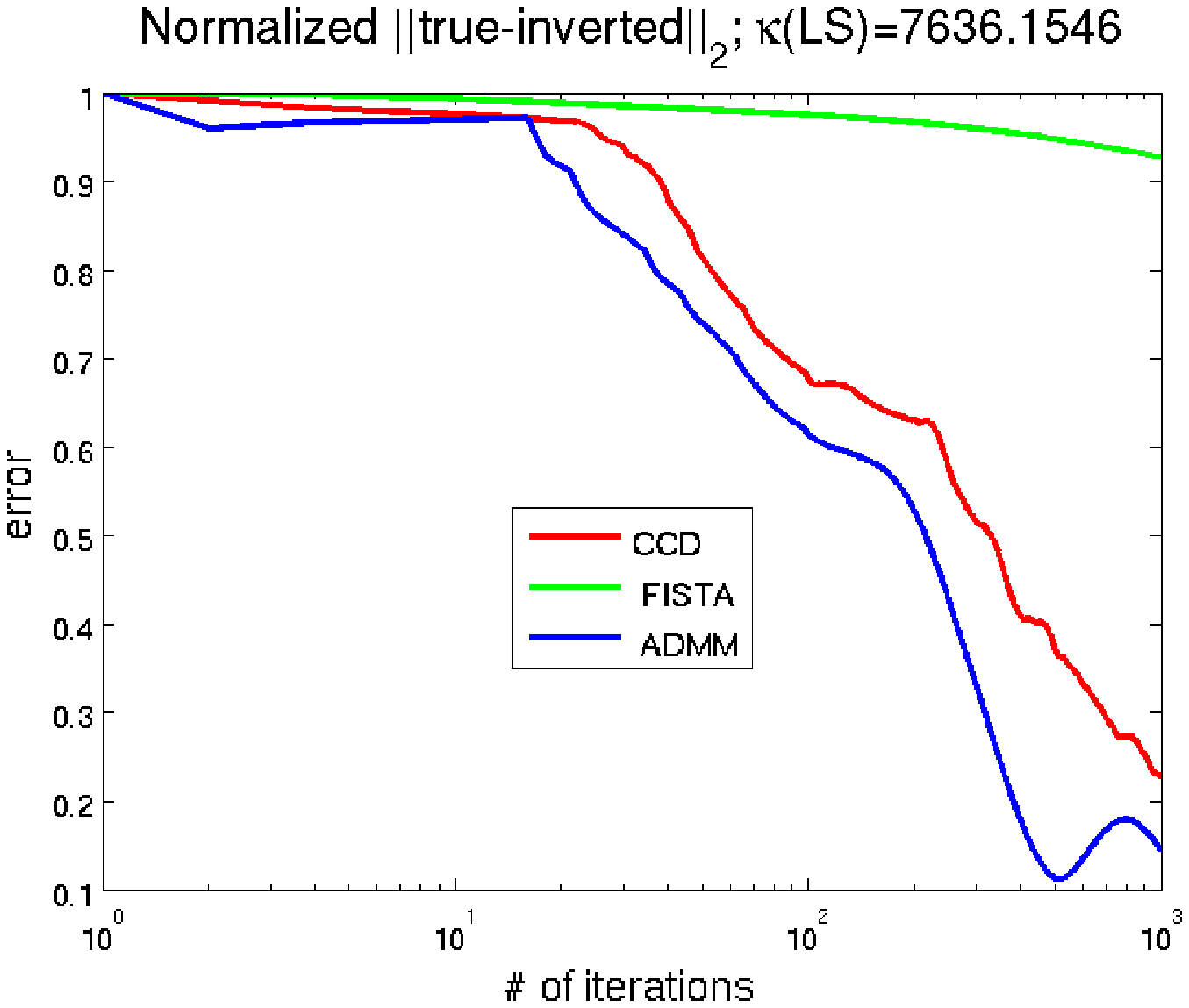}\label{fig:as005}}
          \subfigure[]{\includegraphics[width=.48\textwidth]{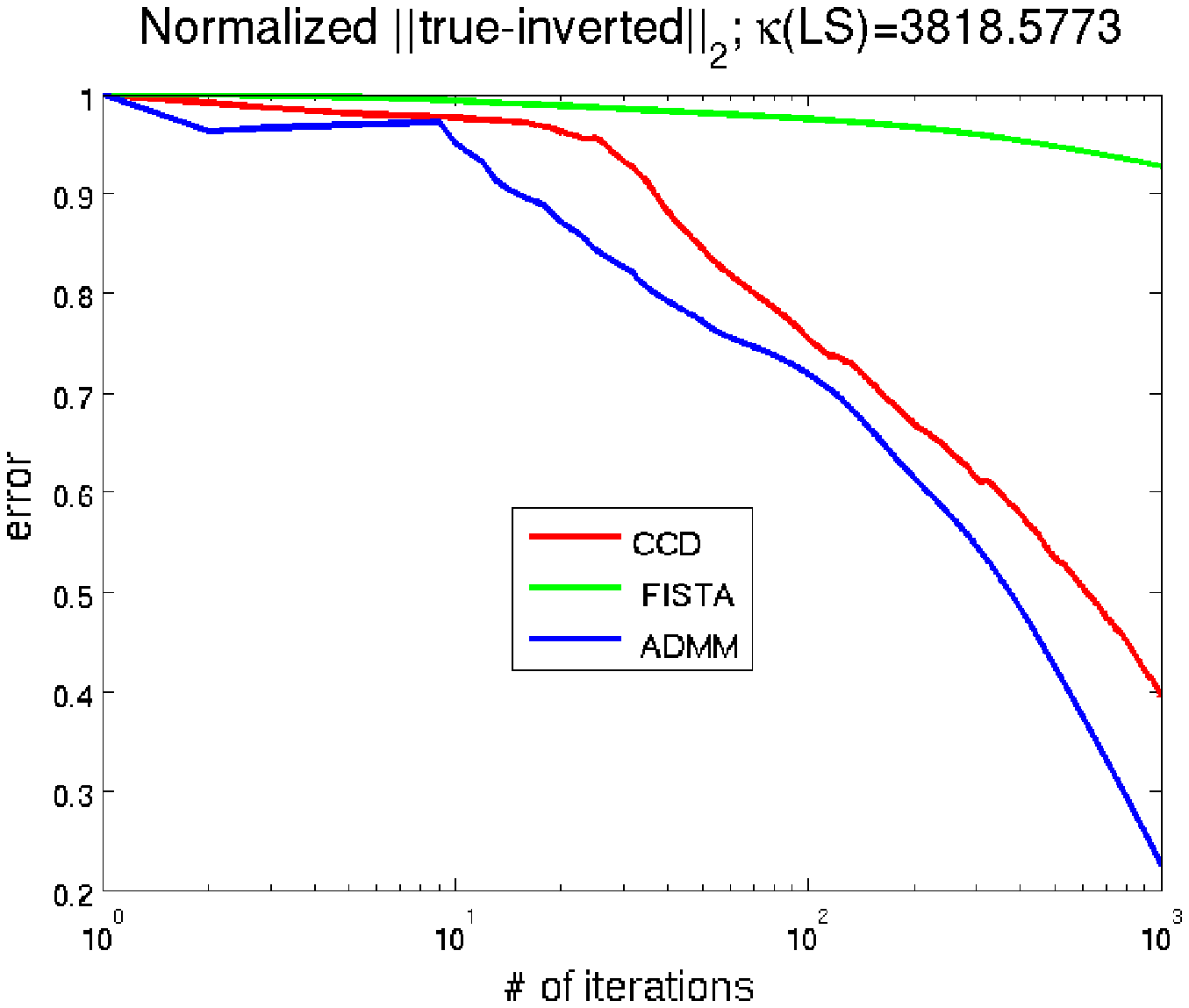}\label{fig:as01}}
          \subfigure[]{\includegraphics[width=.48\textwidth]{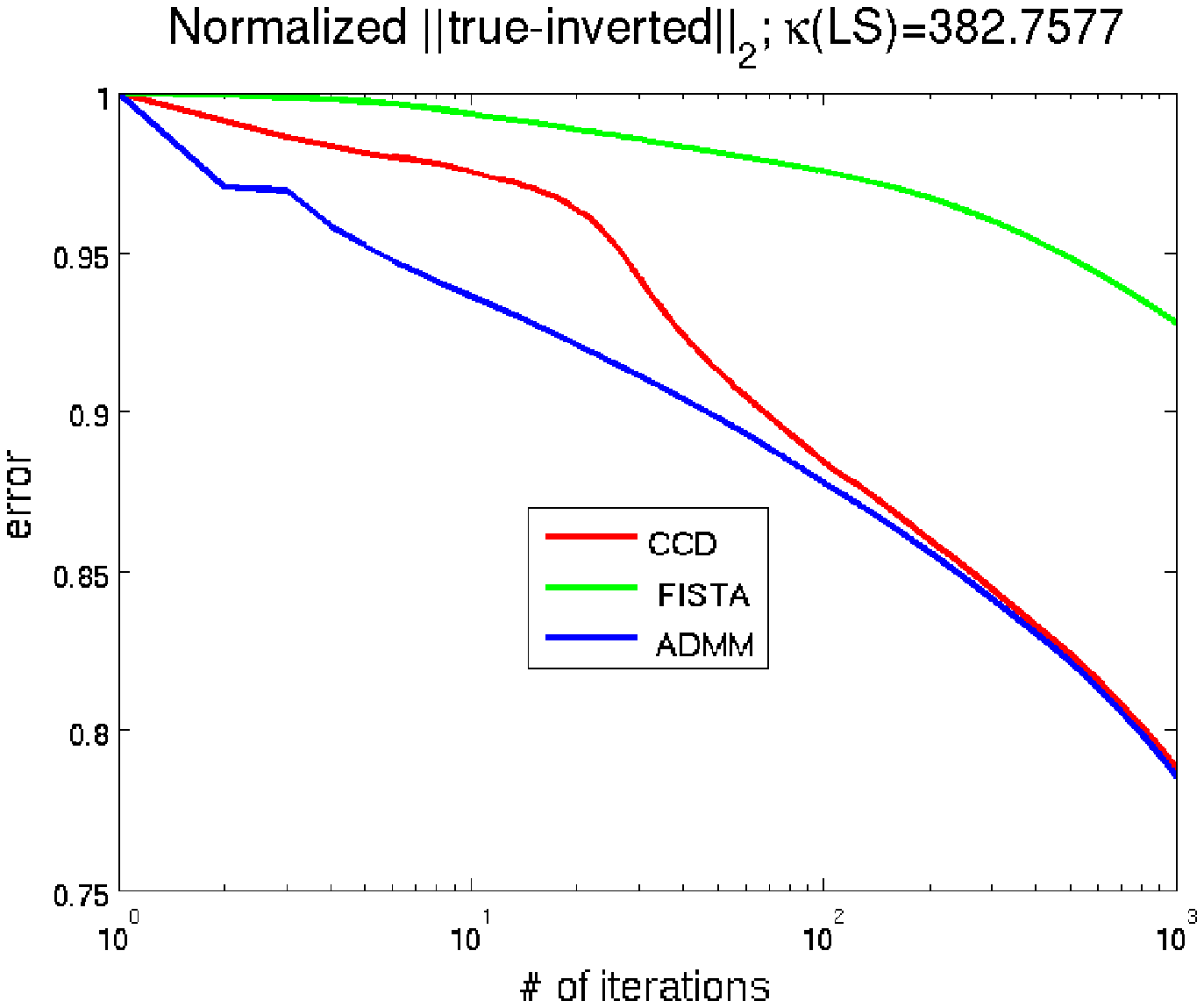}\label{fig:as1}}
          \subfigure[]{\includegraphics[width=.48\textwidth]{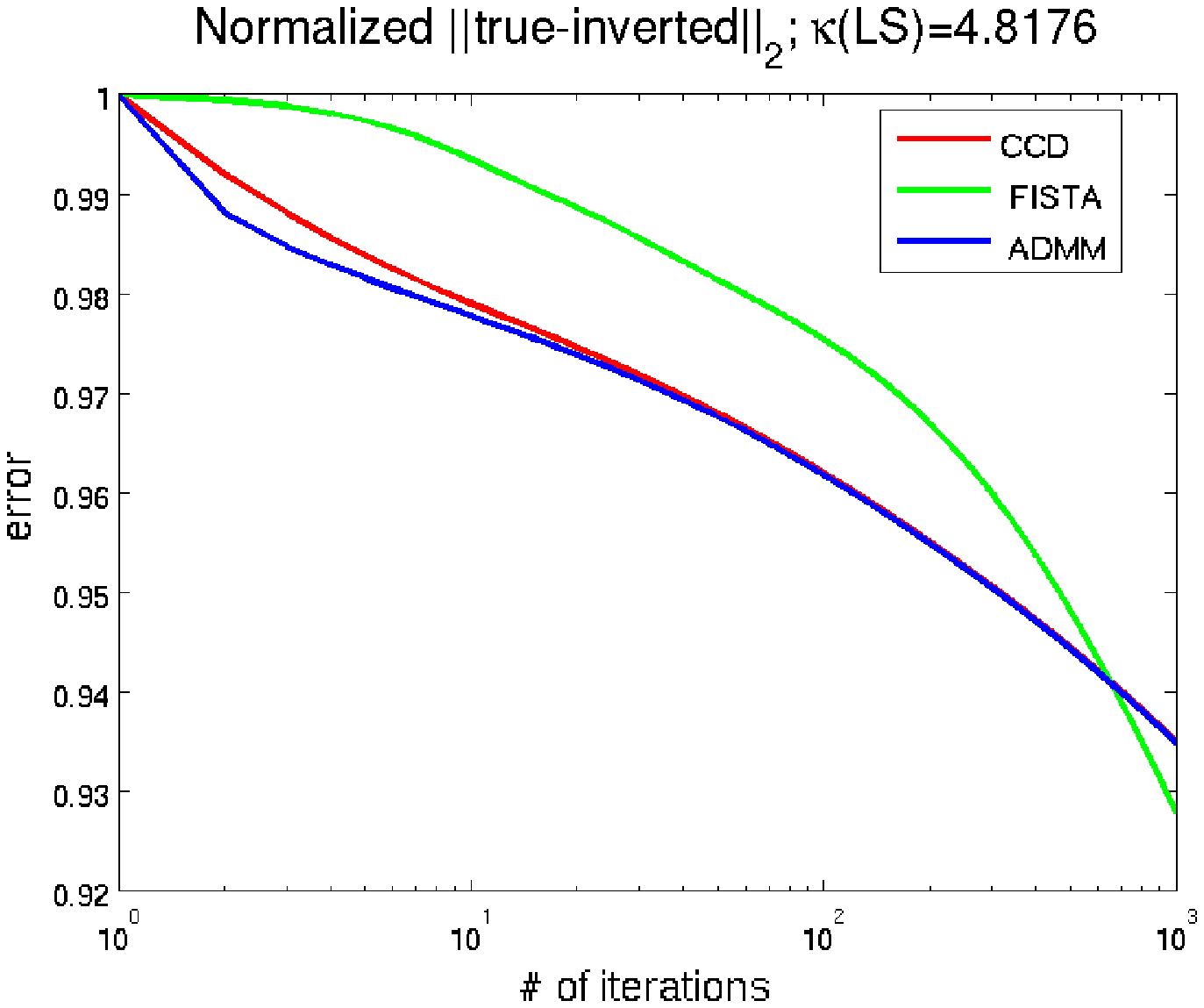}\label{fig:as100}}
          \caption{  Convergence curves for CCD (solid red), ADMM with exact solver (blue), FISTA (green) for (a) $\lambda=.05$; (b) $\lambda=0.1$; (c) $\lambda=1$; (d) $\lambda=100$. Limited-memory Compressive Conjugate Directions with $m=100$ achieves convergence rate comparable to ADMM with exact minimization of (\ref{eq:interm}).}
  \end{center}
\end{figure}

\subsection{Inversion of Pressure Contrasts}
\label{subs:3d}
In this section we apply the Compressive Conjugate Gradients method to identify sharp subsurface pressure contrasts in a reservoir from observations of induced surface displacements. We use a 3-dimensional geomechanical poro\-elastostatic model of pressure-induced deformation based on Biot's theory \cite{SEGDEF}.

We solve a TV-regularized inversion problem (\ref{eq:opt}) with operator $\mathbf{B}$ given by (\ref{eq:aniso}), and operator $\mathbf{A}$ given by extension of (\ref{eq:F})
\begin{equation}
\mathbf{A} \mathbf{u}\;=\; d(x,y),\; d(x,y)\;=\; c\int_0^A \int_0^A \frac{D u(\xi, \eta ) d\xi d\eta  }  {\left(D^2 + (x-\xi)^2 + (y-\eta)^2\right)^{3/2}},
\label{eq:F2}
\end{equation}
where we assume that $\mathbf{u}=u(\xi,\eta), (\xi,\eta)\in [-A,A]\times[-A,A]$ is a relative pore pressure change at a point $(\xi,\eta)$ of the reservoir at a constant depth $D$, $2A$ is the reservoir length and breadth, $\mathbf{d}=d(x,y),(x,y)\in [-A,A]\times[-A,A]$ is the induced vertical displacement at a point $(x,y)$ on the surface, and a constant factor $c$ is determined by the poroelastic medium properties and reservoir thickness.

In this experiment, we discretize the pressure and displacement using a $50\times 50$ grid, with $A=1.2$ km, $D=.455$ km and $c=5.8515\times 10^{3}$, based on a poroelastic model of a real-world unconventional hydrocarbon reservoir \cite{musamark14}. We use a least-squares fitting weight $\alpha=.1$ in (\ref{eq:opt}) to achieve a desirable trade-off between fitting fidelity and blockiness of the inverted pressure change. The blocky model shown in Figure~\ref{fig:btrue} was used to forward-model surface displacements using operator (\ref{eq:F2}). Random Gaussian noise with $\sigma=0.15\%$ of maximum data amplitude, muted below a quarter of the Nyquist wavenumber, was added to the clean data to produce the noisy displacement measurements of Figure~\ref{fig:bdata}.

\begin{figure}[htb]
  \begin{center}
          \subfigure[]{\includegraphics[width=.48\textwidth]{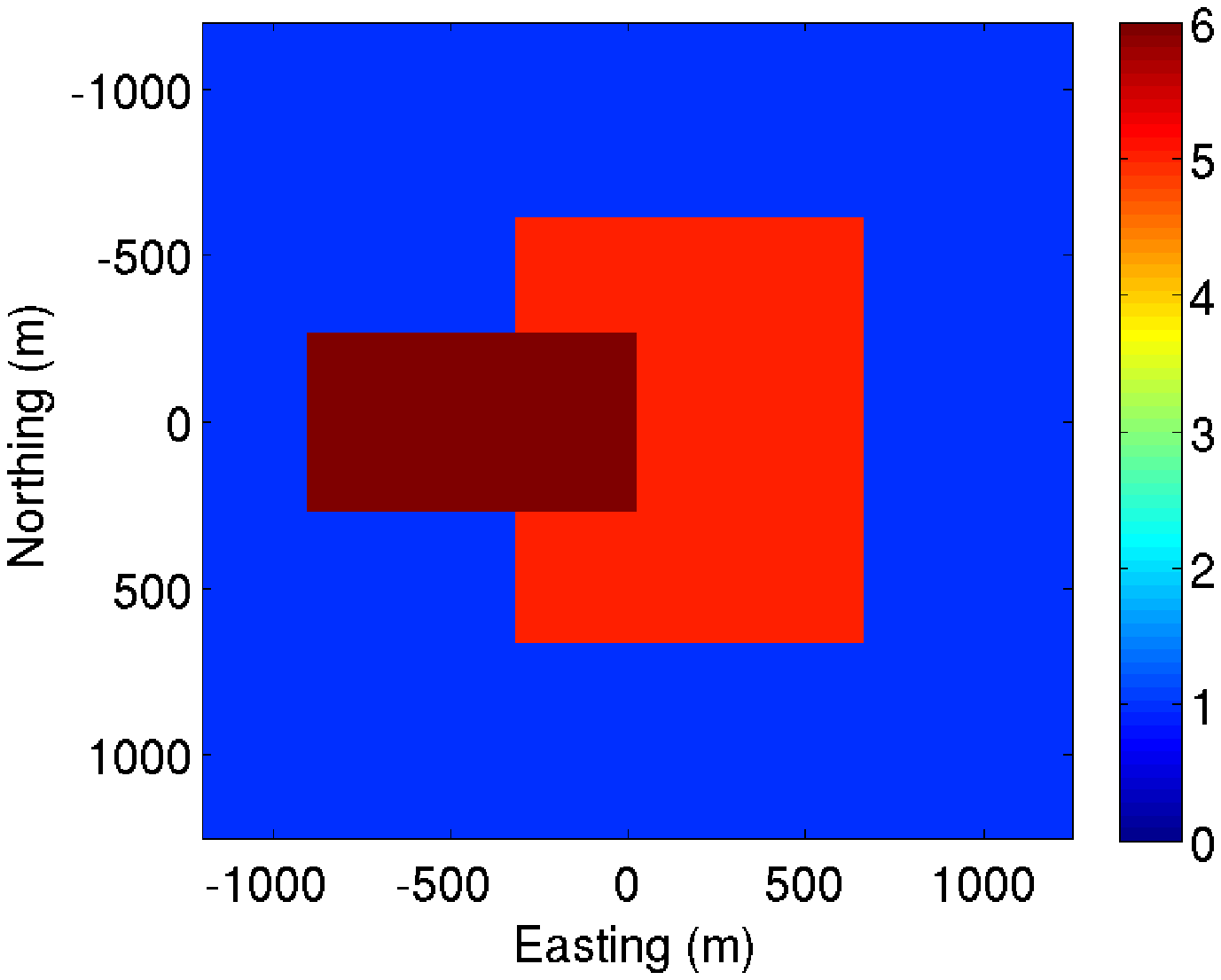}\label{fig:btrue}}
          \quad
          \subfigure[]{\includegraphics[width=.48\textwidth]{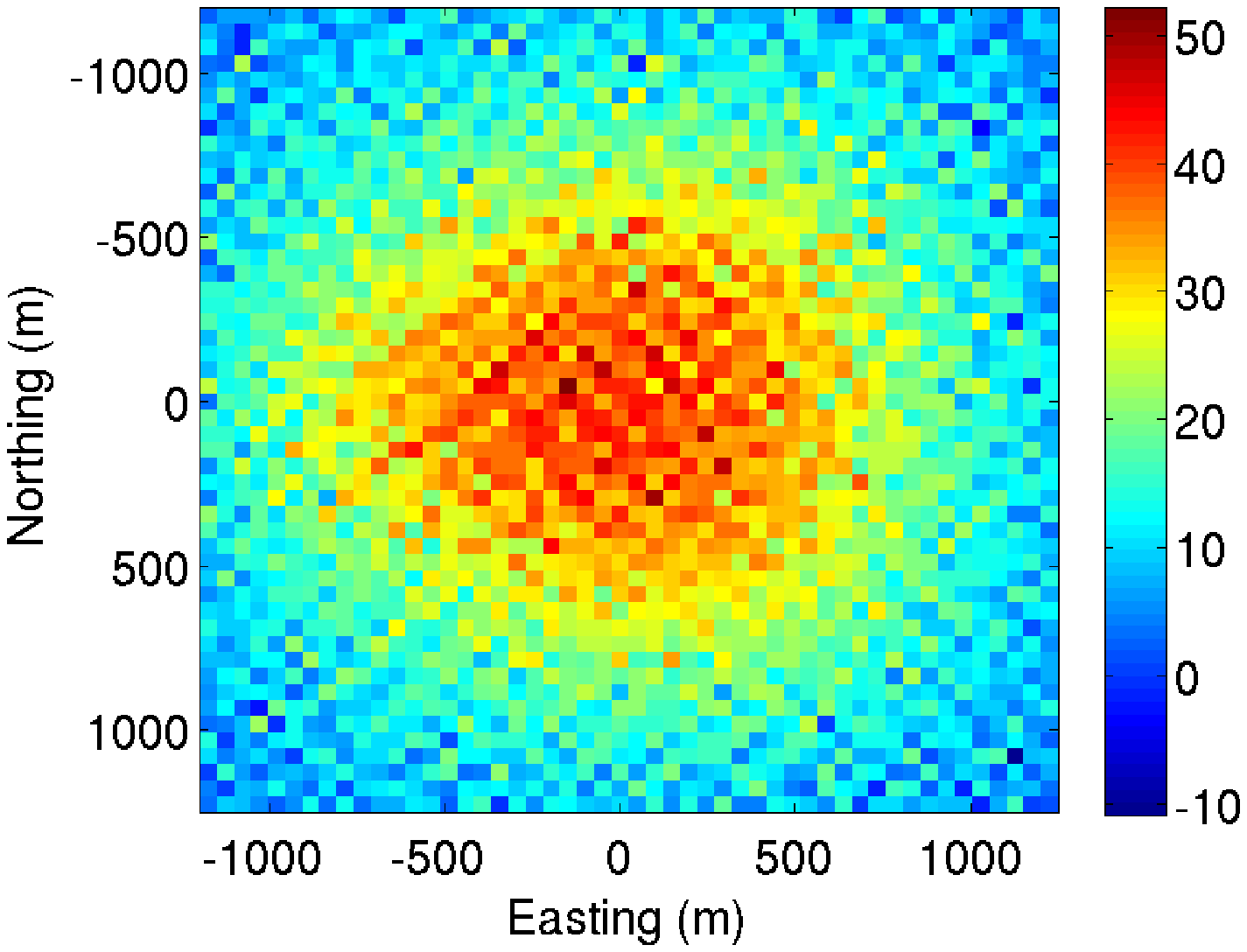}\label{fig:bdata}}
          \vspace{-.5cm}
          \caption{ (a) A blocky true pressure model (MPa); (b) the resulting surface displacements (mm) with added random Gaussian noise with $\sigma=15\%$ of data amplitude.}
  \end{center}
\end{figure}

Figure~\ref{fig:1b10ccd} shows the result of the limited-memory Compressive Conjugate Directions Algorithm~\ref{alg:lmccd} with $m=100$, after a total of 100 combined applications of operator $\mathbf{A}$ and its adjoint. For the same number of operator applications, Figure~\ref{fig:1b10rcg} shows the best result of the ADMM with restarted Conjugate Gradients Algorithm~\ref{alg:rcg}. The corresponding results after 1000 applications of $\mathbf{A}$ and $\mathbf{A}^T$ are shown in Figures~\ref{fig:b10ccd} and \ref{fig:b10rcg}, respectively.

\begin{figure}[htb]
  \begin{center}
          \subfigure[]{\includegraphics[width=.48\textwidth]{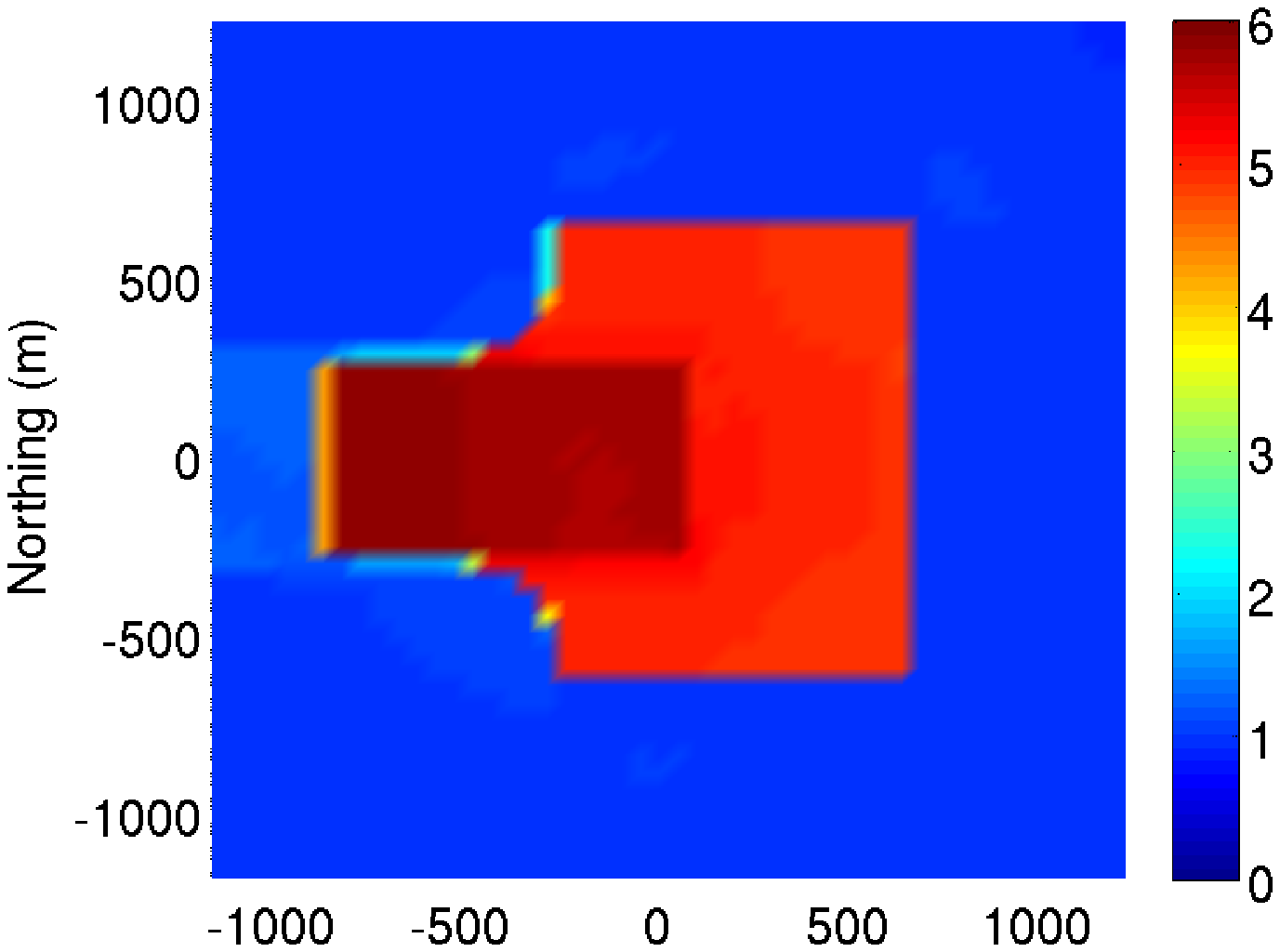}\label{fig:1b10ccd}}
          \subfigure[]{\includegraphics[width=.48\textwidth]{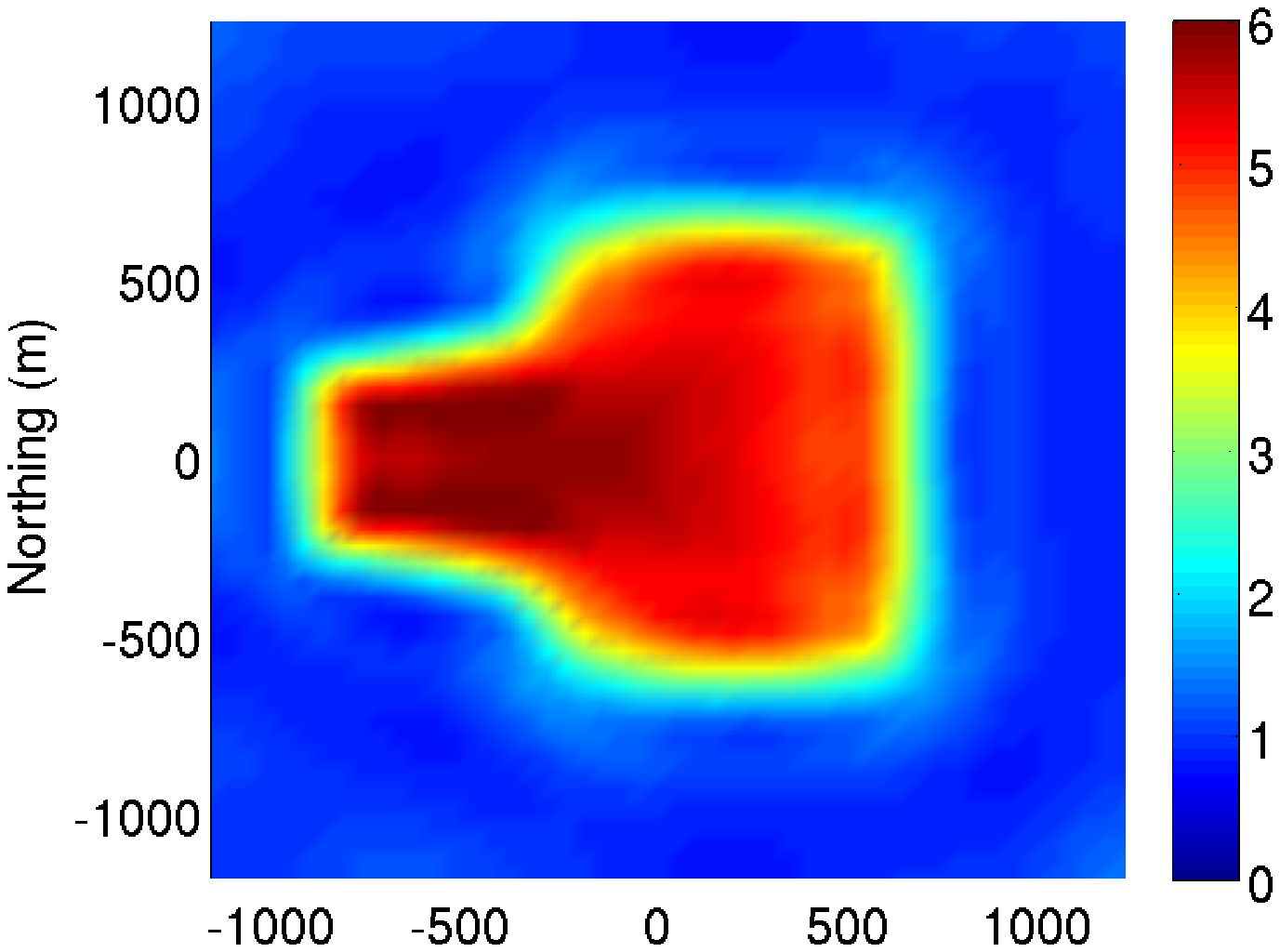}\label{fig:1b10rcg}}\\
          \vspace{-.3cm}
          \subfigure[]{\includegraphics[width=.48\textwidth]{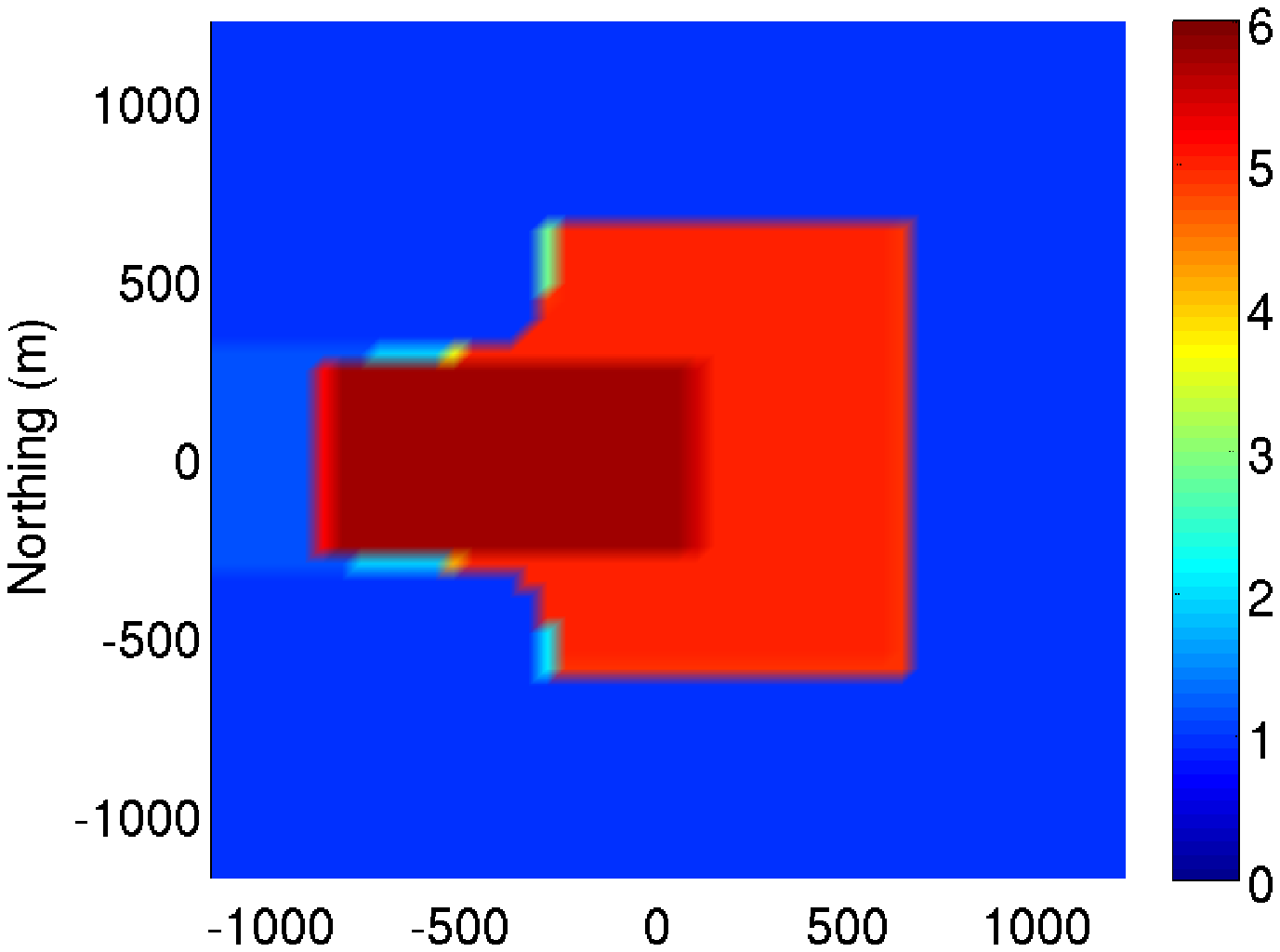}\label{fig:b10ccd}}
          \subfigure[]{\includegraphics[width=.48\textwidth]{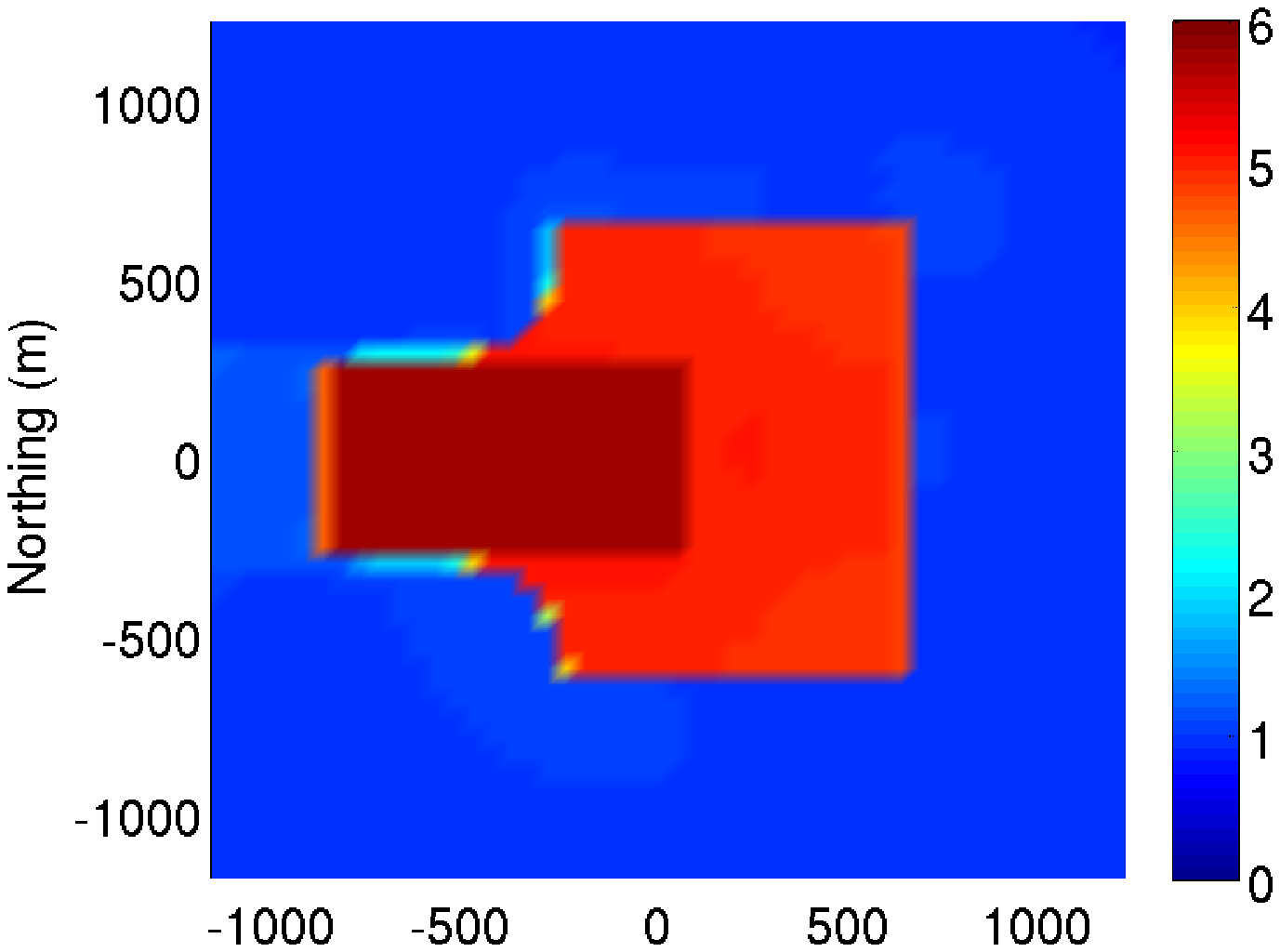}\label{fig:b10rcg}}

          \vspace{-.5cm}
          \caption{Inversion results after (a) 100 iterations (operator and adjoint applications) of CCD with $\lambda=10$; (b) 100 iterations of RCG with $\lambda=10$; (c) 1000 iterations of CCD with $\lambda=10$; (d) 1000 iterations of RCG with $\lambda=10$. In all tests, CCD is the limited-memory Compressive Conjugate Directions method of Algorithm~\ref{alg:lmccd}; RCG is ADMM with restarted Conjugate Gradients of Algorithm~\ref{alg:rcg} showing the most accurate model reconstruction among the outputs for different $N_c$--see Figures~\ref{fig:b1},\ref{fig:b10},\ref{fig:b50},\ref{fig:b100}.}
  \end{center}
\end{figure}

The Compressive Conjugate Directions method resolves key model features faster than the ADMM using iterative solution of (\ref{eq:interm}) restarted at each ADMM iteration. This advantage of our method is particularly pronounced when the intermediate least-squares problem (\ref{eq:interm}) is ill-conditioned---compare Figures~\ref{fig:b1},\ref{fig:b10} with Figures~\ref{fig:b50},\ref{fig:b100}. To accurately resolve the blocky pressure model of Figure~\ref{fig:btrue}, the Compressive Conjugate Directions technique requires about a tenth of operator $\mathbf{A}$ and adjoint applications compared with Algorithm~\ref{alg:rcg} when (\ref{eq:interm}) is poorly conditioned. And again, as in the previous example, there is a trade-off between the convergence rate of the Compressive Conjugate Directions and the condition number of (\ref{eq:interm}): values of $\lambda$ that result in more poorly-conditioned (\ref{eq:interm}) yield the fastest convergence. 

\begin{figure}[htb]
  \begin{center}
          \subfigure[]{\includegraphics[width=.48\textwidth]{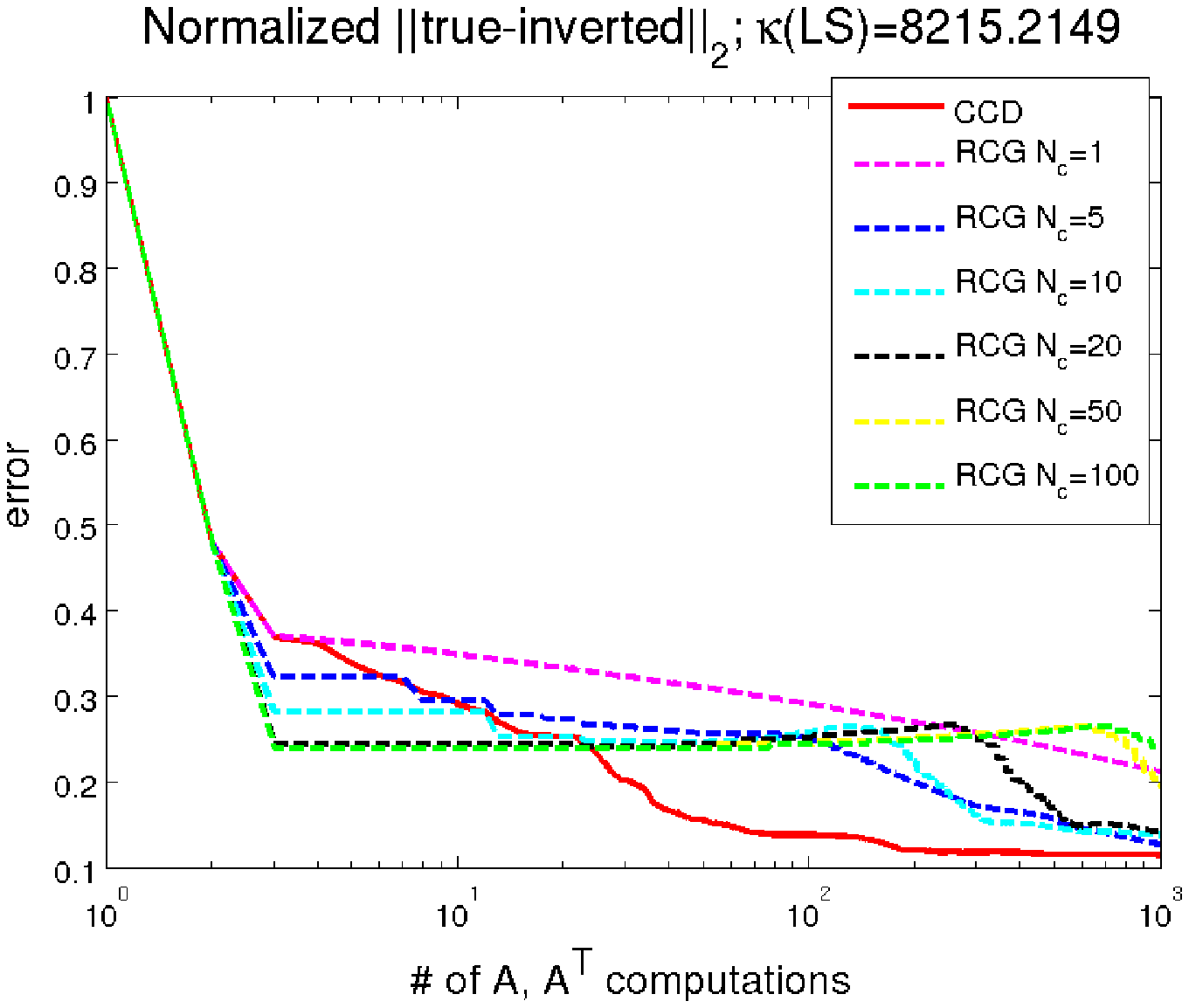}\label{fig:b1}}
          \subfigure[]{\includegraphics[width=.48\textwidth]{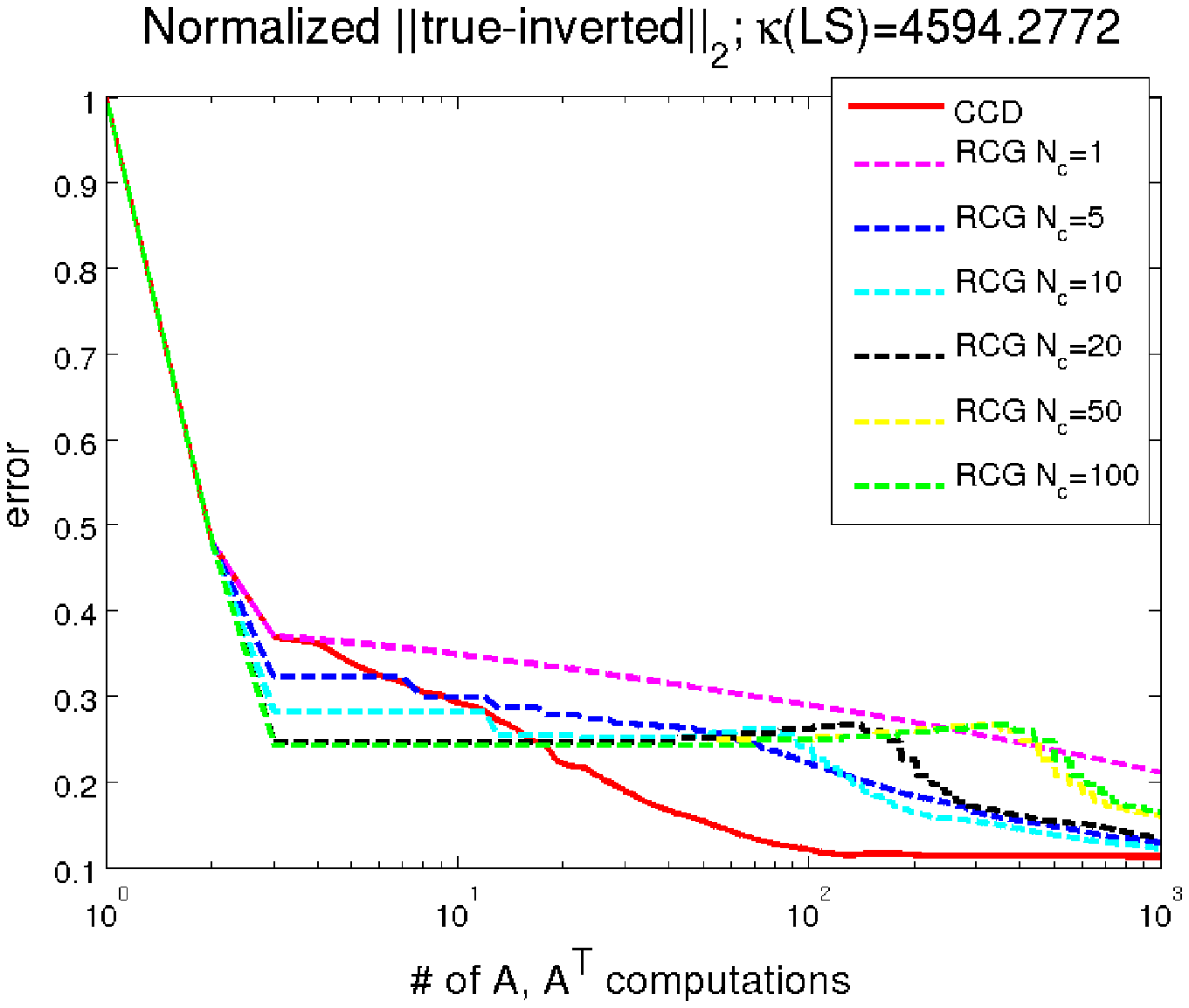}\label{fig:b10}}\\
          \vspace{-.3cm}
          \subfigure[]{\includegraphics[width=.48\textwidth]{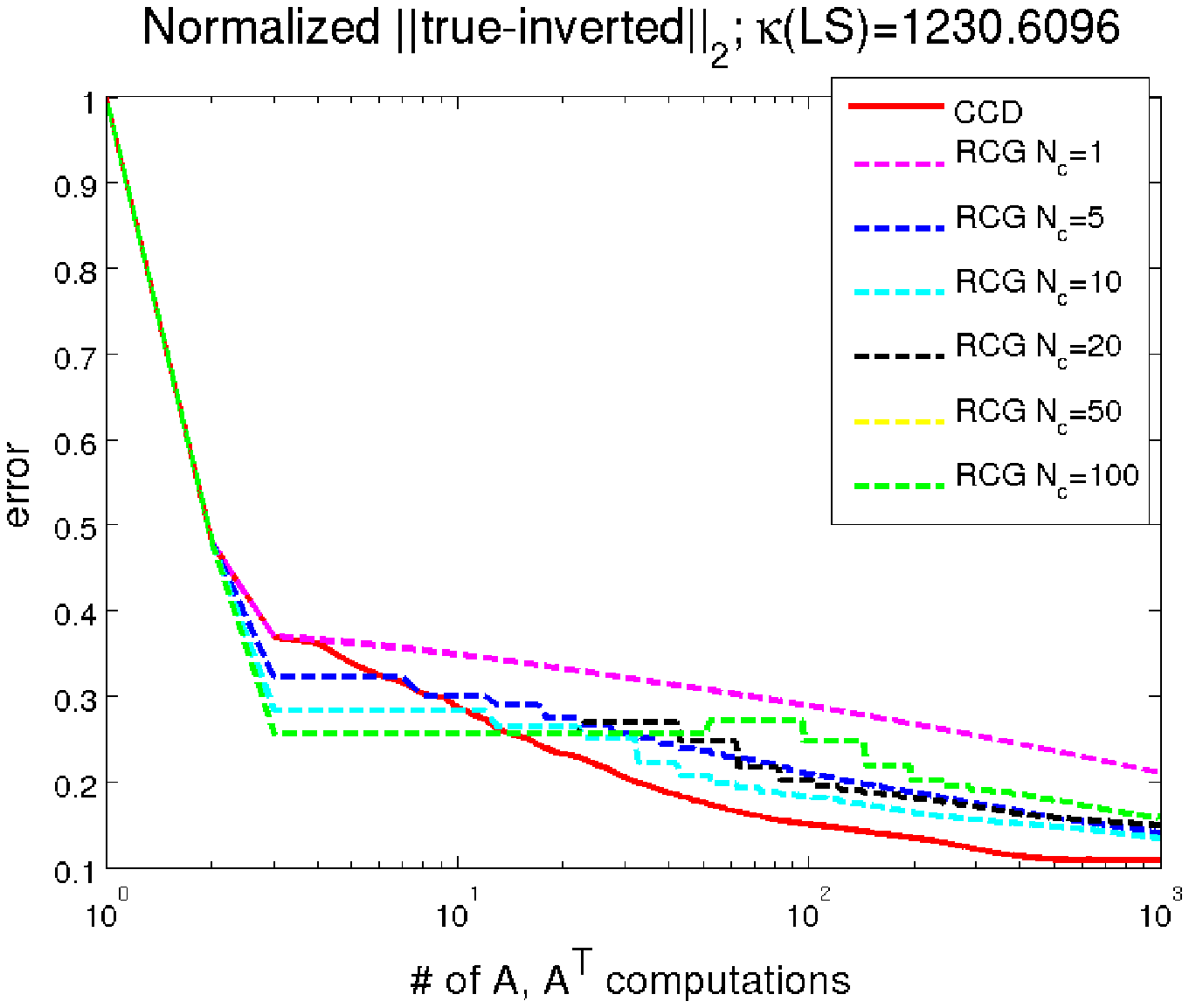}\label{fig:b50}}
          \subfigure[]{\includegraphics[width=.48\textwidth]{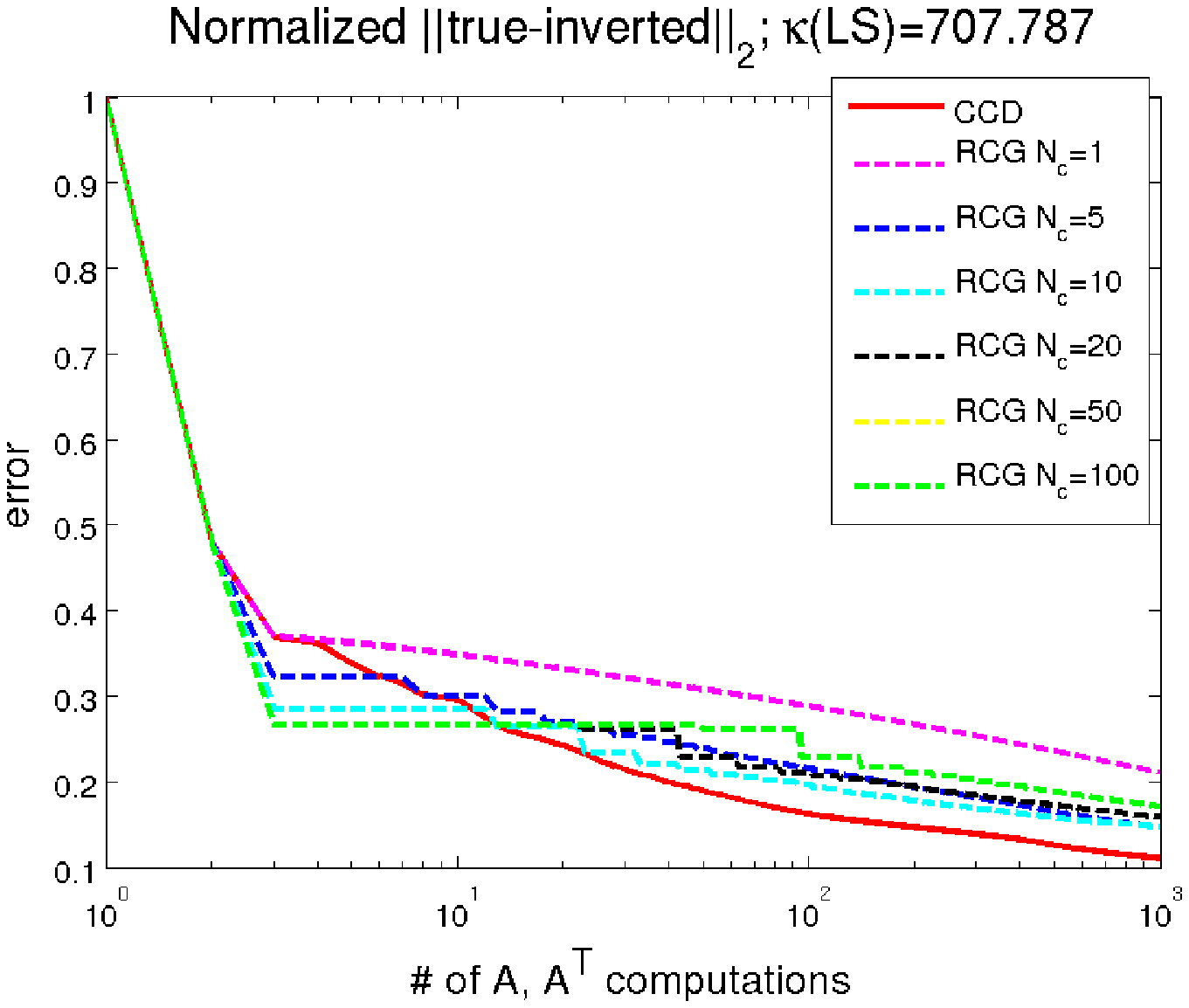}\label{fig:b100}}
          \vspace{-.5cm}
          \caption{ Convergence rates for CCD and RCG with various $N_c$ for (a) $\lambda=5$; (b) $\lambda=10$; (c) $\lambda=50$; (d) $\lambda=100$.}
  \end{center}
\end{figure}

\section{Discussion}

Compressive Conjugate Directions provides an efficient implementation of the Alternating Direction Method of Multipliers in $L_1-TV$ regularized inversion problems (\ref{eq:opt0}) with computationally expensive operators $\mathbf{A}$. By accumulating and reusing information on the geometry of the intermediate quadratic objective function (\ref{eq:interm}), the method requires only one application of the operator $\mathbf{A}$ and its adjoint per ADMM iteration while achieving accuracy comparable to that of the ADMM with exact minimization of (\ref{eq:interm}). The method does not improve the worst-case asymptotic convergence rate of the ADMM. However, it can be used for fast recovery of spiky or blocky solution components. The method trades the computational cost of applying operator  $\mathbf{A}$ and its adjoint for extra memory required to store previous conjugate direction vectors (\ref{eq:pqm}).   

Our numerical experiments involving problems of geomechanical inversion demonstrated a trade-off between the number of ADMM iterations required to achieve a sufficiently accurate solution approximation, and condition number of the intermediate least-squares problem (\ref{eq:interm}). Understanding the extent to which this phenomenon applies to solving (\ref{eq:opt0}) with other classes of modeling operators $\mathbf{A}$ requires further analysis.

\subsection{Generalizations}
The primary focus of this work are $L_1-TV$ regularized inversion problems (\ref{eq:opt0}). However, the Steered Conjugate Directions Algorithm~\ref{alg:scd} can be combined with the Method of Multipliers to solve more general problems of large-scale equality-constrained optimization.

For example, consider the problem 
\begin{equation}
\begin{aligned}
        & \|\mathbf{A}\mathbf{u}-\mathbf{d} \|^2_2 \;\rightarrow\; \min,\\
        & \mathbf{B}\mathbf{u} - \mathbf{c} \; = \; \boldsymbol 0,\\
       &\mathbf{u}\; \in \; \mathbb{R}^N,\; \mathbf{d}\;\in\;\mathbb{R}^M,\; \mathbf{A}:\mathbb{R}^N\to\mathbb{R}^M,\;
        \mathbf{B}:\mathbb{R}^N\to\mathbb{R}^K,\;
\end{aligned}
\label{eq:copt0}
\end{equation}
where $\mathbf{A}$ is a computationally expensive operator. Many ``coupled'' systems governing two or more physical parameters can be described mathematically as a constrained problem (\ref{eq:copt0}). Of special interest are cases when $K\; \ll \; \min \left\{N, M\right\}$---e.g., large-scale optimization problems with a localized constraint. Applying the Augmented Lagrangian Method of Multipliers to (\ref{eq:copt0}), after re-scaling the multiplier vector, we get
\begin{equation}
\begin{aligned}
        \mathbf{u}_{k+1}\;=\; & \mathrm{argmin}\, \| \mathbf{A}\mathbf{u} - \mathbf{d}\|_2^2 + \frac{\lambda}{2}\|\mathbf{c}-\mathbf{B}\mathbf{u} + \mathbf{b}_{k}\|^2_2,\\
        \mathbf{b}_{k+1}\;=\; & \mathbf{b}_k \; + \; \mathbf{c}-\mathbf{B}\mathbf{u}_{k+1}.
\end{aligned}
        \label{eq:mult1}
\end{equation}
As before, the minimization on the first line of (\ref{eq:mult1}) is equivalent to solving a system of normal equations with a fixed left-hand side and changing right-hand sides. Combining the dual-variable updates from (\ref{eq:mult1}) with Algorithm~\ref{alg:scd}, we get Algorithm~\ref{alg:scdmm}.

\begin{algorithm}[htl]
        \caption{Steered Conjugate Directions + Method of Multipliers for solving (\ref{eq:copt0})}
 \label{alg:scdmm}
\begin{algorithmic}[1]
\State $\mathbf{u}_{0} \;\gets\; \boldsymbol 0^N,\;\mathbf{b}_0\;\gets\; \boldsymbol 0^K,\; \mathbf{v}_0\;\gets\;\begin{bmatrix}
                \mathbf{d}\\
                \sqrt{\lambda}\left(\mathbf{c}+\mathbf{b}_0\right)
\end{bmatrix}$
\State {$\mathbf{p}_0\;\gets\; \mathbf{F}^T \mathbf{v}_0,\;\mathbf{q}_0\;\gets\;\mathbf{F}\mathbf{p}_0,\;\delta_0\;\gets\; \mathbf{q}^T_0\mathbf{q}_0$}
\For {$k=0,1,2,3,\ldots$}
    \For {$i=0,1,\ldots,k$}
        \State {$\tau_i\;\gets\; {\mathbf{q}_i^T \mathbf{v}_{k}}/\delta_i$}
    \EndFor
    \State {$\mathbf{u}_{k+1}\;\gets\;\sum_{i=0}^{k} \tau_i \mathbf{p}_i$}
    \State {$\mathbf{b}_{k+1}\;\gets\; \mathbf{b}_k +\mathbf{c} - \mathbf{B} \mathbf{u}_{k+1}$}
    \State {$\mathbf{v}_{k+1}\;\gets\;\begin{bmatrix}
                \mathbf{d}\\
                \sqrt{\lambda}\left(\mathbf{c}+\mathbf{b}_{k+1}\right)
\end{bmatrix}$}
    \State{$\mathbf{r}_{k+1}\;\gets\;\mathbf{v}_{k+1}\;-\; \sum_{i=0}^{k}\tau_i  \mathbf{q}_i$}
        
    \State {$\mathbf{w}_{k+1}\;\gets\;  \mathbf{F}^T \mathbf{r}_{k+1}$}
    \State {$\mathbf{s}_{k+1}\;\gets\;  \mathbf{F} \mathbf{w}_{k+1}$}
    \For {$i=0,1,\ldots,k$}
        \State {$\beta_i\;\gets\;  - {\mathbf{q}_i^T \mathbf{s}_{k+1}}/\delta_i$}
    \EndFor
    \State {$\mathbf{p}_{k+1}\;\gets\;\sum_{i=0}^{k} \beta_i \mathbf{p}_i \; + \; \mathbf{w}_{k+1}$}
    \State {$\mathbf{q}_{k+1}\;\gets\;\sum_{i=0}^{k} \beta_i \mathbf{q}_i \; + \; \mathbf{s}_{k+1}$}
    \State {$\delta_{k+1} \;\gets\;\mathbf{q}_{k+1}^T \mathbf{q}_{k+1}$}
    \If {$\delta_{k+1} = 0$}
    \Comment {Use condition ``$\delta_{k+1}  <  \text{tolerance}$'' in practice}
         \State {$\delta_{k+1}\;\gets\;1,\; \mathbf{p}_{k+1}\;\gets\;\mathbf{0}^N,\; \mathbf{q}_{k+1}\;\gets\;\mathbf{0}^{M+K}$}
    \EndIf
    \State {Exit loop if ${\|\mathbf{u}_{k+1}-\mathbf{u}_k\|_2}/{\| \mathbf{u}_k \|_2} \; \le \; \text{target accuracy}$}
\EndFor
\end{algorithmic}
\end{algorithm}
Operator $\mathbf{F}$ in Algorithm~\ref{alg:scdmm} is given by (\ref{eq:Fop}) with $\alpha=1$. A limited-memory version of Algorithm~\ref{alg:scdmm} is obtained trivially by adapting Algorithm~\ref{alg:lmccd}. We envisage potential utility of Algorithm~\ref{alg:scdmm} in applications where storing a set of previous conjugate direction vectors (\ref{eq:pqm}) is computationally more efficient that iteratively solving the quadratic minimization problem in (\ref{eq:mult1}) from scratch at each iteration of the method of multipliers.

The Compressive Conjugate Directions Algorithm~\ref{alg:lmccd} can be extended for solving non-linear inversion problems with $L_1$ and \emph{isotropic} total-variation regularization. Likewise, the Steered Conjugate Directions Algorithm~\ref{alg:scdmm} can be adapted to solving general equality-constrained non-linear optimization problems. A nonlinear theory and further applications of these techniques will be the subject of our next work. 

\bibliographystyle{siam}  % style file is siam.bst
\bibliography{mmslCCD}

\begin{thebibliography}{10}

\bibitem{Monotone}
{\sc H.~H. Bauschke and P.~L. Combettes}, {\em Convex Analysis and Monotone
  Operator Theory in Hilbert Spaces}, Springer, 2011.

\bibitem{Beck1}
{\sc A.~Beck and M.~Teboulle}, {\em Fast gradient-based algorithms for
  constrained total variation image denoising and deblurring problems}, Image
  Processing, IEEE Transactions on, 18 (2009), pp.~2419--2434.

\bibitem{Beck2}
\leavevmode\vrule height 2pt depth -1.6pt width 23pt, {\em A fast iterative
  shrinkage-thresholding algorithm for linear inverse problems}, SIAM Journal
  on Imaging Sciences, 2 (2009), pp.~183--202.

\bibitem{ISTA3}
{\sc J.~M. Bioucas-Dias and M.~A.T. Figueiredo}, {\em A new {T}w{I}{S}{T}:
  Two-step iterative shrinkage/thresholding algorithms for image restoration},
  Trans. Img. Proc., 16 (2007), pp.~2992--3004.

\bibitem{BJORCK}
{\sc A.~Bj{\"o}rk}, {\em Numerical Methods for Least Squares Problems}, SIAM,
  1996.

\bibitem{ADMM}
{\sc S.~Boyd, N.~Parikh, E.~Chu, B.~Peleato, and J.~Eckstein}, {\em Distributed
  optimization and statistical learning via the alternating direction method of
  multipliers}, Foundations and Trends® in Machine Learning, 3 (2010),
  pp.~1--122.

\bibitem{Boyd}
{\sc S.~P. Boyd and L.~Vandenberghe}, {\em Convex Optimization}, Cambridge
  University Press, 2004.

\bibitem{Bruck}
{\sc R.~E. Bruck~Jr.}, {\em On the weak convergence of an ergodic iteration for
  the solution of variational inequalities for monotone operators in {H}ilbert
  space}, Journal of Mathematical Analysis and Applications, 61 (1977), pp.~159
  -- 164.

\bibitem{Chambolle2004}
{\sc A.~Chambolle}, {\em An algorithm for total variation minimization and
  applications}, Journal of Mathematical Imaging and Vision, 20, pp.~89--97.

\bibitem{Chambolle1997}
{\sc A.~Chambolle and P.~L. Lions}, {\em Image recovery via total variational
  minimization and related problems}, Numerische Mathematik, 76 (1997),
  pp.~167--188.

\bibitem{ISTA1}
{\sc P.~L. Combettes and V.~R. Wajs}, {\em Signal recovery by proximal
  forward-backward splitting}, Multiscale Modeling \& Simulation, 4 (2005),
  pp.~1168--1200.

\bibitem{ISTA2}
{\sc I.~Daubechies, M.~Defrise, and C.~De~Mol}, {\em An iterative thresholding
  algorithm for linear inverse problems with a sparsity constraint},
  Communications on Pure and Applied Mathematics, 57 (2004), pp.~1413--1457.

\bibitem{DR}
{\sc J.~Douglas and H.~H. Rachford}, {\em On the numerical solution of heat
  conduction problems in two and three space variables}, Transactions of the
  American mathematical Society, 82 (1956), pp.~421--439.

\bibitem{Eckstein1992}
{\sc J.~Eckstein and D.~P. Bertsekas}, {\em On the douglas-rachford splitting
  method and the proximal point algorithm for maximal monotone operators},
  Math. Program., 55 (1992), pp.~293--318.

\bibitem{efron}
{\sc B.~Efron, T.~Hastie, I.~Johnstone, R.~Tibshirani, et~al.}, {\em Least
  angle regression}, The Annals of statistics, 32 (2004), pp.~407--499.

\bibitem{Fichtner2011}
{\sc A.~Fichtner}, {\em Full Seismic Modeling and Inversion}, Springer, 2011.

\bibitem{Figuer}
{\sc M.~A.~T. Figueiredo, R.~D. Nowak, and S.~J. Wright}, {\em Gradient
  projection for sparse reconstruction: Application to compressed sensing and
  other inverse problems}, Selected Topics in Signal Processing, IEEE Journal
  of, 1 (2007), pp.~586--597.

\bibitem{Gabay1976}
{\sc D.~Gabay and B.~Mercier}, {\em A dual algorithm for the solution of
  nonlinear variational problems via finite element approximation}, Computers
  \& Mathematics with Applications, 2 (1976), pp.~17 -- 40.

\bibitem{GlowinskiTallec1989}
{\sc R.~Glowinski and P.~Le~Tallec}, {\em Augmented Lagrangian and
  Operator-Splitting Methods in Nonlinear Mechanics}, Society for Industrial
  and Applied Mathematics, 1989.

\bibitem{Glowinski1975}
{\sc R.~Glowinski and A.~Marroco}, {\em Sur l'approximation, par éléments
  finis d'ordre un, et la résolution, par pénalisation-dualité d'une classe
  de problèmes de dirichlet non linéaires}, ESAIM: Mathematical Modelling and
  Numerical Analysis - Modélisation Mathématique et Analyse Numérique, 9
  (1975), pp.~41--76.

\bibitem{Gold2014}
{\sc T.~Goldstein, B.~O'Donoghue, S.~Setzer, and R.~Baraniuk}, {\em Fast
  alternating direction optimization methods}, SIAM Journal on Imaging
  Sciences, 7 (2014), pp.~1588--1623.

\bibitem{GoldOsher09}
{\sc T.~Goldstein and S.~Osher}, {\em The split {B}regman method for
  {L}1-regularized problems}, SIAM Journal on Imaging Sciences, 2 (2009),
  pp.~323--343.

\bibitem{Hastie}
{\sc T.~Hastie, S.~Rosset, R.~Tibshirani, and J.~Zhu}, {\em The entire
  regularization path for the support vector machine}, J. Mach. Learn. Res., 5
  (2004), pp.~1391--1415.

\bibitem{He2012}
{\sc B.~He and X.~Yuan}, {\em On the ${O}(1/n)$ convergence rate of the
  {D}ouglas-{R}achford alternating direction method}, SIAM Journal on Numerical
  Analysis, 50 (2012), pp.~700--709.

\bibitem{Hestenes69}
{\sc M.~R. Hestenes}, {\em Multiplier and gradient methods}, Journal of
  Optimization Theory and Applications, 4 (1969), pp.~303--320.

\bibitem{Kim2007}
{\sc S.~Kim, K.~Koh, M.~Lustig, S.~Boyd, and D.~Gorinevsky}, {\em An
  interior-point method for large-scale $\ell_1$-regularized least squares},
  Selected Topics in Signal Processing, IEEE Journal of, 1 (2007),
  pp.~606--617.

\bibitem{Kosloff1980}
{\sc D.~Kosloff, R.F. Scott, and J.~Scranton}, {\em Finite element simulation
  of {W}ilmington oil field subsidence: I. {L}inear modelling}, Tectonophysics,
  65 (1980), pp.~339 -- 368.

\bibitem{musamark14}
{\sc M.~Maharramov and M.~Zoback}, {\em Monitoring of cyclic steam stimulation
  by inversion of surface tilt measurements}, AGU Fall Meeting, Session
  H23A-0859,  (2014).

\bibitem{Nesterov}
{\sc Y.~E. Nesterov}, {\em A method for solving the convex programming problem
  with rate of convergence ${O}(1/k^2)$}, Dokl. Akad. Nauk SSSR, 269 (1983),
  pp.~543--547.

\bibitem{Nocedal}
{\sc J.~Nocedal and S.~J. Wright}, {\em Numerical Optimization}, Springer,
  2006.

\bibitem{Osborne}
{\sc M.~R. Osborne, B~Presnell, and B.~A. Turlach}, {\em A new approach to
  variable selection in least squares problems}, IMA Journal of Numerical
  Analysis, 20 (2000), pp.~389--403.

\bibitem{Passty}
{\sc G.~B. Passty}, {\em Ergodic convergence to a zero of the sum of monotone
  operators in {H}ilbert space}, Journal of Mathematical Analysis and
  Applications, 72 (1979), pp.~383 -- 390.

\bibitem{Qiu2013}
{\sc Y.~Qiu, W.~Xue, and G.~Yu}, {\em Intelligent Science and Intelligent Data
  Engineering: Third Sino-foreign-interchange Workshop, IScIDE 2012, Nanjing,
  China, October 15-17, 2012. Revised Selected Papers}, Springer Berlin
  Heidelberg, Berlin, Heidelberg, 2013, ch.~A Projected Conjugate Gradient
  Method for Compressive Sensing, pp.~398--406.

\bibitem{Rockafellar1}
{\sc R.~T. Rockafellar}, {\em Convex Analysis}, Princeton University Press,
  1971.

\bibitem{Rockafellar56}
\leavevmode\vrule height 2pt depth -1.6pt width 23pt, {\em Augmented
  lagrangians and applications of the proximal point algorithm in convex
  programming}, Mathematics of Operations Research, 1 (1976), pp.~97--116.

\bibitem{ROF}
{\sc L.~I. Rudin, S.~Osher, and E.~Fatemi}, {\em Nonlinear total variation
  based noise removal algorithms}, Physica D: Nonlinear Phenomena, 60 (1992),
  pp.~259--268.

\bibitem{SAAD}
{\sc Y.~Saad}, {\em Iterative Methods for Sparse Linear Systems, second
  edition}, SIAM, 2003.

\bibitem{SEGDEF}
{\sc P.~Segall}, {\em Earth and Volcano Deformation}, Princeton University
  Press, 2010.

\bibitem{Tarantola}
{\sc A.~Tarantola}, {\em Inversion of seismic reflection data in the acoustic
  approximation}, Geophysics, 49 (1984), pp.~1259--1266.

\bibitem{NLA}
{\sc L.~N. Trefethen and David~Bau III}, {\em Numerical Linear Algebra}, SIAM,
  1997.

\bibitem{Uzawa}
{\sc H.~Uzawa}, {\em Studies in Linear and Non-Linear Programming}, Stanford
  University Press, 1958, ch.~Iterative methods for concave programming.

\bibitem{VogelOman}
{\sc C.~R. Vogel and M.~E. Oman}, {\em Iterative methods for total variation
  denoising}, SIAM J. Sci. Comput., 17 (1996), pp.~227--238.

\bibitem{Zhang2010}
{\sc X.~Zhang, M.~Burger, and S.~Osher}, {\em A unified primal-dual algorithm
  framework based on {B}regman iteration}, Journal of Scientific Computing, 46
  (2010), pp.~20--46.

\end{thebibliography}

\end{document}